\documentclass[11p,a4paper]{article}
\usepackage{amsfonts}
\usepackage{amssymb}
\usepackage{amsthm}
\usepackage{amsmath}
\usepackage{graphicx}
\usepackage{mathrsfs}
\usepackage{graphics}
\usepackage{abstract}
\usepackage{color}
\usepackage{cite}
\usepackage{physics}
\usepackage[shortlabels]{enumitem}
\usepackage{empheq}
\usepackage{appendix}
\usepackage{mathrsfs}
\usepackage{bm}
\usepackage[colorlinks=true,citecolor=red]{hyperref}
\hypersetup{linkcolor=blue}
\graphicspath{ {./result/} }
\usepackage{titling}
\usepackage{authblk}
\usepackage{geometry}
\usepackage{fancyhdr}
\usepackage{lineno}
\usepackage{titlesec}
\titleformat{\section}[block]{\centering\Large\bfseries}{\thesection}{1em}{}
\titleformat{\subsection}[block]{\centering\large\bfseries}{\thesubsection}{1em}{}
\titleformat{\subsubsection}[block]{\centering\normalsize\bfseries}{\thesubsubsection}{1em}{}

\usepackage{ulem}

\newtheorem{theorem}{\textbf{Theorem}}[section]
\newtheorem{lemma}{\textbf{Lemma}}[section]
\newtheorem{proposition}{\textbf{Proposition}}[section]
\newtheorem{corollary}{\textbf{Corollary}}[section]
\newtheorem{remark}{\textbf{Remark}}[section]
\newtheorem{example}{\textbf{Example}}[section]

\newtheorem{definition}{\textbf{Definition}}[section]

\providecommand{\keywords}[1]{\textbf{\textit{Keywords---}} #1}
\allowdisplaybreaks[4]

\numberwithin{equation}{section}
\def\be{\begin{equation}}
\def\ee{\end{equation}}
\def\bea{\begin{eqnarray}}
\def\eea{\end{eqnarray}}
\def\bt{\begin{theorem}}
\def\et{\end{theorem}}
\def\bl{\begin{lemma}}
\def\el{\end{lemma}}
\def\br{\begin{remark}}
\def\er{\end{remark}}
\def\bp{\begin{proposition}}
\def\ep{\end{proposition}}
\def\bc{\begin{corollary}}
\def\ec{\end{corollary}}
\def\bd{\begin{definition}}
\def\ed{\end{definition}}

\def\Pi{\mathbf{\psi}}

\usepackage{tikz}
\usetikzlibrary{calc, arrows, intersections}
\title{\bf{Stability and bifurcation of 2D viscous primitive
equations with full diffusion}}

\author{
Song Jiang$^{a}$
\thanks{jiang@iapcm.ac.cn}
\quad\quad
Quan Wang$^{b}$
\thanks{Corresponding author:xihujunzi@scu.edu.cn }
\\ \footnotesize $^a$
 LCP, Institute of Applied Physics and Computational Mathematics,
\\\footnotesize Huayuan Road 6,
Beijing,100088, China
  \\ \footnotesize $^{b}$ College of Mathematics, Sichuan University,
  \footnotesize
 Chengdu, Sichuan, 610065,  China
}
\medskip

\begin{document}
\maketitle
\begin{abstract}
This paper investigates the stability and bifurcation of the two-dimensional viscous primitive equations with full diffusion under thermal forcing. The system governs perturbations about a motionless basic state with a linear temperature profile in a periodic channel, where the temperature is fixed at $T_0$ and $T_1$ on the bottom and upper boundaries, respectively. Through a rigorous analysis of three distinct thermal regimes, we identify a critical temperature difference $T_c$ that fundamentally dictates the system's dynamical transitions. Our main contributions are fourfold. Firstly, in the subcritical case $T_0 - T_1 < T_c$, we use energy methods to establish the global nonlinear stability in $H^2$-norm, proving that perturbations decay exponentially. Secondly, precisely at the critical threshold $T_0 - T_1 = T_c$, we prove not only the nonlinear stability in $H^1$-norm but also the asymptotic convergence of all solutions to zero, leveraging spectral and dynamical systems theory. Finally, in the supercritical regime $T_0 - T_1 > T_c$, a bootstrap argument reveals that the basic state is nonlinearly unstable across all $L^p$-norms for $1 \leq p \leq \infty$.
Finally, near the critical point, the dynamics are first reduced to a two-dimensional system on a center manifold. This reduced system then undergoes a supercritical bifurcation, generating a countable family of stable steady states that are organized into a local ring attractor.
This work closes a significant gap in the stability analysis of the thermally driven primitive equations, establishing a rigorous mathematical foundation for understanding the formation of convection cells in large-scale geophysical flows.
\end{abstract}

\newpage
\tableofcontents

\newpage
\section{Introduction}
\setcounter{equation}{0}

\subsection{Problem and Literature Review}
The investigation of hydrodynamic stability in thermally driven flows, such as the Rayleigh–Bénard convection \cite{Busse1965,busse1979instabilities,Getling1998}, holds significant relevance in geophysical contexts and finds numerous engineering applications \cite{busse1994convection,Nield2007}. This field examines the transition of a stable system into a turbulent state under the influence of thermal gradients. It serves as a paradigm for exploring nonlinear dynamics, pattern formation, and chaotic behavior, thereby providing insights applicable to a broad range of non-equilibrium systems \cite{Lappa2010}. From a practical standpoint, such studies contribute to the optimization of heat transfer in diverse technologies—ranging from nuclear reactors to climate modeling—where large-scale atmospheric and oceanic circulations are driven by thermal forcing \cite{Mizerski2021,Radko2017}. Ultimately, this field deepens our understanding of universal physical principles and facilitates technological advancement through the prediction and control of thermally induced flows.

In this article, we consider the stability and nonlinear dynamics near the basic state defined by
\begin{align}\label{steady-state-0}
\begin{aligned}
&(v,w)=(v_s,w_x)=(0,0),\quad T=T_s=T_0+\frac{z}{H}(T_1-T_0),\\&p=p_s=
-g\rho_0\int_0^z(1-\beta(T_s(\xi)-T_0))\,d\xi.
\end{aligned}
\end{align}
This motionless basic state, which is characterized by a linear temperature profile and hydrostatic balance, satisfies the two-dimensional viscous primitive equations:
\begin{align}\label{primitive}
\begin{cases}
\partial_tv+v\partial_x v+w\partial_z v=
\mu_x \partial_{xx}v+
\mu_z\partial_{zz}v
-\partial_xp, \\ 
\partial_zp=-g\rho_0(1-\beta(T-T_0)),\\ 
\partial_tT+v\partial_xT+w\partial_z T=
\kappa_x \partial_{xx}T+
\kappa_z\partial_{zz}T,\\
\partial_xv+\partial_z w=0,\\
\rho=\rho_0(1-\beta(T-T_0)),
\end{cases}
\end{align}
where the dependent variables are the velocity field $(v, w)$, temperature $T$, and pressure $p$, with $\rho_0$ being a reference density. The physical parameters include the thermal expansion coefficient $\beta$, gravitational acceleration $g$, and the prescribed boundary temperatures $T_0$ (bottom) and $T_1$ (upper). The turbulent viscosities and diffusion coefficients in the $x$- and $z$-directions are denoted by $\mu_x$, $\mu_z$ and $\kappa_x$, $\kappa_z$, respectively.
The basic state \eqref{steady-state-0} represents a quiescent, linearly stratified configuration that can be realized through heating from below ($T_0 > T_1$) or above ($T_0 < T_1$). 
The dynamics of perturbations about this equilibrium state, examined in both bounded and unbounded domains, constitute a central theme in mathematical convection theory \cite{Getling1998, Lappa2010}. This foundational problem continues to inspire substantial research, bridging rigorous mathematical analysis and applied physics through the elucidation of universal mechanisms governing instability and spontaneous pattern formation in nonequilibrium fluid systems.
The viscous primitive equations \eqref{primitive}, which are derived from the Boussinesq equations via the hydrostatic approximation \cite{Furukawa2020}, establish a standard theoretical framework for investigating geostrophic phenomena in continuously stratified fluids \cite{Pedlosky1987, Lions1992, Lions1992-2}. In this study, the governing system \eqref{primitive} is considered on the domain
\begin{align}
\Omega = \left\{ (x, z) : 0 \le z \le H,\ x \in [0, L] \right\},\quad L\gg H,
\end{align}
subject to the following initial and boundary conditions:
\begin{align}\label{i-b-condition}
\begin{cases}
(v, w, T)_{t=0} = (v^i, w^i, T^i), \\
T|_{z=0} = T_0, \quad T|_{z=H} = T_1, \\
(v_z, w)|_{z=0} = 0, \quad (v_z, w)|_{z=H} = 0, \\
v, w, T \text{ are periodic in $x$ with period $L$}.
\end{cases}
\end{align}

 The medium- and small-scale dynamics that perturb the basic state \eqref{steady-state-0} are primarily governed by the Boussinesq equations \cite{Nield2007, Pedlosky1987, Majda2003}. This model has served as a cornerstone in theoretical studies of convective flows, enabling detailed analysis of linear and nonlinear stability, well-posedness, transition mechanisms, and the nonlinear evolution of perturbations near the basic state. When $T_0 > T_1$, the state \eqref{steady-state-0} corresponds to a fluid layer heated from below, leading to the classical Rayleigh–Bénard convection instability once the temperature difference $T_0 - T_1$ exceeds a critical threshold. In their foundational works, Chandrasekhar \cite{Chandrasekha1961} and Drazin and Reid \cite{Drazin1981} linearized the governing equations about the basic state and introduced normal mode perturbations of the form $u' \propto \exp(\sigma t + i(k_x x + k_y y))$. Through linear stability analysis, they derived an expression for the critical Rayleigh number $Ra_c$—a dimensionless parameter quantifying the balance between buoyant driving and dissipative effects—beyond which the perturbation growth rate $\Re(\sigma)$ becomes positive. Guo and Han \cite{guo2010critical} later established a variational characterization of $Ra_c$ and used it to analyze nonlinear stability in the $H^2$-norm, showing that the basic state is nonlinearly stable in this norm if and only if $Ra < Ra_c$. Nguyen \cite{Nguyen2025} extended these results to account for temperature-dependent viscosity, while Mielke \cite{Mielke1997} provided a rigorous mathematical treatment of sideband instabilities. Substantial attention has also been devoted to understanding bifurcations near the basic state at $Ra = Ra_c$. Ma and Wang \cite{Ma2003, Ma2004} investigated attractor bifurcation phenomena in this regime, the results were later generalized by Han et al. \cite{Han2019} to include the internal heating and variable gravity effects.
Braaksma and Iooss \cite{Braaksma2019} established the existence of bifurcating quasipatterns near the threshold, and Haragus and Iooss \cite{Haragus2021} analyzed the emergence of symmetric domain walls. More recently, Watanabe and Yoshimura \cite{Watanabe2023} examined resonance, symmetry, and periodic orbit bifurcations in the vicinity of the basic state. 
For a broader exposition of bifurcation phenomena in this problem from a dynamical systems perspective, we direct the reader to the recent studies by Welter \cite{Welter2025} and Roberts \cite{Roberts2025} and the pertinent literature reviewed therein.

When $T_0 < T_1$, the basic state \eqref{steady-state-0} represents an equilibrium configuration with hotter fluid overlying colder fluid—a scenario opposite to the classical Rayleigh–Bénard instability. From a physical standpoint, this configuration is expected to be nonlinearly stable. In recent years, the nonlinear stability analysis of this basic state within the Boussinesq framework has attracted considerable attention from applied mathematicians. 
Doering et al.\cite{Doering2018} established stability and long-time behavior of perturbations near the basic state for the 2D Boussinesq equations with only velocity dissipation. Castro et al. \cite{Castro2019} proved the asymptotic stability of the basic state for the 2D Boussinesq system augmented by a velocity damping term. Tao \cite{Tao2020} investigated the stability of perturbations about the basic state for the 2D Boussinesq equations without thermal diffusion in a periodic domain. Further extending these results, Lai et al. \cite{Lai2021} derived optimal decay estimates for perturbations of the basic state in the 2D Boussinesq equations with partial dissipation. Adhikari et al. \cite{Adhikari2022} analyzed the stability of the basic state for the 2D Boussinesq system with horizontal dissipation and vertical thermal diffusion. 
For a more comprehensive exposition of recent developments concerning the nonlinear stability of the basic state, including cases with and without shear flows, we refer the reader to \cite{Masmoudi2022,Jang2023, Zhang2023}, along with the pertinent literature cited therein.

When the horizontal scale of thermally driven fluid flows significantly exceeds the vertical scale, the two-dimensional Boussinesq equations without rotation can be systematically reduced to the primitive equations on a periodic channel \cite{Lions1992, Lions1992-2}. Serving as the fundamental governing equations for large-scale atmospheric and oceanic circulations, the primitive equations provide a principal framework for studying geophysical fluid phenomena \cite{Pedlosky1987}. Owing to their broad physical applicability and intrinsic mathematical interest, the analytical properties of these equations have attracted considerable attention in recent years. A substantial body of research has focused on the well-posedness and long-time dynamics of the primitive equations. In bounded domains, Petcu \cite{Petcu2007} established the existence and uniqueness of $z$-weak solutions for the two-dimensional viscous primitive equations, along with strong solutions in the three-dimensional case. Cao and Titi \cite{Cao2007} proved the global well-posedness of strong solutions to the three-dimensional viscous primitive equations in cylindrical domains. Rousseau et al. \cite{Rousseau2008} investigated the well-posedness of the three-dimensional inviscid primitive equations linearized about a stratified flow. Guo and Huang \cite{Guo2009} demonstrated the existence and uniqueness of global strong solutions to the initial-boundary value problem of stochastic three-dimensional primitive equations, further establishing the existence of random attractors by analyzing the asymptotic behavior of solutions. Temam and Wirosoetisno \cite{Temam2010} examined the stability of the slow manifold in the three-dimensional forced-dissipative rotating primitive equations. Hsia and Shiue \cite{Hsia2013} derived an asymptotic stability criterion for solutions in a three-dimensional finite cylinder with time-dependent forcing. Further developments include the work of Cao et al. \cite{Cao2014}, who established the global well-posedness of strong solutions with vertical eddy diffusivity, building on this foundation, they later proved an analogous result with only horizontal viscosity and diffusion \cite{Cao2016}. For a class of discontinuous initial data, Li and Titi \cite{Li2017} proved the existence and uniqueness of weak solutions to the three-dimensional viscous primitive equations. Regarding non-uniqueness and singularity formation, Chiodaroli and Mich\'{a}lek \cite{Chiodaroli2017} applied the convex integration methods of De Lellis and Székelyhidi \cite{De2009,De2010} to demonstrate the existence of infinitely many global weak solutions to the inviscid primitive and Boussinesq equations. More recently, Collot et al. \cite{Collot2024} provided a comprehensive description of two distinct blow-up mechanisms for the three-dimensional inviscid primitive equations. For further advances in the mathematical theory of primitive equations, we refer readers to \cite{Binz2024, Korn2024, Hu2025, Hieber2025} and references therein.

The primitive equations are formally derived as the small-aspect-ratio limit (i.e., the ratio of depth to horizontal length scale) of the Boussinesq system, which itself models large-scale oceanic and atmospheric dynamics \cite{Pascal2001, Li2019}. Extensive mathematical analysis of the Boussenseq system has yielded profound insights into the local and global dynamics near the basic state \eqref{steady-state-0}. These studies elucidate the onset of the Rayleigh–Bénard convection patterns when the fluid is heated from below and the temperature difference $T_0 - T_1$ surpasses a critical threshold, as well as the global stability of the configuration corresponding to heating from above. For large-scale oceanic and atmospheric phenomena, a central question remains: it is analyzed within the more realistic framework of the primitive equations, does the basic state \eqref{steady-state-0} with hot fluid overlying cold fluid retain its global stability? Conversely, does the reverse configuration with cold fluid over hot fluid give rise to large-scale convective structures—such as Hadley cells, Walker circulation, and thermohaline circulation \cite{Vallis2017}? Addressing these questions using the primitive equations constitutes a fundamental step toward understanding the organization of global climate systems.

Compared to the mathematical analysis via the Boussinesq equations, the study of local and global dynamics near the basic state \eqref{steady-state-0} within the primitive equations framework presents several substantial challenges.
The first major difficulty arises in the energy estimates.
The standard approach of estimating $\nabla v$ or $\nabla T$ as single entities, effective in the Boussinesq setting, fails for the primitive equations. This is due to the strong mutual coupling among the bounds for $\|\partial_{x}v\|_{L^{2}}$, $\|\partial_{z}v\|_{L^{2}}$, $\|\partial_{x}T\|_{L^{2}}$, and $\|\partial_{z}T\|_{L^{2}}$. Consequently, one cannot directly demonstrate that $\|\nabla v\|_{L^{2}}+\|\nabla T\|_{L^{2}}$ satisfies a closed inequality, apply Gronwall's lemma, and thereby deduce either the exponential decay of this quantity when $T_0 \leq T_1$ or its uniform boundedness when $T_0 > T_1$. The second difficulty concerns the handling of the nonlinear terms. Ladyzhenskaya's inequality, a standard tool for estimating terms like $u \cdot \nabla u$ and $u \cdot \nabla T$ in the Boussinesq system, is not directly applicable to the primitive equations. The nonlinear terms in this paper inherently contain non-local nonlinear terms $\left(\int_0^z \partial_x v(t,x,\xi)  d\xi \right) \partial_z v$ and $\left(\int_0^z \partial_x v(t,x,\xi)  d\xi \right) \partial_z T$, which necessitate the development of novel analytical techniques. New methods must be devised to control these terms to establish the desired $H^2$-norm estimates: either exponential decay when $T_0 \leq T_1$ or uniform boundedness of $\|v\|_{H^{2}} + \|\nabla T\|_{H^{2}}$ when $T_0 > T_1$. Furthermore, a third fundamental challenge appears in the critical regime. When $T_0 > T_1$ and the difference $T_0 - T_1$ exceeds a critical threshold $T_c$, the basic state \eqref{steady-state-0} is expected to lose stability precisely at this threshold. In this critical case, nonlinear estimates alone are generally insufficient to establish the global stability of \eqref{steady-state-0}, requiring more refined dynamical systems. An overarching complication arises because the three thermal regimes—$T_1 > T_0$, $T_1 = T_0$, and $0 < T_0 - T_1 < T_c$—induce fundamentally different mathematical structures in its perturbation system \eqref{ondimensional-equations-2}. This structural disparity requires the development of separate and tailored estimation methods to analyze the dynamics in each case.

\subsection{Reformulation and Open Problems}

We begin by reformulating the system \eqref{primitive}–\eqref{i-b-condition}. Using the boundary condition $w(0)=0$ and integrating the equation $\eqref{primitive}_4$ with respect to $z$ yields the vertical velocity component:
\begin{align}\label{w}
w=-\int_0^z\partial_xv(t,x,\xi)\,d\xi.
\end{align}
Making use of the boundary condition $\eqref{i-b-condition}_3$, we obtain:
\begin{align*}
\int_0^H\partial_xv(t,x,z)\,dz=0,
\end{align*}
which implies that the depth-integrated horizontal velocity $\int_0^Hv(t,x,z) , dz$ is independent of $x$. Thus, without loss of generality, we impose:
\[
\int_0^Hv(t,x,\xi)\,d\xi=0,
\]
a condition that is consistent with the initial condition 
\[
\int_0^1 v^i(x,z) , dz = 0.
\]

We now obtain an expression for the pressure by integrating the second equation in \eqref{primitive} vertically. Performing the integration in $z$ from $0$ to $z$ yields:
\begin{align}\label{p}
p(t,x,z)=-g\rho_0\int_0^z(1-\beta(T-T_0))\,d\xi+q(t,x).
\end{align}
where $q(t,x)$ is an integration constant independent of $z$, representing the background pressure distribution in the horizontal direction. This expression decomposes the pressure $p(t,x,z)$ into a buoyancy-related term, and a purely hydrostatic (or horizontally varying) component.

By substituting the integral relations for the vertical velocity \eqref{w} and the pressure \eqref{p} into the governing equations \eqref{primitive} and subsequently applying the boundary conditions \eqref{i-b-condition}, the system can be reduced to a simplified form. This process yields a set of equations that govern the evolution of the horizontal velocity $v$ and the temperature $T$, which are given by
 \begin{align}
\begin{cases}\label{primitive-reformulation}
\begin{aligned}
&\partial_tv+v \partial_x v-\left(\int_0^z\partial_xv(t,x,\xi)\,d\xi \right)\partial_z v=\nu_x\partial_{xx}v+
\nu_z\partial_{zz}v\\&-g\rho_0\beta\partial_x \left(
\int_0^zT(t,x,\xi)\,d\xi\right)-\partial_xq(t,x),
\end{aligned}\\
\begin{aligned}
&\partial_tT+v\partial_x T-\left(\int_0^z\partial_xv(t,x,\xi)\,d\xi \right)\partial_z T= \kappa_{x}\partial_{xx}T+
\kappa_{z}\partial_{zz} T,
\end{aligned}\\
v_z|_{z=0}= v_z|_{z=H}=0, \quad
T|_{z=0}=T_0,\quad T|_{z=H}=T_1,\\
v, T~\text{ are periodic in $x$ with period $L$},\\
\int_0^Hv(t,x,\xi)\,d\xi=0.
\end{cases}
\end{align}

In the framework of \eqref{primitive-reformulation}, the basic state for the steady solution \eqref{steady-state-0} is given by
\begin{align}\label{steady-state}
\begin{aligned}
&v=v_s=0,\quad T=T_s=T_0+\frac{z}{H}(T_1-T_0),\quad q=q_s=0.
\end{aligned}
\end{align}
which constitutes a purely conductive, hydrostatically balanced solution. To analyze its linear and nonlinear stability, we consider its perturbations of the form $(v,T,q)=(v_s, T_s, q_s) + (v', T', q')$. Substituting this ansatz into the governing equations \eqref{primitive-reformulation} yields the nonlinear 
system for the perturbations (where the primes are omitted for clarity):
 \begin{align}
\begin{cases}\label{primitive-reformulation-p}
\begin{aligned}
&\partial_tv+v \partial_x v-\left(\int_0^z\partial_xv(t,x,\xi)\,d\xi \right)\partial_z v=\nu_x\partial_{xx}v+
\nu_z\partial_{zz}v\\&-g\rho_0\beta\partial_x \left(
\int_0^zT(t,x,\xi)\,d\xi\right)-\partial_xq(t,x),
\end{aligned}\\
\begin{aligned}
&\partial_tT+v\partial_x T-\left(\int_0^z\partial_xv(t,x,\xi)\,d\xi \right)\partial_z T= \kappa_{x}\partial_{xx}T+
\kappa_{z}\partial_{zz} T\\
&+\frac{T_1-T_0}{H}\left(\int_0^z\partial_xv(t,x,\xi)\,d\xi \right),
\end{aligned}\\
v_z|_{z=0}= v_z|_{z=H}=0, \quad
T|_{z=0}=0,\quad T|_{z=H}=0,\\
v, T~\text{ are periodic in $x$ with period $L$},\\
\int_0^Hv(t,x,z)\,dz=0,
\end{cases}
\end{align}
The initial conditions supplement the system for perturbations
\begin{align}\label{i-b-condition-2}
(v, T)\big|_{t=0} = (v^i, T^i),\quad
\int_0^Hv^i(x,z)\,dz=0.
\end{align}

A precise characterization of the steady states is indispensable for analyzing the long-time behavior of any fluid-dynamical system.  As the foundational elements of the linear stability theory, these equilibrium configurations represent the possible end-states towards which the system may evolve. Our detailed analysis of the nonlinear system \eqref{primitive-reformulation-p} has led to the identification of the following family of steady states, which we now summarize.
\begin{lemma}\label{steady-one}
Let $(v,\Theta)$ be a smooth steady-state solution of the system \eqref{primitive-reformulation-p},  we have
\begin{itemize}
\item[ \rm{(1)}] If $T_1\geq T_0$,  then $
v=T=0$;
\item[\rm{(2)}] If $T_1<T_0$, we obtain the following identity
\begin{align}\label{steady-231-1}
\begin{aligned}
&g\rho_0\beta H\left(\kappa_{x}
\int_{D}\abs{ \partial_{x}T}^2\,dx\,dz
+\kappa_{z}
\int_{D}\abs{ \partial_{z}T}^2\,dx\,dz\right)\\&=
\left(T_0-T_1\right)
\left(
\nu_x\int_{D}\abs{\partial_xv}^2
\,dx\,dz+\nu_z\int_{D}\abs{\partial_zv}^2
\,dx\,dz
\right).
\end{aligned}
\end{align}
\end{itemize}
\end{lemma}
The proof of Lemma~\ref{steady-one} is presented in 
Section~\ref{section2.2}.

Since the system \eqref{primitive-reformulation-p} governs the perturbations about the steady state \eqref{steady-state} of the original system \eqref{primitive-reformulation}, Lemma \ref{steady-one} leads to the following conclusion: if $T_1 \geq T_0$, the only admissible steady state of the system \eqref{primitive-reformulation} is the quiescent solution given by \eqref{steady-state}; whereas if $T_1 < T_0$, the relation \eqref{steady-231-1} indicates the potential existence of additional non-trivial steady states. This dichotomy motivates the investigation of three fundamental open questions for the system \eqref{primitive-reformulation}:
\begin{itemize}
\item \textbf{Global nonlinear stability for $T_1 \geq T_0$}. 
Establish the unconditional global nonlinear stability of the motionless state \eqref{steady-state} in system \eqref{primitive-reformulation}. Specifically, does the basic state attract all solutions emanating from finite-amplitude initial data?
\item \textbf{Bifurcation phenomena for $T_1 < T_0$}. Do non-trivial steady states bifurcate from the basic solution \eqref{steady-state} ? What are their multiplicity and stability? What is the critical temperature difference $T_0 - T_1$ that triggers such a bifurcation?
\item \textbf{Influence of the temperature difference}. How does the temperature difference $T_0-T_1$ dictate the stability landscape and the long-time dynamics?
\end{itemize}
Despite the physical relevance of these questions, a comprehensive survey of the literature reveals a striking theoretical gap: no rigorous mathematical treatment has been established to date.

\subsection{Statement of the Main Theorems}

The precise formulation of our main results involves the critical temperature difference $T_c$, which serves as the linear stability 
threshold for the basic state \eqref{steady-state}. It is given by
\begin{align}\label{difference}
T_c = \frac{\pi^4 \mu_z \kappa_x}{H^3 \rho_0 g \beta} \left(
\frac{\kappa_z \mu_x}{\kappa_x \mu_z} + 1 +
\min_{m \in \mathbb{Z}^+} \left( \frac{4 m^2 \mu_x H^2}{\mu_z L^2} +
\frac{\kappa_z L^2}{4 m^2 \kappa_x H^2} \right) \right).
\end{align}
This expression originates from a standard linear stability analysis \cite{Chandrasekha1961,Drazin1981, Ma2004}: $T_c$ is defined as the critical value where the leading eigenvalue of the linearized operator crosses into the right half-plane, marking the onset of exponential growth and the loss of linear stability.

When $T_0 - T_1 < T_c$, the trivial solution $(v, T) = (0, 0)$ of the system \eqref{primitive-reformulation-p} is linearly stable. Moreover, from the derivation of $T_c$, it follows that any solution $(v, T)$ satisfying \eqref{steady-231-1} must be identically zero. This implies that $(0, 0)$ is the unique steady state of the system \eqref{primitive-reformulation-p} in the subcritical regime $T_0 - T_1 < T_c$.
It is natural to expect that this unique steady state also exhibits nonlinear stability. Our first main result establishes its global nonlinear stability properties.

\begin{theorem}\label{global-stability}
Assume that $T_0 - T_1 < T_c$. Then the steady-state solution \eqref{steady-state} is nonlinearly stable in the $H^2$-norm. Specifically, there exists a constant $\lambda > 0$ such that every global strong solution $
((v, T), q) \in C\big([0, +\infty); H^2(\Omega) \times H^1(\mathbb{T})\big)$
to the perturbation system \eqref{primitive-reformulation-p}, with initial data $(v^i, T^i) \in H^2(\Omega)$ satisfying $\int_0^H v^i(\cdot, x, z) \, dz = 0$,
also satisfies the decay estimate
\begin{align*}
\| v(t) \|_{H^2} + \| T(t) \|_{H^2} + \| \partial_x q(t) \|_{L^2} \leq C e^{-2\lambda t}, \quad \forall t \geq 0,
\end{align*}
for some constant $C > 0$ independent of time.
\end{theorem}
The proof of Theorem~\ref{global-stability} is presented in Section~\ref{theorem21}.

Having established the nonlinear stability for the subcritical regime $T_0 - T_1 < T_c$, we now examine the critical case where the temperature difference exactly equals the threshold value.
\begin{theorem}\label{critical-stability}
Assume that $T_0 - T_1 = T_c$. Then the steady-state solution \eqref{steady-state} is nonlinearly stable in the $H^1$-norm. Moreover, every global strong solution 
$((v, T), q) \in C\big([0, +\infty); H^2(\Omega) \times H^1(\mathbb{T})\big)$
to the perturbation system \eqref{primitive-reformulation-p}, with initial data $(v^i, T^i) \in H^2(\Omega)$ satisfying 
$\int_0^H v^i(\cdot, x, z) \, dz = 0$,
exhibits asymptotic decay in the $H^1$-norm: 
$\lim_{t \to \infty} \left( \| v(t) \|_{H^1} + \| T(t) \|_{H^1} \right) = 0$.
\end{theorem}
The proof of Theorem~\ref{critical-stability} is provided in Section~\ref{proof-theorem22}.

While the system remains stable at the critical threshold, the situation changes dramatically in the supercritical regime, where nonlinear instability emerges.
\begin{theorem}\label{noninstability}
Assume that $T_1 - T_0 > T_c$. Then, for any $p \in [1, \infty]$, the steady-state solution \eqref{steady-state} is nonlinearly unstable in the $L^p$-norms. More precisely, there exist positive constants $\epsilon$ and $C^*$ and unit-norm functions $(v^i_0, T^i_0) \in H^2(\Omega)$ satisfying $\|(v^i_0, T^i_0)\|_{H^2} = 1$ and 
$\int_0^H v_0^i(\cdot, x, z) \, dz = 0$,
 such that for every $\delta \in (0, \epsilon)$, the unique global strong solution
$((v, T), q) \in C\big([0, +\infty); H^2(\Omega) \times H^1(\mathbb{T})\big)$
to the perturbation system \eqref{primitive-reformulation-p} with initial data $(v^i, T^i) = \delta (v^i_0, T^i_0)$ satisfies $
\|(v(T_\delta), T(T_\delta))\|_{L^p} \geq \epsilon$,
where $T_\delta$ is the escape time defined by
$T_\delta = \left(\beta_{m_c,1}^{+}\right)^{-1} \ln \frac{\epsilon}{\delta}$,
in which $\beta_{m_c,1}^{+}$ given in \eqref{all-eigenvalue} is the principal eigenvalue (the largest eigenvalue) of the problem \eqref{perturbation-31}.
\end{theorem}
The proof of Theorem~\ref{noninstability} is given in Section~\ref{proof-theorem23}.

\begin{remark}
The global existence of strong solutions to the two-dimensional primitive equations with full dissipation—that is, with complete viscosities and diffusion—was first established by Bresch, Kazhikhov, and Lemoine in \cite{Bresch2004} and independently by Temam and Ziane in \cite{Temam2004}. Although the boundary conditions considered in the system \eqref{primitive-reformulation-p} differ from those in the aforementioned works, a similar approach can be applied to prove the global existence of strong solutions for \eqref{primitive-reformulation-p}.
\end{remark}

The transition from stability to instability near the critical threshold motivates a deeper investigation into the local dynamics. The following reduction theorem reveals that \eqref{primitive-reformulation-p} can be effectively described by a two-dimensional system in the vicinity of the critical threshold $T_c$.
\begin{theorem}\label{reduction}
Assume that $T_0 - T_1$ lies in a neighborhood of $T_c$. Then the dynamics of the global strong solution $
((v, T), q) \in C\big([0, +\infty); H^2(\Omega) \times H^1(\mathbb{T})\big)$
to the perturbation system \eqref{primitive-reformulation-p} near the trivial state $(v, T) = (0, 0)$, with initial data $(v^i_0, T^i_0) \in H^2(\Omega)$ satisfying 
$\int_0^1 v^i_0(\cdot, x, z)  dz = 0$, 
are equivalent to the dynamics of the following two-dimensional system:
\begin{align}\label{req-m1-m}
\begin{cases}
	\frac{dx_1}{dt} = \beta_{m_c,1}^{+} x_1 + l x_1 \|\mathbf{x}\| + o\left(\|\mathbf{x}\|^3\right), \\
	\frac{dx_2}{dt} = \beta_{m_c,1}^{+} x_2 + l x_2 \|\mathbf{x}\| + o\left(\|\mathbf{x}\|^3\right),
\end{cases}
\end{align}
where $\beta_{m_c,1}^{+}$ is the
largest eigenvalue given in \eqref{all-eigenvalue}, and $l$ is the bifurcation coefficient given by
\[
l = \frac{16 H^3 \kappa_x (16 - 3\pi^2) A_{m_c}^2 m_c^2 (1 + A_{m_c}^2)^{-2}}{12 L^3 (4 \kappa_z \pi^2 + 2 \kappa_x \beta_{m_c,1}^{+})}<0.
\]
Here, $A_{m_c}$ is defined as
\begin{align*}
A_{m_c} = \frac{2\pi m_c \text{R}}{\alpha \pi \left( \frac{4\pi^2 m_c^2}{\alpha^2} + \kappa_a \pi^2 + \beta_{m_c,1}^{+} \right)}, \quad \text{with} \quad \text{R} = \frac{H \sqrt{H \rho_0 g \beta (T_1 - T_0)}}{\kappa_x},
\end{align*}
and $m_c$ is the positive integer satisfying
\[
\frac{4m_c^2 \nu_x H^2}{\nu_z L^2} + \frac{\kappa_z L^2}{4 m_c^2 \kappa_x H^2}= \min_{m \in \mathbb{Z}^+} \left( \frac{4m^2 \nu_x H^2}{\nu_z L^2} + \frac{\kappa_z L^2}{4 m^2 \kappa_x H^2} \right).
\]
Moreover, the bifurcation of the system \eqref{req-m1-m} at $\mathrm{R} = \mathrm{R}_c$ is supercritical. In this case, it admits infinitely many stable steady-state solutions $(x_1, x_2) = (s_1, s_2)$ for $T_0-T_1 >T_c$. These solutions form a local ring attractor with $|\mathbf{s}|^2 = s_1^2 + s_2^2 = -\lambda_1 / l$, as illustrated in Fig.~\ref{attractor-ring}.
\end{theorem}

The proof of Theorem~\ref{reduction} is presented in Section~\ref{proof-theorem24}.

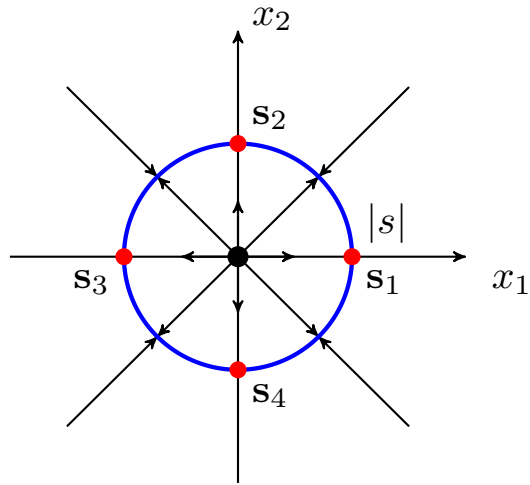
\begin{figure}[tbh]
\centering 
        \begin{tikzpicture}[>=stealth',xscale=1.5,yscale=1.5,every node/.style={scale=1.5}]
\draw [thick,->] (0,-2) -- (0,2) ;
\draw [thick,->] (-2,0) -- (2,0) ;
\draw [thick,->] (0,0) -- (0.707,0.707) ;
\draw [thick,->] (0,0) -- (-0.707,0.707) ;
\draw [thick,->] (0,0) -- (-0.707,-0.707) ;
\draw [thick,->] (0,0) -- (0.707,-0.707) ;
\draw [thick,->] (1.5,1.5) -- (0.707,0.707) ;
\draw [thick,->] (-1.5,1.5) -- (-0.707,0.707) ;
\draw [thick,->] (-1.5,-1.5) -- (-0.707,-0.707) ;
\draw [thick,->] (1.5,-1.5) -- (0.707,-0.707) ;
\draw [thick,->] (0,0) -- (0.5,0) ;
\draw [thick,->] (0,0) -- (-0.5,0) ;
\draw [thick,->] (0,0) -- (0,0.5) ;
\draw [thick,->] (0,0) -- (0,-0.5) ;
\node [below right] at (2.1,0) {$x_1$};
\node [above right] at (1,0) {$\abs{s}$};
\node [below right ] at (1,0) {$\mathbf{s}_1$};
\node [above right ] at (0,1) {$\mathbf{s}_2$};
\node [below left ] at (-1,0) {$\mathbf{s}_3$};
\node [below right] at (0,-1) {$\mathbf{s}_4$};
\node [right] at (0,2.1) {$x_2$};
\draw[domain = -2:360][draw=blue, ultra thick,samples = 200] plot({cos(\x)}, {sin(\x)});
\draw[fill,red] (1,0) circle [radius=2pt];
\draw[fill,red] (0,1) circle [radius=2pt];
\draw[fill,red] (-1,0) circle [radius=2pt];
\draw[fill,red] (0,-1) circle [radius=2pt];
\draw[fill,black] (0,0) circle [radius=2.5pt];
\end{tikzpicture}
\caption{For system \eqref{req-m1-m}, each point $\mathbf{s}$ on the blue circle corresponds to a stable equilibrium when $T_0 - T_1 > T_c$, while the origin $(0,0)$ becomes unstable. Accordingly, each such $\mathbf{s}$ gives rise to a stable equilibrium
$(v^s_b,T_b^s)=s_1(V_1,\Theta_1)^T+s_2(V_2,\Theta_2)^T
+\abs{\mathbf{s}}^2(V_3,\Theta_3)^T
+o(\abs{\mathbf{s}}^2)$
of the full system \eqref{primitive-reformulation-p}. The four labeled points $\mathbf{s}_1$--$\mathbf{s}_4$, defined by
 \\
\centerline{\text{\parbox{0.9\linewidth}{%
\[
\begin{aligned}
&\mathbf{s}_1=-\mathbf{s}_3,\quad \mathbf{s}_3=\left(-\sqrt{-\beta_{m_c,1}^{+}/l},0\right),\\&\mathbf{s}_2=-\mathbf{s}_4,\quad
\mathbf{s}_4=\left(0,-\sqrt{-\beta_{m_c,1}^{+}/l}\right),
\end{aligned}
\]}}}\\
correspond to stable steady states $(v^j_b, T_b^j)$, $j = 1, \dots, 4$, of the system \eqref{primitive-reformulation-p},
\\
\centerline{\text{\parbox{0.9\linewidth}{%
\[
(v^j_b,T_b^j)=s_1^j(V_1,\Theta_1)^T+s_2^j(V_2,\Theta_2)^T
+\abs{\mathbf{s}^j}^2(V_3,\Theta_3)^T
+o(\abs{\mathbf{s}^j}^2),
\]}}}
where the expressions of $(V_j,\Theta_j)^T$ with $j=1,2,3$ are given in Theorem \ref{bifurcation}. }
\label{attractor-ring}
\end{figure}

Building upon the reduction to the two-dimensional system, we can now characterize the precise nature of the bifurcation that occurs as the system \eqref{primitive-reformulation-p} passes through the critical threshold.

\begin{theorem}\label{bifurcation}
Under the hypotheses of Theorem~\ref{reduction}, the dynamic bifurcation of the system \eqref{primitive-reformulation-p} at the critical threshold $T_0 - T_1 = T_c$ is supercritical. Specifically, for every $T_0 - T_1 > T_c$ sufficiently close to $T_c$, the system admits a countably infinite family of stable steady states of the form
\[
(v, T) = (v^s_b, T_b^s),
\]
whose leading-order asymptotic expansions are given by
\begin{align*}
\begin{pmatrix}
v_b^s \\
T_b^s
\end{pmatrix}
= s_1
\begin{pmatrix}
V_1 \\
\Theta_1
\end{pmatrix}
+ s_2
\begin{pmatrix}
V_2 \\
\Theta_2
\end{pmatrix}
+ \abs{\mathbf{s}}^2
\begin{pmatrix}
V_3 \\
\Theta_3
\end{pmatrix}
+ o(\abs{\mathbf{s}}^2),
\end{align*}
where $\abs{\mathbf{s}}^2 = \beta_{m_c,1}^{+} / l$, and the vector profiles $(V_j, \Theta_j)^\top$ for $j = 1, 2, 3$ are given by:
\begin{align*}
\begin{aligned}
\begin{pmatrix}
V_1 \\
\Theta_1
\end{pmatrix} &=
\left( \frac{L}{4H} + \frac{L}{4H} A_{m_c}^2 \right)^{-1}
\begin{pmatrix}
\frac{\kappa_x}{H}\cos \left( \frac{2\pi m_c x}{L} \right) \cos \left( \frac{\pi z}{H} \right) \\
\abs{T}_0A_{m_c} \sin \left( \frac{2\pi m_c x}{L} \right) \sin \left( \frac{\pi z}{H} \right)
\end{pmatrix}, \\
\begin{pmatrix}
V_2 \\
\Theta_2
\end{pmatrix} &=
\left( \frac{L}{4H} + \frac{L}{4H} A_{m_c}^2 \right)^{-1}
\begin{pmatrix}
\frac{\kappa_x}{H}\sin \left( \frac{2\pi m_c x}{L} \right) \cos \left( \frac{\pi z}{H} \right) \\
-\abs{T}_0A_{m_c} \cos \left( \frac{2\pi m_c x}{L} \right) \sin \left( \frac{\pi z}{H} \right)
\end{pmatrix}, \\
\begin{pmatrix}
V_3 \\
\Theta_3
\end{pmatrix} &=
- \frac{HA_{m_c} \kappa_x \pi m_c}{2L\left( 4 \kappa_z \pi^2 + 2 \kappa_x \beta_{m_c,1}^{+} \right)}
\left( \frac{L}{4H} + \frac{L}{4H} A_{m_c}^2 \right)^{-2}
\begin{pmatrix}
0 \\
\abs{T}_0\sin \left( \frac{2\pi z}{H} \right)
\end{pmatrix}.
\end{aligned}
\end{align*}
These steady-state solutions form a local ring attractor.
\end{theorem}

The proof of Theorem~\ref{bifurcation} is presented in Section~\ref{proof-theorem25}.

\subsection{Key Ideas of the Proof}

The three thermal regimes—$T_1 > T_0$, $T_1 = T_0$, and $0 < T_0 - T_1 < T_c$—endow the system \eqref{primitive-reformulation-p} with fundamentally distinct mathematical structures. As a result, Theorem~\ref{global-stability} requires three separate proofs, one for each case. The dynamical behavior and stability characteristics differ significantly across these regimes, necessitating the use of different analytical methods and mathematical tools to rigorously establish global stability under each specific set of conditions.

For $T_1 > T_0$, the exponential decay of $\norm{v}_{L^2} + \norm{T}_{L^2}$ is directly obtained via the standard energy method. In the second case, $T_1 = T_0$, the proof follows a two-step argument: the exponential decay of $\norm{T}_{L^2}$ is proven first, and the decay of $\norm{v}_{L^2}$ then follows as a consequence. For the third regime, $0 < T_0 - T_1 < T_c$, we exploit the variational structure of the principal eigenvalue to demonstrate the exponential decay of the combined norm 
$\norm{v}_{L^2} + \norm{T}_{L^2}$.

The central challenge in the proof is to establish the exponential decay of the $H^1$-norm:
\[
\|\nabla v\|_{L^2} + \|\nabla T\|_{L^2} \leq C e^{-2\lambda t}.
\]
Owing to the strong mutual coupling among the bounds for $\|\partial_x v\|_{L^2}$, $\|\partial_z v\|_{L^2}$, $\|\partial_x T\|_{L^2}$, and $\|\partial_z T\|_{L^2}$, the conventional strategy of estimating $\nabla v$ or $\nabla T$ as single entities—which is effective for the Boussinesq equations—fails in the context of the primitive equations. To overcome this, we employ distinct iterative schemes tailored to the thermal regime. For the regimes $T_1 > T_0$ and $0 < T_0 - T_1 < T_c$, The decay is established through the following chain of implications:
\[
\begin{aligned}
&\|\partial_z v\|_{L^2} \leq C e^{-2\lambda t}
\quad\Rightarrow\quad
\|\partial_x v\|_{L^\infty} \leq C
\quad\Rightarrow\quad\\
&\|\partial_z T\|_{L^2} \leq C e^{-2\lambda t}
\quad\Rightarrow\quad
\|(\partial_x v, \partial_x T)\|_{L^2} \leq C e^{-2\lambda t}.
\end{aligned}
\]
For $T_1 = T_0$: A modified, five-step iterative procedure is required:
\[
\begin{aligned}
&\|\partial_z v\|_{L^2} \leq C e^{-2\lambda t}
\quad\Rightarrow\quad
\|\partial_x v\|_{L^\infty} \leq C
\quad\Rightarrow\quad\\
&\|\partial_z T\|_{L^2} \leq C e^{-2\lambda t}
\quad\Rightarrow\quad
\|\partial_x T\|_{L^2} \leq C e^{-2\lambda t}
\\&\Rightarrow\quad
\|\partial_x v\|_{L^2} \leq C e^{-2\lambda t}.
\end{aligned}
\]

An analogous, yet more intricate, difficulty arises in establishing the decay of the $H^2$-norm:
\[
\|\Delta v\|_{L^2} + \|\Delta T\|_{L^2} \leq C e^{-2\lambda t}.
\]
The strong nonlinear coupling, inherent to the primitive equations, again precludes a direct proof that treats the Laplacians as monolithic entities.
Reflecting the system's anisotropic structure, our proof follows a sequential, component-wise approach, proceeding from vertical to horizontal derivatives. For the regimes $T_1 > T_0$ and $0 < T_0 - T_1 < T_c$, this strategy unfolds in two stages.
We first establish the exponential decay of the second-order vertical gradients:
\[
\|\partial_z \nabla v\|_{L^2} + \|\partial_z \nabla T\|_{L^2} \leq C e^{-2\lambda t}.
\]
This result then enables the subsequent derivation of the bound for the horizontal gradients:
\[
\|\partial_x \nabla v\|_{L^2} + \|\partial_x \nabla T\|_{L^2} \leq C e^{-2\lambda t}.
\]
For the case $T_1 = T_0$, the conventional strategy of collectively estimating the horizontal gradients fails. A more nuanced, three-step iterative procedure is required:
\[
\begin{aligned}
&\|\partial_z \nabla T\|_{L^2} \leq C e^{-2\lambda t} \Rightarrow \|\partial_z \nabla v\|_{L^2} \leq C e^{-2\lambda t} \\
&\Rightarrow \|\partial_x \nabla v\|_{L^2} + \|\partial_x \nabla T\|_{L^2} \leq C e^{-2\lambda t}.
\end{aligned}
\]

The proof of Theorem~\ref{critical-stability} presents a distinct challenge at the critical threshold $T_0 - T_1 = T_c$, where the system \eqref{primitive-reformulation-p} can admit a bounded invariant set in $H^1(\Omega)$ containing infinitely many elements. To show that all such invariant sets collapse to the origin, we proceed through a multi-step argument. We first establish the existence of a global attractor. Consequently, we prove that this attractor is contained within a one-dimensional bounded invariant set in $H^1(\Omega)$. Crucially, we then demonstrate that every solution within this invariant set obeys a similarity invariance. Ultimately, by applying a carefully chosen similarity transformation, we show that the attractor must reduce to the singleton ${0}$, thereby confirming the global stability of the trivial equilibrium.

To establish the desired $H^2$-norm behavior—exponential decay for $T_0 \leq T_1$ or uniform boundedness for $T_0 > T_1$—we encounter a principal obstacle: the presence of nonlocal nonlinear terms in the system \eqref{primitive-reformulation-p} and its differentiated forms, such as
\[
\left(\int_0^z \partial_x v(t,x,\xi)  d\xi \right) \partial_z v \quad \text{and} \quad \left(\int_0^z \partial_x v(t,x,\xi)  d\xi \right) \partial_z T.
\]
The standard approach of applying Ladyzhenskaya's inequality to control the $L^2$-norm of these terms proves ineffective.
To circumvent this, we develop a refined strategy based on estimating their $L^2$-norms via mixed norms of the constituent functions. This yields the key inequalities:
\[
\begin{aligned}
&\left\| \left( \int_0^z \partial_x v(t,x,\xi)  d\xi \right) \partial_z v \right\|_{L^2}^2 \leq \| \partial_x v \|_{L_z^{2} L_x^{\infty}}^2 \| \partial_z v \|^2_{L_z^{\infty} L_x^{2}}, \\
&\left\| \left( \int_0^z \partial_x v(t,x,\xi)  d\xi \right) \partial_z T \right\|_{L^2}^2 \leq C \| \partial_z T \|^2_{L_z^{\infty} L_x^{2}} \| \partial_x v \|_{L_z^{2} L_x^{\infty}}^2.
\end{aligned}
\]
These novel estimates form the analytical foundation that enables us to derive closed differential inequalities for the quantities
\[
\|\nabla v\|_{L^{2}} + \|\nabla T\|_{L^{2}}, \quad 
\|\partial_{xx} v\|_{L^{2}} + \|\partial_{xx} T\|_{L^{2}}\quad
\text{and}\quad \|\partial_{zz} v\|_{L^{2}} + \|\partial_{zz} T\|_{L^{2}}.
\]
From these inequalities, we conclusively establish the exponential decay of these norms when $T_0 \leq T_1$, or their uniform boundedness when $T_0 > T_1$.

Establishing the nonlinear instability of fluid flows based on linear instability analysis remains a fundamental challenge in hydrodynamic stability theory. Seminal works by Sattinger \cite{MR261182}, who employed a Galerkin approach to prove that linear instability implies nonlinear instability for $H^1$ weak solutions on bounded domains, and Yudovich \cite{MR1003607}, who established similar results for the Navier-Stokes equations in $L^p(\Omega)$ ($p \geq n$, with $n$ the spatial dimension), have rigorously connected linear and nonlinear instability concepts. These developments extended earlier foundational work by Henry \cite{MR610244}.
However, such implications are not universal: linear instability does not invariably imply nonlinear instability, nor does linear stability guarantee nonlinear stability. The inherent inequivalence of norms in infinite-dimensional spaces prevents direct deductions between linear and nonlinear stability/instability for specific function space settings. Moreover, the abstract framework typically used to justify linearization in evolution equations, as comprehensively treated by Henry \cite{MR610244}, relies on technical assumptions—such as the existence of local semiflows with suitable smoothing properties or the equivalence of certain operator-induced norms—that are frequently violated in concrete hydrodynamical stability problems, as evidenced by counterexamples in Yudovich \cite{MR1003607} and the intrinsic structure of the Navier-Stokes equations.
To overcome these limitations of conventional linear-to-nonlinear arguments, we develop a boosted energy method. This approach constructs carefully tailored nonlinear energy functionals incorporating specific correction terms, which enables us to establish nonlinear instability (Theorem~\ref{noninstability}) across all $L^p$ spaces ($1 \leq p \leq \infty$) for the thermally unstable regime $T_0 - T_1 > T_c$ (Linear instability).

\section{New Formulation}

\subsection{Nondimensional Equations}
To simplify the analysis and identify key parameters, we non-dimensionalize each equation in the system \eqref{primitive-reformulation-p}. Under the assumption (without loss of generality) that $T_0 \neq 0$, we introduce the following dimensionless quantities:
\[
(x, z) = H(x', z'), \quad 
t = \frac{H^2}{\kappa_x} t', \quad 
v = \frac{\kappa_x}{H} v', \quad 
T = |T_0| \Theta', \quad 
q = \frac{\kappa_x^2}{H^2} q'.
\]

We first introduce these dimensionless quantities and then drop the primes for simplicity. The resulting dimensionless system of the system \eqref{primitive-reformulation-p} is given by:
 \begin{align}
\begin{cases}\label{ondimensional-equations}
\begin{aligned}
&\partial_tv+v \partial_x v-\left(\int_0^z\partial_xv(t,x,\xi)\,d\xi \right)\partial_z v=\text{Pr}_x\partial_{xx}v+
\text{Pr}_z\partial_{zz}v\\&- \text{R}_0^2\partial_x \left(
\int_0^z\Theta(t,x,\xi)\,d\xi\right)-\partial_xq(t,x),
\end{aligned}\\
\begin{aligned}
&\partial_t\Theta+v\partial_x \Theta-\left(\int_0^z\partial_xv(t,x,\xi)\,d\xi \right)\partial_z \Theta= \partial_{xx}\Theta+
\kappa_{a}\partial_{zz} \Theta\\
&+\frac{T_1-T_0}{\abs{T_0}}\left(\int_0^z\partial_xv(t,x,\xi)\,d\xi \right),
\end{aligned}\\
v_z|_{z=0}= v_z|_{z=1}=0, \quad
\Theta|_{z=0}=0,\quad \Theta|_{z=1}=0,\\
v, \Theta|~\text{ are periodic in $x$ with period $\alpha=L/H $},\\
\int_0^1v(t,x,z)\,dz=0,
\end{cases}
\end{align}
The nondimensional parameters appearing in the preceding equations are defined as follows:
\begin{align}\label{parameters}
\text{Pr}_x := \frac{\mu_x}{\kappa_x}, \quad
\text{Pr}_z := \frac{\mu_z}{\kappa_x}, \quad
\kappa_a := \frac{\kappa_z}{\kappa_x}, \quad
\text{R}_0 = \frac{H \sqrt{H |T_0| \rho_0 g \beta}}{\kappa_x}.
\end{align}

To further simplify the nondimensional system, we introduce a rescaling of the variable $\Theta$. The specific scaling depends on the thermal regime:
\[
\Theta(t,x,z) =
\begin{cases}
\text{R}_0^{-2} \Theta'(t,x,z), & T_0 = T_1, \\
\dfrac{\sqrt{|T_0 - T_1|} \Theta'(t,x,z)}{\text{R}_0 \sqrt{|T_0|}}, & T_0 \neq T_1.
\end{cases}
\]
After applying this rescaling and dropping the prime on $\Theta'$ for notational simplicity, the dimensionless system \eqref{ondimensional-equations} reduces to the following form:
\begin{align}
\begin{cases}\label{ondimensional-equations-2}
\begin{aligned}
&\partial_tv+v \partial_x v-\left(\int_0^z\partial_xv(t,x,\xi)\,d\xi \right)\partial_z v=\text{Pr}_x\partial_{xx}v+
\text{Pr}_z\partial_{zz}v\\&- \text{R}^{\abs{\text{Sgn}(T_0-T_1)}}\partial_x \left(
\int_0^z\Theta(t,x,\xi)\,d\xi\right)-\partial_xq(t,x),
\end{aligned}\\
\begin{aligned}
&\partial_t\Theta+v\partial_x \Theta-\left(\int_0^z\partial_xv(t,x,\xi)\,d\xi \right)\partial_z \Theta= \partial_{xx}\Theta+
\kappa_{a}\partial_{zz} \Theta\\
&-\text{Sgn}(T_0-T_1)\text{R}\left(\int_0^z\partial_xv(t,x,\xi)\,d\xi \right),
\end{aligned}\\
v_z|_{z=0}= v_z|_{z=1}=0, \quad
\Theta|_{z=0}=0,\quad \Theta|_{z=1}=0,\\
v, \Theta|~\text{ are periodic in $x$ with period $\alpha=L/H $},\\
\int_0^1v(t,x,z)\,dz=0,\quad
\int_0^{\alpha }\Theta(t,x,z)\,dx=0,
\end{cases}
\end{align}
where $\text{R}$ is the Rayleigh number and $\text{Sgn}(T_0 - T_1)$ is the sign function, defined respectively by
\begin{align}\label{parameters-2}
\begin{aligned}
&\text{R} = \frac{H \sqrt{H \rho_0 g \beta |T_0 - T_1|}}{\kappa_x}, \quad
\text{Sgn}(T_0 - T_1) = 
\begin{cases}
1,  & T_0 > T_1, \\
0,  & T_0 = T_1, \\
-1, & T_0 < T_1.
\end{cases}
\end{aligned}
\end{align}
Note that the system \eqref{ondimensional-equations-2} is defined on the domain $D$ given by
\begin{align}\label{domain}
D = (0,\alpha) \times (0, 1).
\end{align}
 
For a function $g(x, z)$, this article employs the mixed norms 
$\|g\|_{L^{\infty}_{x} L^{2}_{z}}^2$, 
$\|g\|_{L^{2}_{x} L^{\infty}_{z}}^2$,
$\|g\|_{L^{\infty}_{z} L^{2}_{x}}^2$, and
$\|g\|_{L^{2}_{z} L^{\infty}_{x}}^2$
to estimate nonlinear terms in the system \eqref{ondimensional-equations-2}, 
which are defined as follows:
\begin{align}
\begin{aligned}
&\|g\|_{L^{\infty}_{x} L^{2}_{z}}^2 = \left\| \int_0^1 |g(\cdot, z)|^2  dz \right\|_{L^{\infty}_{x}}, \quad
\|g\|_{L^{2}_{x} L^{\infty}_{z}}^2 = \int_0^1 \|g(x, \cdot)\|_{L^{\infty}_{z}}^2  dx, \\
&\|g\|_{L^{\infty}_{z} L^{2}_{x}}^2 = \left\| \int_0^1 |g(x, \cdot)|^2  dx \right\|_{L^{\infty}_{z}}, \quad
\|g\|_{L^{2}_{z} L^{\infty}_{x}}^2 = \int_0^1 \|g(\cdot, z)\|_{L^{\infty}_{x}}^2  dz.
\end{aligned}
\end{align}

\subsection{Proof of Theorem \ref{steady-one}}\label{section2.2}
\begin{proof}
Since the system \eqref{ondimensional-equations-2} is equivalent to the equations \eqref{primitive-reformulation-p}, we will use the former to establish Theorem~\ref{steady-one}. To this end, we consider its steady-state formulation:
\begin{align}
\begin{cases}\label{pes-xz-ns}
\begin{aligned}
&v \partial_x v-\left(\int_0^z\partial_xv(t,x,\xi)\,d\xi \right)\partial_z v=\text{Pr}_x\partial_{xx}v+
\text{Pr}_z\partial_{zz}v\\&- \text{R}^{\abs{\text{Sgn}(T_0-T_1)}}\partial_x \left(
\int_0^z\Theta(t,x,\xi)\,d\xi\right)-\partial_xq(t,x),
\end{aligned}\\
\begin{aligned}
&v\partial_x \Theta-\left(\int_0^z\partial_xv(t,x,\xi)\,d\xi \right)\partial_z \Theta= \partial_{xx}\Theta+
\kappa_{a}\partial_{zz} \Theta\\
&-\text{Sgn}(T_0-T_1)\text{R}\left(\int_0^z\partial_xv(t,x,\xi)\,d\xi \right),
\end{aligned}\\
v_z|_{z=0}= v_z|_{z=1}=0, \quad
\Theta|_{z=0}=0,\quad \Theta|_{z=1}=0,\\
v, T~\text{ are periodic in $x$ with period $\alpha=L/H $},\\
\int_0^1v(t,x,z)\,dz=0,\quad
\int_0^{\alpha }\Theta(t,x,z)\,dx=0.
\end{cases}
\end{align}

From the second equation in \eqref{pes-xz-ns}, together with an integration by parts, we find
\begin{align}\label{steady-1}
\begin{aligned}
0&=\int_{D}\left(v\partial_x \Theta-\left(\int_0^z\partial_xv(t,x,\xi)\,d\xi \right)\partial_z \Theta\right)\Theta\,dx\,dz\\
&\quad-\int_{D}\Theta \partial_{xx}\Theta\,dx\,dz
-\kappa_{a}\int_{D}\Theta \partial_{zz}\Theta\,dx\,dz
\\&\quad+\text{Sgn}(T_0-T_1)\text{R}\int_{D}\Theta
\left(\int_0^z\partial_xv(t,x,\xi)\,d\xi\right)
\,dx\,dz
\\&= \int_{D}\abs{ \partial_{x}\Theta}^2\,dx\,dz
+\kappa_{a}
\int_{D}\abs{ \partial_{z}\Theta}^2\,dx\,dz
\\&\quad+\text{Sgn}(T_0-T_1)\text{R}\int_{D}\Theta
\left(\int_0^z\partial_xv(t,x,\xi)\,d\xi\right).
\end{aligned}
\end{align}

Multiplying the first equation in \eqref{pes-xz-ns} by $v$ in $L^2$ and integrating by parts yields
\begin{align}\label{steady-222}
\begin{aligned}
&\text{R}^{|\text{Sgn}(T_0-T_1)|}
\int_{D} \left( \Theta \int_0^z \partial_x v(t,x,\xi)  d\xi \right)  dx  dz \\
&= \text{Pr}_x \int_{D} |\partial_x v|^2  dx  dz + 
\text{Pr}_z \int_{D} |\partial_z v|^2  dx  dz,
\end{aligned}
\end{align}
where we have used the two identities:
\begin{align*}
\int_{D} \left( v \partial_x v - \int_0^z \partial_x v(t,x,\xi)  d\xi  \partial_z v \right) v  dx  dz = 0,
\end{align*}
and
\begin{align*}
\begin{aligned}
&\int_{D} \left( -\text{R}^{|\text{Sgn}(T_0-T_1)|} \partial_x \left( \int_0^z \Theta(t,x,\xi)  d\xi \right) - \partial_x q(t,x) \right) v  dx  dz \\
&= \int_{D} \left( \text{R}^{|\text{Sgn}(T_0-T_1)|} \int_0^z \Theta(t,x,\xi)  d\xi + q \right) \partial_x v  dx  dz \\
&= \int_{D} \left( \text{R}^{|\text{Sgn}(T_0-T_1)|} \int_0^z \Theta(t,x,\xi)  d\xi + q \right) \partial_z \left( \int_0^z \partial_x v(t,x,\xi)  d\xi \right)  dx  dz \\
&= -\text{R}^{|\text{Sgn}(T_0-T_1)|} \int_{D} \Theta \left( \int_0^z \partial_x v(t,x,\xi)  d\xi \right)  dx  dz.
\end{aligned}
\end{align*}
Consequently, for $T_1 > T_0$, it follows from \eqref{steady-1}--\eqref{steady-222} that
\begin{align*}
\begin{aligned}
&\text{Pr}_x \int_{D} |\partial_x v|^2  dx  dz + \text{Pr}_z \int_{D} |\partial_z v|^2  dx  dz \\
&+ \int_{D} |\partial_x \Theta|^2  dx  dz + \kappa_a \int_{D} |\partial_z \Theta|^2  dx  dz = 0.
\end{aligned}
\end{align*}
Together with the constraints $\int_0^H v(t,x,\xi)  d\xi = 0$ and $\Theta|_{z=0,1} = 0$, this yields
\[
v = \Theta = 0.
\]

For $T_1 < T_0$, combining \eqref{steady-1}--\eqref{steady-222} leads to \eqref{steady-231-1}. When $T_1 = T_0$, equation \eqref{steady-1} directly gives
\[
\int_{D} |\partial_x \Theta|^2  dx  dz + \kappa_a \int_{D} |\partial_z \Theta|^2  dx  dz = 0.
\]
Given the condition $\Theta|_{z=0,1} = 0$, it follows that 
$\Theta = 0$.
Substituting $\Theta = 0$ into \eqref{steady-222} yields
\begin{align*}
\text{Pr}_x \int_{D} |\partial_x v|^2  dx  dz + \text{Pr}_z \int_{D} |\partial_z v|^2  dx  dz = 0.
\end{align*}
Combined with $\int_0^1 v(t,x,\xi)  d\xi = 0$, this implies $v = 0$.
\end{proof}

\section{Proof of Theorem \ref{global-stability}}\label{theorem21}
To prove Theorem~\ref{global-stability}, we first note that systems \eqref{ondimensional-equations-2} and \eqref{primitive-reformulation-p} are mathematically equivalent. This allows us to restrict the stability analysis to the formulation in \eqref{ondimensional-equations-2}.
A central observation is that the three distinct thermal regimes—$T_1 > T_0$, $T_1 = T_0$, and $0 < T_0 - T_1 < T_c$—endow the system \eqref{ondimensional-equations-2} with fundamentally different mathematical structures. Consequently, a unified treatment is not feasible, and the stability analysis must be carried out separately for each case.

\subsection{Case $T_1 > T_0$}
For the case $T_1 > T_1$, Theorem \ref{global-stability} follows from Lemma~\ref{v-l2-l2}--Lemma \ref{lemma-grad-1-8}.

\subsubsection{$L^2$-Estimates and Global Stability}
\begin{lemma}\label{v-l2-l2}
Let $T_1 > T_0$ and consider any smooth solution $(v, \Theta)$ to the system \eqref{ondimensional-equations-2} with initial data $(v_0, \Theta_0) \in L^2(D)$. Then the following uniform-in-time estimates hold:
\begin{align}\label{u21-1-1}
\begin{aligned}
&v \in L^\infty\big( (0,\infty); L^2(D) \big) \cap L^2\big( (0,\infty); H^1(D) \big), \\
&\Theta \in L^\infty\big( (0,\infty); L^2(D) \big) \cap L^2\big( (0,\infty); H^1(D) \big).
\end{aligned}
\end{align}
Moreover, there exists a constant $\lambda > 0$ such that the solution decays exponentially in time:
\begin{align}
\norm{(v(t), \Theta(t))}_{L^2(D)}^2 \leq e^{-2\lambda t} \norm{(v_0, \Theta_0)}_{L^2(D)}^2, \quad \forall t \geq 0.
\end{align}
\end{lemma}
\begin{proof}
For $T_1 > T_0$, multiplying the equation \eqref{ondimensional-equations-2}$_1$ by $v$ and the equation \eqref{ondimensional-equations-2}$_2$ by $\Theta$, then adding the resulting expressions, we obtain
\begin{align}\label{case1-L2-1} 
\begin{aligned}
\frac{d}{dt} \int_{D} \left( v^2 + \Theta^2 \right) dx dz 
= & -2\text{Pr}_x \int_{D} |\partial_x v|^2 dx dz 
   - 2\text{Pr}_z \int_{D} |\partial_z v|^2 dx dz \\
  & - 2 \int_{D} |\partial_x \Theta|^2 dx dz 
   - 2\kappa_a \int_{D} |\partial_z \Theta|^2 dx dz,
\end{aligned}
\end{align}
where we have used the identity
\begin{align*}
\begin{aligned}
&\int_{D} \left( -\partial_x \left( \int_0^z \Theta(t,x,\xi) d\xi \right) - \partial_x q(t,x) \right) v dx dz 
 + \int_{D} \Theta \left( \int_0^z \partial_x v(t,x,\xi) d\xi \right) dx dz \\
&= \int_{D} \left( \int_0^z \Theta(t,x,\xi) d\xi + q \right) \partial_x v dx dz 
 + \int_{D} \Theta \left( \int_0^z \partial_x v(t,x,\xi) d\xi \right) dx dz \\
&= \int_{D} \left( \int_0^z \Theta(t,x,\xi) d\xi + q \right) \partial_z \left( \int_0^z \partial_x v(t,x,\xi) d\xi \right) dx dz \\
&\quad + \int_{D} \Theta \left( \int_0^z \partial_x v(t,x,\xi) d\xi \right) dx dz = 0.
\end{aligned}
\end{align*}

Integrating \eqref{case1-L2-1} over the time interval $[0, t]$ leads to
\begin{align*}
\begin{aligned}
&\norm{v(t)}_{L^2}^2 + \norm{\Theta(t)}_{L^2}^2 + 2\int_0^t \left( \text{Pr}_x \int_{D} |\partial_x v|^2 dx dz + \text{Pr}_z \int_{D} |\partial_z v|^2 dx dz \right) dt' \\
& + 2\int_0^t \left( \int_{D} |\partial_x \Theta|^2 dx dz + \kappa_a \int_{D} |\partial_z \Theta|^2 dx dz \right) dt' 
= \norm{v_0}_{L^2}^2 + \norm{\Theta_0}_{L^2}^2,
\end{aligned}
\end{align*}
which establishes \eqref{u21-1-1}. An application of the Poincaré inequality yields
\begin{align*}
\int_{D} \left( v^2 + \Theta^2 \right) dx dz \leq C \int_{D} \left( |\nabla v|^2 + |\nabla \Theta|^2 \right) dx dz.
\end{align*}
Together with \eqref{case1-L2-1}, this implies the existence of a constant $\lambda > 0$ such that
\begin{align*}
\frac{d}{dt} \int_{D} \left( v^2 + \Theta^2 \right) dx dz \leq -2\lambda \int_{D} \left( v^2 + \Theta^2 \right) dx dz.
\end{align*}
Applying Gronwall's inequality then yields
\[
\norm{v(t)}_{L^2(D)}^2 + \norm{\Theta(t)}_{L^2(D)}^2 \leq e^{-2\lambda t} \left( \norm{v_0}_{L^2(D)}^2 + \norm{\Theta_0}_{L^2(D)}^2 \right).
\]
\end{proof}
\subsubsection{Compatibility Condition}

Setting $\sigma = \partial_z v$, we deduce from the system \eqref{ondimensional-equations-2} that $(\sigma, \Theta)$ satisfies
\begin{align}\label{new-eq-vz-1}
\begin{cases}
\begin{aligned}
\partial_t \sigma = &\ \text{Pr}_x \partial_{xx} \sigma + \text{Pr}_z \partial_{zz} \sigma - \text{R}  \partial_x \Theta \\
& - v \partial_x \sigma + \left( \int_0^z \partial_x v(t,x,\xi)  d\xi \right) \partial_z \sigma,
\end{aligned} \\
\begin{aligned}
\partial_t \Theta = &\ \partial_{xx} \Theta + \kappa_a \partial_{zz} \Theta + \text{R} \left( \int_0^z \partial_x v(t,x,\xi)  d\xi \right) \\
& + v \partial_x \Theta - \left( \int_0^z \partial_x v(t,x,\xi)  d\xi \right) \partial_z \Theta,
\end{aligned} \\
\sigma |_{z=0} = \sigma |_{z=1} = 0, \quad \Theta |_{z=0} = \Theta |_{z=1} = 0, \\
\sigma, \Theta \text{ are periodic in $x$ with period $\alpha$}, \\
\sigma |_{t=0} = \sigma_0.
\end{cases}
\end{align}

Introducing the stream function $\Psi$ via $v = \partial_z \Psi$, where $\Psi$ satisfies
\begin{align}\label{new-eq-111}
\begin{cases}
\partial_{zz} \Psi = \sigma, \\
\Psi |_{z=0,1} = 0,
\end{cases}
\end{align}
we observe that the system \eqref{new-eq-vz-1} is underdetermined. To ensure the existence of a global solution and to identify $\Psi$ as the stream function satisfying
\[
(v, w) = \left( \partial_z \Psi, -\partial_x \Psi \right) = \left( v, -\int_0^z \partial_x v(t,x,\xi)  d\xi \right),
\]
an additional compatibility condition between $\sigma$ and $\Psi$ is required. This condition is given by
\begin{align}\label{compatibility-cond}
\partial_z \Psi |_{z=0} = \int_0^1 \xi \sigma(\cdot, \xi)  d\xi - \int_0^1 \sigma(\cdot, \xi)  d\xi.
\end{align}

Integrating the first equation in \eqref{new-eq-111} with respect to $z$ from $0$ to $z$ yields
\[
\partial_z \Psi = \int_0^z \sigma(\cdot, \xi)  d\xi + C,
\]
where the constant of integration $C$ is determined by the compatibility condition \eqref{compatibility-cond} as
\[
C = \int_0^1 \xi \sigma(\cdot, \xi)  d\xi - \int_0^1 \sigma(\cdot, \xi)  d\xi.
\]
Thus,
\[
\partial_z \Psi = \int_0^z \sigma(\cdot, \xi)  d\xi + \left( \int_0^1 \xi \sigma(\cdot, \xi)  d\xi - \int_0^1 \sigma(\cdot, \xi)  d\xi \right).
\]

Integrating once more with respect to $z$ from $0$ to $z$, we obtain
\begin{align*}
\begin{aligned}
\Psi &= \int_0^z \left( \int_0^y \sigma(\cdot, \xi)  d\xi \right) dy + z \left( \int_0^1 \xi \sigma(\cdot, \xi)  d\xi - \int_0^1 \sigma(\cdot, \xi)  d\xi \right) \\
&= z \int_0^z \sigma(\cdot, \xi)  d\xi - \int_0^z \xi \sigma(\cdot, \xi)  d\xi + z \left( \int_0^1 \xi \sigma(\cdot, \xi)  d\xi - \int_0^1 \sigma(\cdot, \xi)  d\xi \right).
\end{aligned}
\end{align*}
It follows that $\Psi|_{z=1} = 0$, confirming the boundary condition.

Using the relations $v = \partial_z \Psi$ and $\sigma = \partial_z v$, we simplify the expression for $\Psi$ to
\begin{align*}
\Psi = z \int_0^z \sigma(\cdot, \xi)  d\xi - \int_0^z \xi \sigma(\cdot, \xi)  d\xi + z v|_{z=0} = \int_0^z v(\cdot, \xi)  d\xi.
\end{align*}
Consequently,
\[
\partial_x \Psi = \int_0^z \partial_x v(\cdot, \xi)  d\xi = -w,
\]
which verifies that $\Psi$ is indeed the stream function for the velocity field $(v, w)$.

\subsubsection{$H^1$-Estimates and Global Stability}

\begin{lemma}\label{lemma-grad-1-2}
Let $T_1 > T_0$ and consider any smooth solution $(v, \Theta)$ to the system \eqref{ondimensional-equations-2} emanating from initial data $(v_0, \Theta_0) \in L^2(D)$ with $\partial_z v_0 \in L^2(D)$. The following regularity properties hold:
\begin{align}\label{infty-vz-1}
\partial_z v \in L^\infty\big( (0,\infty); L^2(D) \big), \quad
\nabla \partial_z v \in L^2\big( (0,\infty); L^2(D) \big).
\end{align}
Moreover, there exist constants $C > 0$ (independent of the initial data) and $\lambda > 0$ such that the vertical derivative of $v$ decays exponentially in time:
\begin{align}
\norm{\partial_z v(t)}_{L^2(D)}^2 \leq C e^{-2\lambda t} \left( \norm{(v_0, \Theta_0)}_{L^2(D)}^2 + \norm{\partial_z v_0}_{L^2(D)}^2 \right), \quad \forall t \geq 0.
\end{align}
\end{lemma}

\begin{proof}
Set $T_1 > T_0$ in \eqref{ondimensional-equations-2}. Observe that $\sigma = \partial_z v$ satisfies the initial-boundary value problem
\begin{align}\label{three-proof--8}
\begin{cases}
\begin{aligned}
\partial_t \sigma = &\ \text{Pr}_x \partial_{xx} \sigma + \text{Pr}_z \partial_{zz} \sigma - \text{R}  \partial_x \Theta \\
& - v \partial_x \sigma + \left( \int_0^z \partial_x v(t,x,\xi)  d\xi \right) \partial_z \sigma,
\end{aligned} \\
\sigma |_{z=0} = 0, \quad \sigma |_{z=1} = 0, \\
\sigma |_{t=0} = \partial_z v_0.
\end{cases}
\end{align}
Multiplying the equation \eqref{three-proof--8} by $\sigma$ and integrating by parts over $D$, we obtain
\begin{align}\label{three-proof--9-1}
\begin{aligned}
\frac{1}{2} \frac{d}{dt} \int_D |\sigma|^2 dx dz 
&= -\text{Pr}_x \int_D |\partial_x \sigma|^2 dx dz - \text{Pr}_z \int_D |\partial_z \sigma|^2 dx dz + \text{R} \int_D \Theta  \partial_x \sigma  dx dz \\
&\leq - \frac{\lambda}{2} \left( \int_D |\partial_x \sigma|^2 dx dz + \int_D |\partial_z \sigma|^2 dx dz \right) + C \int_D \Theta^2 dx dz \\
&\leq - \frac{\lambda}{2} \int_D |\sigma|^2 dx dz + C \int_D \Theta^2 dx dz,
\end{aligned}
\end{align}
where the last inequality follows from the Poincaré inequality.

From this, Lemma~\ref{v-l2-l2}, and \eqref{general-inequality1}, it follows that there exists $\lambda > 0$ such that
\begin{align*}
\norm{\sigma(t)}_{L^2(D)}^2 \leq C e^{-\lambda t} \left( \norm{v_0}_{L^2(D)}^2 + \norm{\Theta_0}_{L^2(D)}^2 + \norm{\partial_z v_0}_{L^2(D)}^2 \right).
\end{align*}
Finally, integrating \eqref{three-proof--9-1} in time yields the regularity properties stated in \eqref{infty-vz-1}.
\end{proof}

\begin{lemma}\label{lemma-grad-1-3}
Let $T_1 > T_0$ and consider any smooth solution $(v, \Theta)$ to the system \eqref{ondimensional-equations-2} emanating from initial data $(v_0, \Theta_0) \in L^2(D)$ with $\nabla v_0 \in L^2(D)$. Then the horizontal derivative of $v$ satisfies the following regularity estimates:
\begin{align}\label{infty}
\partial_x v \in L^\infty\big( (0,\infty); L^2(D) \big) \cap L^2\big( (0,\infty); H^1(D) \big).
\end{align}
\end{lemma}
\begin{proof}
For the case $T_1 > T_0$ in the system \eqref{ondimensional-equations-2}, we multiply its first equation by $\partial_{xx} v$ and integrate over the domain, which yields
\begin{align}\label{estimatevx-infinity}
\begin{aligned}
&\frac{d}{dt}\norm{\partial_x v}_{L^2}^2
+2\text{Pr}_x \int_{D}\abs{\partial_{xx}v}^2
+2\text{Pr}_z \int_{D}\abs{\partial_{xz}v}^2
\\&=2\int_{D}
\left(v \partial_x v-\left(\int_0^z\partial_xv(t,x,\xi)\,d\xi \right)\partial_z v
\right)\partial_{xx} v\,dx\,dz
\\&\quad+2\int_{D}
\left(\text{R}
\partial_x \left(
\int_0^z\Theta(t,x,\xi)\,d\xi\right)+\partial_x q(t,x)
\right)\partial_z
\left(\int_0^z\partial_{xx}v(t,x,\xi)\,d\xi \right)
\,dx\,dz
\\&=2\int_{D}
\left(v \partial_x v-\left(\int_0^z\partial_xv(t,x,\xi)\,d\xi \right)\partial_z v
\right)\partial_{xx} v\,dx\,dz
\\&\quad-2\text{R}\int_{D}
\partial_x \Theta
\left(\int_0^z\partial_{xx}v(t,x,\xi)\,d\xi \right)
\,dx\,dz=:H_1+H_2+H_3.
\end{aligned}
\end{align}

We now estimate the terms $H_j$ for $j = 1, 2, 3$. For $H_1$, 
with the help of the Cauchy–Schwarz inequality, we obtain
\begin{align}\label{estimateh1-1}
\begin{aligned}
H_1 &= 2 \int_D (v \partial_x v) \partial_{xx} v  dx dz \leq 2 \norm{\partial_{xx} v}_{L^2} \norm{v \partial_x v}_{L^2}.
\end{aligned}
\end{align}
To control the nonlinear term in \eqref{estimateh1}, we employ the following inequalities
\begin{align}\label{ut-proof-333-}
\begin{aligned}
&\norm{v \partial_x v}_{L^2}^2 \leq \int_0^\alpha \left( \norm{|v|^2}_{L_z^\infty} \int_0^1 (\partial_x v)^2 dz \right) dx 
\leq \norm{\partial_x v}_{L_x^\infty L_z^2}^2 \norm{v}_{L_x^2 L_z^\infty}^2.
\end{aligned}
\end{align}
Furthermore, based on Lemma \ref{mixed-norm}-
 Lemma \ref{mixed-norm-1}, the following bounds hold:
\begin{align}\label{ut-proof-333}
\begin{aligned}
&\norm{\partial_x v}_{L_x^\infty L_z^2}^2 \leq 2 \int_0^1 \left( \int_0^\alpha |\partial_x v \partial_{xx} v| dx \right) dz 
\leq 2 \norm{\partial_x v}_{L^2} \norm{\partial_{xx} v}_{L^2}, \\
&\norm{v}_{L_x^2 L_z^\infty}^2 \leq 2 \int_0^\alpha \left( \int_0^1 |v \partial_z v| dz \right) dx 
\leq 2 \norm{\partial_z v}_{L^2} \norm{v}_{L^2}.
\end{aligned}
\end{align}
Substituting these into the previous estimate \eqref{estimateh1-1} 
and applying Young’s inequality yields
\begin{align}\label{estimateh1}
\begin{aligned}
H_1 \leq \epsilon \norm{\partial_{xx} v}_{L^2}^2 + C_\epsilon \norm{v}_{L^2}^2 \norm{\partial_z v}_{L^2}^2 \norm{\partial_x v}_{L^2}^2,
\end{aligned}
\end{align}

We now turn to the estimate for $H_2$. Recall that
\[
H_2 = -2 \int_D \left( \int_0^z \partial_x v(t,x,\xi)  d\xi \right) \partial_z v  \partial_{xx} v  dx dz.
\]
Applying the Cauchy–Schwarz inequality and Young’s inequality, we obtain
\begin{align}\label{estimateh2}
\begin{aligned}
H_2 &\leq 2 \norm{\partial_{xx} v}_{L^2} \cdot \norm{ \left( \int_0^z \partial_x v(t,x,\xi)  d\xi \right) \partial_z v }_{L^2} \\
&\leq \epsilon \norm{\partial_{xx} v}_{L^2}^2 + C_\epsilon \norm{ \left( \int_0^z \partial_x v(t,x,\xi)  d\xi \right) \partial_z v }_{L^2}^2.
\end{aligned}
\end{align}
To control the nonlinear term, we note that
\begin{align*}
\begin{aligned}
\norm{ \left( \int_0^z \partial_x v(t,x,\xi)  d\xi \right) \partial_z v }_{L^2}^2  &\leq \int_0^1 \left( \norm{ \int_0^z \partial_x v(\cdot,\xi)  d\xi }_{L_x^\infty}^2 \int_0^\alpha |\partial_z v|^2 dx \right) dz \\
&\leq \left( \int_0^1 \norm{ \int_0^z \partial_x v(\cdot,\xi)  d\xi }_{L_x^\infty}^2 dz \right) \norm{\partial_z v}_{L_z^\infty L_x^2}^2 \\
&\leq \norm{\partial_x v}_{L_z^2 L_x^\infty}^2 \cdot \norm{\partial_z v}_{L_z^\infty L_x^2}^2.
\end{aligned}
\end{align*}
Lemma \ref{mixed-norm}-
 Lemma \ref{mixed-norm-1} means that the following bounds hold:
\begin{align*}
\begin{aligned}
&\norm{\partial_x v}_{L_z^2 L_x^\infty}^2 \leq 2 \norm{\partial_x v}_{L^2} \norm{\partial_{xx} v}_{L^2}, \\
&\norm{\partial_z v}_{L_z^\infty L_x^2}^2 \leq 2 \norm{\partial_z v}_{L^2} \norm{\partial_{zz} v}_{L^2}.
\end{aligned}
\end{align*}
Substituting these into the previous estimate \eqref{estimateh2} yields
\[
\norm{ \left( \int_0^z \partial_x v(t,x,\xi)  d\xi \right) \partial_z v }_{L^2}^2 \leq 4 \norm{\partial_x v}_{L^2} \norm{\partial_{xx} v}_{L^2} \norm{\partial_z v}_{L^2} \norm{\partial_{zz} v}_{L^2}.
\]
Inserting this into \eqref{estimateh2} and applying Young's inequality, we conclude that
\begin{align*}
H_2 &\leq \epsilon \norm{\partial_{xx} v}_{L^2}^2 + C_\epsilon \norm{\partial_x v}_{L^2} \norm{\partial_{xx} v}_{L^2} \norm{\partial_z v}_{L^2} \norm{\partial_{zz} v}_{L^2} \\
&\leq \epsilon \norm{\partial_{xx} v}_{L^2}^2 + C_\epsilon \norm{\partial_z v}_{L^2}^2 \norm{\partial_{zz} v}_{L^2}^2 \norm{\partial_x v}_{L^2}^2.
\end{align*}

We now estimate the term $H_3$, defined by
\[
H_3 = \text{R} \int_D \partial_x \Theta \left( \int_0^z \partial_{xx} v(t,x,\xi)  d\xi \right) dx dz.
\]
Applying the Cauchy–Schwarz inequality, we obtain
\begin{align}\label{estimateh3-1}
\begin{aligned}
H_3 &\leq \text{R} \cdot \norm{\partial_x \Theta}_{L^2} \cdot \norm{ \int_0^z \partial_{xx} v(t,x,\xi)  d\xi }_{L^2}.
\end{aligned}
\end{align}
Note that by Cauchy–Schwarz inequality in the vertical direction, we have
\[
\norm{ \int_0^z \partial_{xx} v(t,x,\xi)  d\xi }_{L^2} \leq C \norm{\partial_{xx} v}_{L^2}.
\]
Substituting this bound into \eqref{estimateh3-1} and applying Young's inequality yields
\[
H_3 \leq \epsilon \norm{\partial_{xx} v}_{L^2}^2 + C_\epsilon \norm{\partial_x \Theta}_{L^2}^2.
\]

By selecting $\epsilon$ sufficiently small and combining the estimates for $H_1$, $H_2$, and $H_3$, we arrive at the following differential inequality:
\begin{align*}
\begin{aligned}
\frac{d}{dt} \norm{\partial_x v}_{L^2}^2 
&+ \text{Pr}_x \int_D |\partial_{xx} v|^2  dx  dz 
+ \text{Pr}_z \int_D |\partial_{xz} v|^2  dx  dz \\
&\leq \phi(t) \norm{\partial_x v}_{L^2}^2 + h(t),
\end{aligned}
\end{align*}
where the functions $h(t)$ and $\phi(t)$ are defined as
\begin{align*}
h(t) &= C \norm{\partial_x \Theta}_{L^2}^2, \\
\phi(t) &= C \norm{\partial_z v}_{L^2}^2 \left( \norm{v}_{L^2}^2 + \norm{\partial_{zz} v}_{L^2}^2 \right).
\end{align*}

Integrating the preceding differential inequality with respect to time from $0$ to $t$, we obtain
\begin{align}\label{onexxx}
\begin{aligned}
&\norm{\partial_x v(t)}_{L^2}^2
+ \int_0^t \left( \text{Pr}_x \int_D |\partial_{xx} v|^2 + \text{Pr}_z \int_D |\partial_{xz} v|^2 \right) dt' \\
&\leq \int_0^t \phi(t') \norm{\partial_x v(t')}_{L^2}^2 dt' 
+ \int_0^t h(t') dt' + \norm{\partial_x v_0}_{L^2}^2.
\end{aligned}
\end{align}
This implies the integral inequality
\begin{align*}
\begin{aligned}
\norm{\partial_x v(t)}_{L^2}^2 \leq \int_0^t \phi(t') \norm{\partial_x v(t')}_{L^2}^2 dt' + \psi(t),
\end{aligned}
\end{align*}
where $\phi(t)$ and $\psi(t)$ satisfy
\[
\begin{aligned}
&\int_0^\infty \phi(t') dt' < C \int_0^\infty \norm{\nabla \partial_z v}_{L^2}^2 dt' < \infty, \\
&\psi(t) := \int_0^t h(t') dt' + \norm{\partial_x v_0}_{L^2}^2 
\\&\leq C \left( \norm{v_0}_{L^2(D)}^2 + \norm{\Theta_0}_{L^2(D)}^2 + \norm{\partial_z v_0}_{L^2}^2 \right).
\end{aligned}
\]
Applying Gronwall's inequality yields
\begin{align*}
\begin{aligned}
\norm{\partial_x v(t)}_{L^2}^2 
&\leq \psi(t) + \int_0^t \phi(s) \psi(s) \exp\left( \int_s^t \phi(t') dt' \right) ds \\
&\leq C \left( \norm{v_0}_{L^2(D)}^2 + \norm{\Theta_0}_{L^2(D)}^2 + \norm{\partial_x v_0}_{L^2}^2 \right) \\&\quad
\times \exp\left( C \left( \norm{v_0}_{L^2(D)}^2 + \norm{\Theta_0}_{L^2(D)}^2 + \norm{\partial_z v_0}_{L^2}^2 \right) \right).
\end{aligned}
\end{align*}

Finally, returning to \eqref{onexxx}, we get 
the following uniform-in-time estimates:
\[
\int_0^t \left( \text{Pr}_x \int_D |\partial_{xx} v|^2 + \text{Pr}_z \int_D |\partial_{xz} v|^2 \right) dt'  <C_0< \infty,
\]
where the constant $C_0$ depends on $\norm{v_0}_{H^1(D)}$ and $\norm{\Theta_0}_{L^2(D)}$.
\end{proof}

\begin{lemma}\label{lemma-grad-1-4}
Let $T_1 > T_0$ and consider any smooth solution $(v, \Theta)$ to the system \eqref{ondimensional-equations-2} with initial data $(v_0, \Theta_0) \in H^1(D)$. Then the vertical derivative of temperature satisfies the regularity estimate
\begin{align}\label{infty-thetaz-1}
\partial_z \Theta \in L^\infty\big( (0,\infty); L^2(D) \big) \cap L^2\big( (0,\infty); H^1(D) \big).
\end{align}
Moreover, there exists a constant $\lambda > 0$ such that
\begin{align}\label{decaythetaz-1}
\norm{\partial_z \Theta(t)}_{L^2(D)}^2 \leq C C_0 e^{-2\lambda t} \norm{(v_0, \Theta_0)}_{H^1(D)}^2, \quad \forall t \geq 0,
\end{align}
where the constant $C_0$ depends on $\norm{v_0}_{H^1(D)}$ and $\norm{\Theta_0}_{L^2(D)}$.
\end{lemma}
 \begin{proof}Setting $T_1 > T_0$ in \eqref{ondimensional-equations-2}, multiplying  \eqref{ondimensional-equations-2}$_2$ by $\partial_{zz} \Theta$ and integrating by parts yields
\begin{align}
\begin{aligned}
&\frac{d}{dt} \norm{\partial_z \Theta}_{L^2}^2
+ 2\kappa_a \int_D |\partial_{zz} \Theta|^2 dx dz + 2 \int_D |\partial_{xz} \Theta|^2 dx dz \\
&= 2\text{R} \int_D \partial_z \Theta  \partial_x v dx dz 
+ 2 \int_D \left( v \partial_x \Theta - \left( \int_0^z \partial_x v(t,x,\xi) d\xi \right) \partial_z \Theta \right) \partial_{zz} \Theta dx dz \\
&= -2\text{R} \int_D v  \partial_{zx} \Theta dx dz 
+ 2 \int_D v \partial_x \Theta  \partial_{zz} \Theta dx dz 
+ \int_D \partial_x v (\partial_z \Theta)^2 dx dz \\
&= -2\text{R} \int_D v  \partial_{zx} \Theta dx dz 
+ 2 \int_D v \partial_x \Theta  \partial_{zz} \Theta dx dz 
- 2 \int_D v (\partial_z \Theta) (\partial_{xz} \Theta) dx dz \\
&= I_1 + I_2 + I_3.
\end{aligned}
\end{align}

We now estimate the terms $I_j$ for $j = 1, 2, 3$. For $I_1$ and $I_2$, 
applying Young’s inequality yields
\begin{align*}
&\begin{aligned}
I_1 &:= -2\text{R} \int_D v  \partial_{zx} \Theta dx dz \leq \epsilon \norm{\partial_{zx} \Theta}_{L^2}^2 + C_\epsilon \norm{v}_{L^2}^2.
\end{aligned} \\
&\begin{aligned}
I_2 &:= 2 \int_D v \partial_x \Theta  \partial_{zz} \Theta dx dz \leq \epsilon \norm{\partial_{zz} \Theta}_{L^2}^2 + C_\epsilon \norm{v \partial_x \Theta}_{L^2}^2 \\
&\leq \epsilon \norm{\partial_{zz} \Theta}_{L^2}^2 + C_\epsilon \norm{v}_{L^2} \norm{\partial_x \Theta}_{L^2} \norm{\partial_{xz} \Theta}_{L^2} \left( \norm{v}_{L^2} + \norm{\partial_x v}_{L^2} \right) \\
&\leq \epsilon \norm{\partial_{zz} \Theta}_{L^2}^2 + \epsilon \norm{\partial_{zx} \Theta}_{L^2}^2 + C_\epsilon \norm{v}_{L^2}^2 \norm{\partial_x \Theta}_{L^2}^2 \left( \norm{v}_{L^2} + \norm{\partial_x v}_{L^2} \right)^2,
\end{aligned}
\end{align*}
where we have used the inequalities
\begin{align*}
&\begin{aligned}
&\norm{v \partial_x \Theta}_{L^2}^2 \leq \int_0^\alpha \left( \norm{\partial_x \Theta}_{L_z^\infty}^2 \int_0^1 |v|^2 dz \right) dx \leq \norm{\partial_x \Theta}_{L_x^2 L_z^\infty}^2 \norm{v}_{L_x^\infty L_z^2}^2, \\
&\norm{\partial_x \Theta}_{L_x^2 L_z^\infty}^2 \leq 2 \int_D |\partial_x \Theta  \partial_{xz} \Theta| dx dz \leq 2 \norm{\partial_x \Theta}_{L^2} \norm{\partial_{xz} \Theta}_{L^2}, \\
&\norm{v}_{L_x^\infty L_z^2}^2 \leq C \norm{v}_{L^2} \left( \norm{v}_{L^2} + \norm{\partial_x v}_{L^2} \right),
\end{aligned}
\end{align*}
which are derived by using Lemma \ref{mixed-norm}-
 Lemma \ref{mixed-norm-1}.

For the third term $I_3$, we have
\begin{align*}
\begin{aligned}
I_3 &:= -2 \int_D v (\partial_z \Theta) (\partial_{zx} \Theta) dx dz \leq \epsilon \norm{\partial_{zx} \Theta}_{L^2}^2 + C_\epsilon \norm{v \partial_z \Theta}_{L^2}^2 \\
&\leq \epsilon \norm{\partial_{zx} \Theta}_{L^2}^2 + C_\epsilon \norm{v}_{L^2} \norm{\partial_z v}_{L^2} \norm{\partial_z \Theta}_{L^2} \left( \norm{\partial_z \Theta}_{L^2} + \norm{\partial_{xz} \Theta}_{L^2} \right) \\
&\leq \epsilon \norm{\partial_{zx} \Theta}_{L^2}^2 + C_\epsilon \norm{v}_{L^2} \norm{\partial_z v}_{L^2} \norm{\partial_z \Theta}_{L^2}^2 \left( 1 + \norm{v}_{L^2} \norm{\partial_z v}_{L^2} \right),
\end{aligned}
\end{align*}
where we have used
\begin{align*}
&\begin{aligned}
&\norm{v \partial_z \Theta}_{L^2}^2 \leq \norm{\partial_z \Theta}_{L_x^\infty L_z^2}^2 \norm{v}_{L_x^2 L_z^\infty}^2, \\
&\norm{\partial_z \Theta}_{L_x^\infty L_z^2}^2 \leq C \norm{\partial_z \Theta}_{L^2} \left( \norm{\partial_z \Theta}_{L^2} + \norm{\partial_{xz} \Theta}_{L^2} \right), \\
&\norm{v}_{L_x^2 L_z^\infty}^2 \leq 2 \norm{v}_{L^2} \norm{\partial_z v}_{L^2}.
\end{aligned}
\end{align*}

Combining the estimates for $I_1$--$I_3$, we obtain the differential inequality
\begin{align}\label{thetazt}
\begin{aligned}
\frac{d}{dt} \norm{\partial_z \Theta}_{L^2}^2 
+ \kappa_a \int_D |\partial_{zz} \Theta|^2 dx dz 
+ \int_D |\partial_{xz} \Theta|^2 dx dz \leq g(t),
\end{aligned}
\end{align}
where the function $g(t)$ is defined by
\begin{align*}
\begin{aligned}
g(t) := & C \norm{v}_{L^2}^2 
+ C \norm{v}_{L^2}^2 \norm{\partial_x \Theta}_{L^2}^2 \left( \norm{v}_{L^2} + \norm{\partial_x v}_{L^2} \right)^2 \\
& + C \norm{v}_{L^2} \norm{\partial_z v}_{L^2} \norm{\partial_z \Theta}_{L^2}^2 \left( 1 + \norm{v}_{L^2} \norm{\partial_z v}_{L^2} \right) \in L^1(0,\infty).
\end{aligned}
\end{align*}
Integrating \eqref{thetazt} with respect to time from $0$ to $t$ gives
\[
\norm{\partial_z \Theta(t)}_{L^2}^2 
+ \int_0^t \left( \kappa_a \norm{\partial_{zz} \Theta(t')}_{L^2}^2 + \norm{\partial_{zx} \Theta(t')}_{L^2}^2 \right) dt' 
\leq \int_0^t g(t') dt' + \norm{\partial_z \Theta(0)}_{L^2}^2,
\]
which establishes \eqref{infty-thetaz-1}. There exists a constant $\lambda_1 > 0$ such that $g(t)$ can be expressed as
\begin{align*}
\begin{aligned}
g(t) &= e^{-\lambda_2 t} D_1(t), \quad \lambda_2 > 0, \\
D_1(t) &= C \norm{v}_{L^2} 
+ C \norm{v}_{L^2} \norm{\partial_x \Theta}_{L^2}^2 \left( \norm{v}_{L^2} + \norm{\partial_x v}_{L^2} \right)^2 \\
&\quad + C \norm{\partial_z v}_{L^2} \norm{\partial_z \Theta}_{L^2}^2 \left( 1 + \norm{v}_{L^2} \norm{\partial_z v}_{L^2} \right) \in L^1(0,\infty).
\end{aligned}
\end{align*}
This decomposition, together with Lemma \ref{general-inequality1}, yields the exponential decay estimate \eqref{decaythetaz-1}.

\end{proof}

\begin{lemma}\label{lemma-grad-1-5}
Let $T_1 > T_0$ and consider any smooth solution $(v, \Theta)$ to the system \eqref{ondimensional-equations-2} with initial data $(v_0, \Theta_0) \in H^1(D)$. Then the horizontal derivatives satisfy the regularity estimate
\begin{align}\label{infty-1}
(\partial_x v, \partial_x \Theta) \in L^\infty\big( (0,\infty); L^2(D) \big) \cap L^2\big( (0,\infty); H^1(D) \big).
\end{align}
Moreover, there exists a constant $\lambda > 0$ such that
\begin{align}\label{dcayvxthetax}
\norm{\partial_x v(t)}_{L^2(D)}^2 + \norm{\partial_x \Theta(t)}_{L^2(D)}^2 \leq C C_0 e^{-2\lambda t} \norm{(v_0, \Theta_0)}_{H^1(D)}^2, \quad \forall t \geq 0,
\end{align}
where the constant $C_0$ depends on $\norm{v_0}_{H^1(D)}$ and $\norm{\Theta_0}_{L^2(D)}$.
\end{lemma}
\begin{proof}
Setting $T_1 > T_0$ in \eqref{ondimensional-equations-2}, multiplying \eqref{ondimensional-equations-2}$_2$ by $\partial_{xx} \Theta$ and integrating by parts, we obtain
\begin{align}\label{estimate-thetax-vx}
\begin{aligned}
&\frac{d}{dt} \norm{\partial_x \Theta}_{L^2}^2
+ 2 \int_D |\partial_{xx} \Theta|^2 dx dz 
+ 2\kappa_a \int_D |\partial_{xz} \Theta|^2 dx dz \\
&= 2\text{R} \int_D \partial_x \Theta \left( \int_0^z \partial_{xx} v(t,x,\xi) d\xi \right) dx dz \\
&\quad + 2 \int_D \left( v \partial_x \Theta - \left( \int_0^z \partial_x v(t,x,\xi) d\xi \right) \partial_z \Theta \right) \partial_{xx} \Theta dx dz.
\end{aligned}
\end{align}
Combining \eqref{estimatevx-infinity} and \eqref{estimate-thetax-vx}, we derive
\begin{align}
\begin{aligned}
&\frac{d}{dt} \left( \norm{\partial_x v}_{L^2}^2 + \norm{\partial_x \Theta}_{L^2}^2 \right)
+ 2\text{Pr}_x \int_D |\partial_{xx} v|^2 dx dz 
+ 2\text{Pr}_z \int_D |\partial_{xz} v|^2 dx dz \\
&+ 2 \int_D |\partial_{xx} \Theta|^2 dx dz 
+ 2\kappa_a \int_D |\partial_{xz} \Theta|^2 dx dz \\
&= 2 \int_D \left( v \partial_x v - \left( \int_0^z \partial_x v(t,x,\xi) d\xi \right) \partial_z v \right) \partial_{xx} v dx dz \\
&\quad + 2 \int_D \left( v \partial_x \Theta - \left( \int_0^z \partial_x v(t,x,\xi) d\xi \right) \partial_z \Theta \right) \partial_{xx} \Theta dx dz \\
&=: H_1 + H_2 + S_1 + S_2.
\end{aligned}
\end{align}

For $H_1 + H_2$, it follows from \eqref{estimateh1} and \eqref{estimateh2} that
\begin{align*}
\begin{aligned}
H_1 + H_2 \leq & \epsilon \norm{\partial_{xx} v}_{L^2}^2 
+ C_\epsilon \norm{v}_{L^2}^2 \norm{\partial_z v}_{L^2}^2 \norm{\partial_x v}_{L^2}^2 \\
& + C_\epsilon \norm{\partial_z v}_{L^2}^2 \norm{\partial_{zz} v}_{L^2}^2 \norm{\partial_x v}_{L^2}^2.
\end{aligned}
\end{align*}

For $S_1$, applying the Cauchy–Schwarz inequality, we obtain
\begin{align*}
\begin{aligned}
S_1 &= 2 \int_D (v \partial_x \Theta) \partial_{xx} \Theta dx dz \leq 2 \norm{\partial_{xx} \Theta}_{L^2} \norm{v \partial_x \Theta}_{L^2}.
\end{aligned}
\end{align*}
To control the nonlinear term, Lemma \ref{mixed-norm}-
 Lemma \ref{mixed-norm-1} means that the following bounds hold:
\begin{align*}
&\begin{aligned}
&\norm{v \partial_x \Theta}_{L^2}^2 \leq \norm{\partial_x \Theta}_{L_x^\infty L_z^2}^2 \norm{v}_{L_x^2 L_z^\infty}^2, \\
&\norm{\partial_x \Theta}_{L_x^\infty L_z^2}^2 \leq 2 \int_D |\partial_x \Theta  \partial_{xx} \Theta| dx dz \leq 2 \norm{\partial_x \Theta}_{L^2} \norm{\partial_{xx} \Theta}_{L^2}, \\
&\norm{v}_{L_x^2 L_z^\infty}^2 \leq 2 \norm{\partial_z v}_{L^2} \norm{v}_{L^2}.
\end{aligned}
\end{align*}
Substituting these into the previous estimate  and 
applying the Young’s inequality yields
\[
S_1\leq \epsilon \norm{\partial_{xx} \Theta}_{L^2}^2 + C_\epsilon \norm{v}_{L^2}^2 \norm{\partial_z v}_{L^2}^2 \norm{\partial_x \Theta}_{L^2}^2.
\]

Regarding $S_2$, applying the Cauchy–Schwarz inequality, we obtain
\begin{align}\label{S1-S2}
\begin{aligned}
S_2 &= -2 \int_D \left( \int_0^z \partial_x v(t,x,\xi) d\xi \right) \partial_z \Theta  \partial_{xx} \Theta dx dz \\
&\leq 2 \norm{\partial_{xx} \Theta}_{L^2} \norm{ \left( \int_0^z \partial_x v(t,x,\xi) d\xi \right) \partial_z \Theta }_{L^2}.
\end{aligned}
\end{align}
To control the nonlinear term, we note that
\begin{align*}
&\begin{aligned}
&\norm{ \left( \int_0^z \partial_x v(t,x,\xi) d\xi \right) \partial_z \Theta }_{L^2}^2 \\
&\quad \leq \int_0^1 \left( \left( \int_0^z \norm{\partial_x v(\cdot,\xi)}_{L_x^\infty} d\xi \right)^2 \int_0^\alpha |\partial_z \Theta|^2 dx \right) dz \\
&\quad \leq C \norm{\partial_z \Theta}_{L_z^\infty L_x^2}^2 \norm{\partial_x v}_{L_z^2 L_x^\infty}^2.
\end{aligned}
\end{align*}
Furthermore, Lemma \ref{mixed-norm}-
 Lemma \ref{mixed-norm-1} means that the following bounds hold:
\[
\begin{aligned}
&\norm{\partial_z \Theta}_{L_z^\infty L_x^2}^2 \leq C \norm{\partial_z \Theta}_{L^2} \left( \norm{\partial_z \Theta}_{L^2} + \norm{\partial_{zz} \Theta}_{L^2} \right), \\
&\norm{\partial_x v}_{L_z^2 L_x^\infty}^2 \leq 2 \norm{\partial_x v}_{L^2} \norm{\partial_{xx} v}_{L^2}.
\end{aligned}
\]
Substituting these into the previous estimate \eqref{S1-S2} and
applying Young’s inequality yields
\[
S_2\leq\epsilon \norm{\partial_{xx} \Theta}_{L^2}^2 + \epsilon \norm{\partial_{xx} v}_{L^2}^2 + C_\epsilon \norm{\partial_z \Theta}_{L^2}^2 \norm{\partial_x v}_{L^2}^2 \norm{\partial_{zz} \Theta}_{L^2}^2.
\]

A suitable choice of a sufficiently small $\epsilon$ leads to the differential inequality
\begin{align}\label{thetazt-new} 
\begin{aligned}
&\frac{d}{dt} \left( \norm{\partial_x v}_{L^2}^2 + \norm{\partial_x \Theta}_{L^2}^2 \right)
+ \text{Pr}_x \int_D |\partial_{xx} v|^2 dx dz 
+ \text{Pr}_z \int_D |\partial_{xz} v|^2 dx dz \\
&+ \int_D |\partial_{xx} \Theta|^2 dx dz 
+ \kappa_a \int_D |\partial_{xz} \Theta|^2 dx dz \leq g(t),
\end{aligned}
\end{align}
where there exists a constant $\lambda_2 > 0$ such that $g(t)$ is given by
\begin{align*}
\begin{aligned}
g(t) &= e^{-\lambda_2 t} D_1(t), \\
D_1(t) &= C e^{\lambda_2 t} \norm{\partial_z v}_{L^2}^2 \left( \norm{v}_{L^2}^2 \norm{\partial_x v}_{L^2}^2 + \norm{\partial_{zz} v}_{L^2}^2 \norm{\partial_x v}_{L^2}^2 \right) \\
&\quad + C e^{\lambda_2 t} \norm{v}_{L^2}^2 \norm{\partial_z v}_{L^2}^2 \norm{\partial_x \Theta}_{L^2}^2 \\
&\quad + C e^{\lambda_2 t} \norm{\partial_z \Theta}_{L^2}^2 \norm{\partial_x v}_{L^2}^2 \left( \norm{\partial_{zz} \Theta}_{L^2}^2 + \norm{\partial_z \Theta}_{L^2}^2 \right) \in L^1([0,+\infty)).
\end{aligned}
\end{align*}
This decomposition, together with Lemma \ref{general-inequality1}, yields the exponential decay estimate \eqref{dcayvxthetax}.
Integrating \eqref{thetazt-new} with respect to time from $0$ to $t$ gives
\eqref{infty-1}.
\end{proof}
  
\subsubsection{$H^2$-Estimates and Global Stability}

\begin{lemma}\label{lemma-grad-1-6}
Let $T_1 > T_0$ and consider any smooth solution $(v, \Theta)$ to the system \eqref{ondimensional-equations-2} with initial data $(v_0, \Theta_0) \in H^2(D)$. Then the following regularity estimates hold:
\begin{align}\label{vh2-infty-16}
\partial_z(\nabla v, \nabla \Theta) \in L^\infty\big( (0,\infty); L^2(D) \big), \quad
\Delta(\partial_z v, \partial_z \Theta) \in L^2\big( (0,\infty); L^2(D) \big).
\end{align}
Moreover, there exists a constant $\lambda > 0$ such that
\begin{align}\label{decy-vxz-vzz-16}
\norm{\partial_z (\nabla v, \nabla \Theta)(t)}_{L^2(D)}^2 \leq C C_0 e^{-2\lambda t} \norm{(v_0, \Theta_0)}_{H^2(D)}^2, \quad \forall t \geq 0,
\end{align}
where the constant $C_0$ depends on $\norm{v_0}_{H^1(D)}$ and $\norm{\Theta_0}_{L^2(D)}$.
\end{lemma}
\begin{proof}
Multiplying \eqref{three-proof--8} successively by $\partial_{xx}\sigma$
and $\partial_{zz}\sigma$, then integrating by parts and applying the Cauchy–Schwarz inequality and Young’s inequality, we obtain
 \begin{align*}
\begin{aligned}
&\frac{d}{dt}
\int_{D}\abs{\nabla \sigma}^2
\,dx\,dz
+2\text{Pr}_x\norm{\partial_{xx}\sigma}_{L^2}^2
+2\text{Pr}_z \norm{\partial_{zz}\sigma}_{L^2}^2
\\&\leq \epsilon 
\left(\norm{\partial_{xx}\sigma}_{L^2}^2
+
\norm{\partial_{zz}\sigma}_{L^2}^2\right)
+2\text{R}(\partial_x\Theta,\partial_{xx}\sigma+\partial_{zz}\sigma)
\\&\quad+C_{\epsilon}\norm{v \partial_x \sigma}_{L^2}^2+C_{\epsilon}\norm{\left(\int_0^z\partial_xv(t,x,\xi)\,d\xi \right)\partial_z \sigma}
_{L^2}^2.
\end{aligned}
\end{align*}

To control the nonlinear terms, we apply Lemma \ref{mixed-norm}-
 Lemma \ref{mixed-norm-1}  to get
    \begin{align*}
&\begin{aligned}
&\norm{v\partial_x\sigma}_{L^2}^2
\leq
\norm{\partial_x\sigma}_{L^{\infty}_{x}L^2_{z}}^2
\norm{v}_{L^2_{x}L^{\infty}_{z}}^2,\\
&\norm{\partial_x\sigma}
_{L^{\infty}_{x}L^2_{z}}^2
\leq 2\int_{D}\abs{\partial_x\sigma
\partial_{xx}\sigma}\,dx\,dz
\leq 2\norm{
\partial_x\sigma}_{L^2}
\norm{
\partial_{xx}\sigma}_{L^2},\\
&\norm{v}_{L^2_{x}L^{\infty}_{z}}^2\leq 
2\norm{
\partial_zv}_{L^2}
\norm{v}_{L^2},
\end{aligned}
\end{align*}
and
    \begin{align*}
&\begin{aligned}
&\begin{aligned}
& \norm{
\left(\int_0^z\partial_xv(t,x,\xi)\,d\xi \right)\partial_z\sigma}_{L^2}^2
\\& \leq \int_{0}^{\alpha}\left(
\norm{\partial_z\sigma}_{L_z^{\infty}}^2
\int_{0}^1
\left(\int_0^z\partial_xv(\xi)\,d\xi\right)^2\,dz\right)\,dx\\
&\leq  \norm{\partial_z\sigma}^2_{L_x^{2}L_z^{\infty}}
\norm{\partial_xv}_{L_x^{\infty}L_z^{2}}^2,
\end{aligned}\\
&\norm{\partial_z\sigma}
_{L^{2}_{x}L^{\infty}_{z}}^2
\leq 2\int_{D}\abs{\partial_z\sigma
\partial_{zz}\sigma}\,dx\,dz
\leq 2\norm{
\partial_z\sigma}_{L^2}
\norm{
\partial_{zz}\sigma}_{L^2},\\
&\norm{\partial_xv}_{L_x^{\infty}L_z^{2}}^2\leq 2\norm{
\partial_xv}_{L^2}\norm{
\partial_{xx}v}_{L^2}.
\end{aligned}
\end{align*}
 Substituting these into the previous estimate yields
 \begin{align}\label{nablasigma}
\begin{aligned}
&\frac{d}{dt}
\int_{D}\abs{\nabla \sigma}^2
\,dx\,dz
+2\text{Pr}_x\norm{\partial_{xx}\sigma}_{L^2}^2
+2\text{Pr}_z \norm{\partial_{zz}\sigma}_{L^2}^2\\
&\leq \epsilon 
\left(\norm{\partial_{xx}\sigma}_{L^2}^2
+
\norm{\partial_{zz}\sigma}_{L^2}^2\right)
+C_{\epsilon}
\norm{
\partial_zv}_{L^2}^2
\norm{v}_{L^2}^2\norm{
\partial_x\sigma}_{L^2}^2
\\
&\quad+C_{\epsilon}
\norm{
\partial_xv}_{L^2}^2\norm{
\partial_{xx}v}_{L^2}^2
\norm{
\partial_z\sigma}_{L^2}^2+2\text{R}(\partial_x\Theta,\partial_{xx}\sigma+\partial_{zz}\sigma).
\end{aligned}
\end{align}

Let us denot $\eta:=\partial_z\Theta$, then $\eta$ solves
\begin{align}\label{thetaz-eta} 
\begin{aligned}
&\partial_t\eta+v\partial_x \eta-\left(\int_0^z\partial_xv(t,x,\xi)\,d\xi \right)\partial_z\eta= \partial_{xx}\eta+
\kappa_{a}\partial_{zz} \eta\\
&+\text{R}\partial_xv
-\sigma \partial_x \Theta+\partial_xv \eta.
\end{aligned}
\end{align}
Multiplying the equation \eqref{thetaz-eta} successively by $\partial_{xx}\eta$
and $\partial_{zz}\eta$, then integrating by parts, we obtain
 \begin{align*}
\begin{aligned}
&\frac{d}{dt}
\int_{D}\abs{\nabla \eta}^2
\,dx\,dz
+2\norm{\partial_{xx}\eta}_{L^2}^2
+2\kappa_a \norm{\partial_{zz}\eta}_{L^2}^2
+2(1+\kappa_a) \norm{\partial_{xz}\eta}_{L^2}^2\\
&\leq \epsilon 
\left(\norm{\partial_{xx}\eta}_{L^2}^2
+
\norm{\partial_{zz}\eta}_{L^2}^2\right)
-2\text{R}(\partial_xv,\partial_{xx}\eta+\partial_{zz}\eta)
\\&\quad+2\norm{v \partial_x \eta}_{L^2}\left(\norm{\partial_{xx}\eta}_{L^2}
+
\norm{\partial_{zz}\eta}_{L^2}\right)
\\&
+2\norm{\left(\int_0^z\partial_xv(t,x,\xi)\,d\xi \right)\partial_z \eta}
_{L^2}\left(\norm{\partial_{xx}\eta}_{L^2}
+
\norm{\partial_{zz}\eta}_{L^2}\right)\\
&\quad+2\norm{\sigma \partial_x \Theta}_{L^2}\left(\norm{\partial_{xx}\eta}_{L^2}
+
\norm{\partial_{zz}\eta}_{L^2}\right)\\
&\quad+2\norm{\eta \partial_x v}_{L^2}\left(\norm{\partial_{xx}\eta}_{L^2}
+
\norm{\partial_{zz}\eta}_{L^2}\right).
\end{aligned}
\end{align*}
To establish these bounds, we employ the following inequalities. First, we observe that
    \begin{align*}
&\begin{aligned}
&\norm{v\partial_x\eta}_{L^2}^2
\leq
\norm{\partial_x\eta}_{L^{\infty}_{x}L^2_{z}}^2
\norm{v}_{L^2_{x}L^{\infty}_{z}}^2,\\
\end{aligned}\\
&\begin{aligned}
 \norm{
\left(\int_0^z\partial_xv(t,x,\xi)\,d\xi \right)\partial_z\eta}_{L^2}^2
&\leq \int_{0}^{\alpha}
\norm{\left(\int_0^z\partial_xv(\xi)\,d\xi\right)}_{L_z^{\infty}}^2
\left(
\int_{0}^1
\abs{\partial_z\eta}^2\,dz\right)\,dx\\
&\leq  \norm{\partial_z\eta}^2_{L_x^{\infty}L_z^{2}}
\norm{\partial_xv}_{L_x^{2}L_z^{\infty}}^2,
\end{aligned}\\
&\norm{\sigma \partial_x\Theta}_{L^2}^2
\leq
\norm{\partial_x\Theta}_{L^{\infty}_{x}L^2_{z}}^2
\norm{\sigma}_{L^2_{x}L^{\infty}_{z}}^2,\quad
\norm{\eta \partial_xv}_{L^2}^2
\leq
\norm{\partial_xv}_{L^{\infty}_{x}L^2_{z}}^2
\norm{\eta }_{L^2_{x}L^{\infty}_{z}}^2,
\end{align*}
and we further bound the mixed norms via
    \begin{align*}
&\begin{aligned}
&\norm{\partial_x\eta}
_{L^{\infty}_{x}L^2_{z}}^2
\leq 2\int_{D}\abs{\partial_x\eta
\partial_{xx}\eta}\,dx\,dz
\leq 2\norm{
\partial_x\eta}_{L^2}
\norm{
\partial_{xx}\eta}_{L^2},\\
&\norm{v}_{L^2_{x}L^{\infty}_{z}}^2\leq 
2\norm{
\partial_zv}_{L^2}
\norm{v}_{L^2},\\
&\norm{\partial_z\eta}
_{L^{\infty}_{x}L^{2}_{z}}^2
\leq C\norm{
\partial_z\eta}_{L^2}^2
+C\norm{
\partial_z\eta}_{L^2}
\norm{
\partial_{xz}\eta}_{L^2},\\
&\norm{\partial_xv}_{L_x^{2}L_z^{\infty}}^2\leq 2\norm{
\partial_xv}_{L^2}\norm{
\partial_{zx}v}_{L^2},
\end{aligned}\\
&\begin{aligned}
&\norm{\partial_x\Theta}
_{L^{\infty}_{x}L^2_{z}}^2
\leq 2\int_{D}\abs{\partial_x\Theta
\partial_{xx}\Theta}\,dx\,dz
\leq 2\norm{
\partial_x\Theta}_{L^2}
\norm{
\partial_{xx}\Theta}_{L^2},\\
&\norm{\sigma}_{L^2_{x}L^{\infty}_{z}}^2\leq 
2\norm{
\partial_z\sigma}_{L^2}
\norm{\sigma}_{L^2},
\end{aligned}\\
&\begin{aligned}
&\norm{\partial_xv}
_{L^{\infty}_{x}L^2_{z}}^2
\leq 2\int_{D}\abs{\partial_xv
\partial_{xx}v}\,dx\,dz
\leq 2\norm{
\partial_xv}_{L^2}
\norm{
\partial_{xx}v}_{L^2},\\
&\norm{\eta }_{L^2_{x}L^{\infty}_{z}}^2\leq 
2\norm{
\partial_z\eta }_{L^2}
\norm{\eta }_{L^2},
\end{aligned}
\end{align*}
where Lemma \ref{mixed-norm}-
 Lemma \ref{mixed-norm-1} have been used.
  Substituting these into the previous estimate yields
 \begin{align}\label{nablaeta}
\begin{aligned}
&\frac{d}{dt}
\int_{D}\abs{\nabla \eta}^2
\,dx\,dz
+2\norm{\partial_{xx}\eta}_{L^2}^2
+2\kappa_a \norm{\partial_{zz}\eta}_{L^2}^2
+2(1+\kappa_a) \norm{\partial_{xz}\eta}_{L^2}^2\\
&\leq 
 \epsilon 
\left(\norm{\partial_{xx}\eta}_{L^2}^2
+\norm{\partial_{xz}\eta}_{L^2}^2+
\norm{\partial_{zz}\eta}_{L^2}^2\right)
+C_{\epsilon}\norm{
\partial_x\eta}_{L^2}^2
\norm{
\partial_zv}_{L^2}^2
\norm{v}_{L^2}^2\\
&\quad +C_{\epsilon}\norm{
\partial_z\eta}_{L^2}^2
\norm{
\partial_xv}_{L^2}\norm{
\partial_{zx}v}_{L^2}+
C_{\epsilon}
\norm{
\partial_z\eta}_{L^2}^2
\norm{
\partial_xv}_{L^2}^2\norm{
\partial_{zx}v}_{L^2}^2\\
&\quad+C_{\epsilon}
\norm{
\partial_z\sigma}_{L^2}
\norm{\sigma}_{L^2}
\norm{
\partial_x\Theta}_{L^2}
\norm{
\partial_{xx}\Theta}_{L^2}
\\&\quad+C_{\epsilon}\norm{
\partial_xv}_{L^2}
\norm{
\partial_{xx}v}_{L^2}
\norm{
\partial_z\eta }_{L^2}
\norm{\eta }_{L^2}\\
&\quad-2\text{R}(\partial_xv,\partial_{xx}\eta+\partial_{zz}\eta),
\end{aligned}
\end{align}

Note that the cross terms cancel identically:
\[
2\text{R} (\partial_x \Theta, \partial_{xx} \sigma + \partial_{zz} \sigma)
- 2\text{R} (\partial_x v, \partial_{xx} \eta + \partial_{zz} \eta) = 0.
\]
Using this cancellation and adding \eqref{nablasigma} and \eqref{nablaeta}, then choosing $\epsilon$ sufficiently small, we obtain
\begin{align}\label{h2dtestimate-all-1}
\begin{aligned}
&\frac{d}{dt}\left(
\int_{D}\abs{\nabla \sigma}^2
\,dx\,dz
+\int_{D}\abs{\nabla \eta}^2
\,dx\,dz
\right)
\\&+\lambda_1\left(\norm{\partial_{xx} \sigma}_{L^2}^2
+ \norm{\partial_{zx} \sigma}_{L^2}^2
+\norm{\partial_{xx}\eta}_{L^2}^2
+ \norm{\partial_{zx}\eta}_{L^2}^2
\right)
\\&\leq C\left(
\norm{
\partial_zv}_{L^2}^2
\norm{v}_{L^2}^2\norm{
\partial_x\sigma}_{L^2}^2
+\norm{
\partial_xv}_{L^2}^2\norm{
\partial_{xx}v}_{L^2}^2
\norm{
\partial_z\sigma}_{L^2}^2\right)\\
&\quad+C\norm{
\partial_x\eta}_{L^2}^2
\norm{
\partial_zv}_{L^2}^2
\norm{v}_{L^2}^2 +C\norm{
\partial_z\eta}_{L^2}^2
\norm{
\partial_xv}_{L^2}\norm{
\partial_{zx}v}_{L^2}\\&\quad+
C
\norm{
\partial_z\eta}_{L^2}^2
\norm{
\partial_xv}_{L^2}^2\norm{
\partial_{zx}v}_{L^2}^2\\
&\quad+C
\norm{
\partial_z\sigma}_{L^2}
\norm{\sigma}_{L^2}
\norm{
\partial_x\Theta}_{L^2}
\norm{
\partial_{xx}\Theta}_{L^2}
\\&\quad+C\norm{
\partial_xv}_{L^2}
\norm{
\partial_{xx}v}_{L^2}
\norm{
\partial_z\eta }_{L^2}
\norm{\eta }_{L^2}.
\end{aligned}
\end{align}
This, together with Gronwall’s inequality, yields 
\[
\partial_z( \nabla v,\nabla \Theta) \in L^{\infty}\left((0,\infty);L^{2}(D)\right).
\]

Let us denote $f:=\norm{(\sigma,\eta)}_{L^2}^2+\norm{(\nabla\sigma,\nabla\eta)}_{L^2}$, we then use \eqref{h2dtestimate-all-1},  \eqref{three-proof--9-1} and \eqref{thetazt} to get that there exists a $\lambda_1>0$ such that
$f$ satisfies
\begin{align}\label{fff}
\frac{df}{dt}+\lambda_1f\leq g(t),
\end{align}
where $g(t)=D_1(t)e^{-\lambda_2t}$ for some $\lambda_2>0$, and 
 is given by
\[
\begin{aligned}
D_1(t)&:=e^{-\lambda_2t}\bigg(C\left(
\norm{
\partial_zv}_{L^2}^2
\norm{v}_{L^2}^2\norm{
\partial_x\sigma}_{L^2}^2
+\norm{
\partial_xv}_{L^2}^2\norm{
\partial_{xx}v}_{L^2}^2
\norm{
\partial_z\sigma}_{L^2}^2\right)\\
&\quad+C\norm{
\partial_x\eta}_{L^2}^2
\norm{
\partial_zv}_{L^2}^2
\norm{v}_{L^2}^2 +C\norm{
\partial_z\eta}_{L^2}^2
\norm{
\partial_xv}_{L^2}\norm{
\partial_{zx}v}_{L^2}\\&\quad+
C
\norm{
\partial_z\eta}_{L^2}^2
\norm{
\partial_xv}_{L^2}^2\norm{
\partial_{zx}v}_{L^2}^2\\
&\quad+C
\norm{
\partial_z\sigma}_{L^2}
\norm{\sigma}_{L^2}
\norm{
\partial_x\Theta}_{L^2}
\norm{
\partial_{xx}\Theta}_{L^2}
\\&\quad+C\norm{
\partial_xv}_{L^2}
\norm{
\partial_{xx}v}_{L^2}
\norm{
\partial_z\eta }_{L^2}
\norm{\eta }_{L^2}+
C\norm{\Theta}_{L^2}^2\\
&\quad+C \norm{v}_{L^2}^2 
+ C \norm{v}_{L^2}^2 \norm{\partial_x \Theta}_{L^2}^2 \left( \norm{v}_{L^2} + \norm{\partial_x v}_{L^2} \right)^2 \\
& \quad+ C \norm{v}_{L^2} \norm{\partial_z v}_{L^2} \norm{\partial_z \Theta}_{L^2}^2 \left( 1 + \norm{v}_{L^2} \norm{\partial_z v}_{L^2} \right)\bigg)
\in L^1([0,+\infty)).
\end{aligned}
\]

Based on the estimates established in Lemmas~\ref{v-l2-l2}--\ref{lemma-grad-1-5}, integrating \eqref{h2dtestimate-all-1} with respect to time yields the regularity result \eqref{vh2-infty-16}. Furthermore, combining \eqref{fff}
and \eqref{vh2-infty-16} and employing arguments similar to those used in the proof of \eqref{dcayvxthetax}, we establish the decay estimate \eqref{decy-vxz-vzz-16}.

\end{proof}

\begin{lemma}\label{lemma-grad-1-7}
Let $T_1 > T_0$ and consider any smooth solution $(v, \Theta)$ to the system \eqref{ondimensional-equations-2} with initial data $(v_0, \Theta_0) \in H^2(D)$. Then the following regularity estimates hold:
\begin{align}\label{infty-vxx-delta-17-1}
(\partial_{xx} v, \partial_{xx} \Theta) \in L^\infty\big( (0,\infty); L^2(D) \big), \quad
\Delta(\partial_x v, \partial_x \Theta) \in L^2\big( (0,\infty); L^2(D) \big).
\end{align}
Moreover, there exists a constant $\lambda > 0$ such that
\begin{align}\label{decy-vxz-vzz-17-1-1}
\norm{(\partial_{xx} v, \partial_{xx} \Theta)(t)}_{L^2(D)}^2 \leq C C_0 e^{-2\lambda t} \norm{(v_0, \Theta_0)}_{H^2(D)}^2, \quad \forall t \geq 0,
\end{align}
where the constant $C_0$ depends on $\norm{v_0}_{H^1(D)}$ and $\norm{\Theta_0}_{L^2(D)}$.
\end{lemma}
\begin{proof}
For $T_1 > T_0$, the pair $(\omega, \theta) := (\partial_x v, \partial_x \Theta)$ satisfies the system
   \begin{align}\label{three-proof--8-1-1}
\begin{cases}
\begin{aligned}
&\partial_t\omega=\text{Pr}_x\partial_{xx}\omega+
\text{Pr}_z\partial_{zz}\omega-\text{R}\left(
\int_0^z\partial_{x} \theta(t,x,\xi)\,d\xi\right)\\&-v \partial_x \omega+\left(\int_0^z\omega(t,x,\xi)\,d\xi \right)\partial_z\omega-\omega^2+\left(\int_0^z\partial_{x}\omega(t,x,\xi)\,d\xi \right)\partial_zv,
\end{aligned}\\
\begin{aligned}
&\partial_t\theta+v\partial_x \theta-\left(\int_0^z\partial_xv(t,x,\xi)\,d\xi \right)\partial_z \theta= \partial_{xx}\theta+
\kappa_{a}\partial_{zz}\theta\\
&+\text{R}\left(\int_0^z\partial_x\omega(t,x,\xi)\,d\xi \right)
-\omega \theta+\left(\int_0^z\partial_x\omega(t,x,\xi)\,d\xi \right)\partial_z \Theta,
\end{aligned}\\
\omega_z |_{z=0}=0,\quad \omega_z |_{z=1}=0,\\
\omega_z|_{t=0}=\omega_0.
\end{cases}
\end{align}

Multiplying the first equation in \eqref{three-proof--8-1-1} by $\partial_{xx} \omega$ and integrating by parts over the domain yields
 \begin{align}\label{nablaomega}
\begin{aligned}
&\frac{d}{dt}
\int_{D}\abs{\partial_x\omega}^2
\,dx\,dz
+2\text{Pr}_x\norm{\partial_{xx}\omega}_{L^2}^2
+2\text{Pr}_z \norm{\partial_{zx}\omega}_{L^2}^2\\&
\leq \epsilon 
\left(\norm{\partial_{xx}\omega}_{L^2}^2
+
\norm{\partial_{zz}\omega}_{L^2}^2\right)
+2\text{R}\left(
\int_0^z\partial_{x} \theta(t,x,\xi)\,d\xi,\partial_{xx}\omega\right)
\\
&\quad+K_1+K_2+K_3+K_4.
\end{aligned}
\end{align}
We now estimate the terms $K_j$ for $j = 1, \cdots, 4$. 
We observe that
\begin{align*}
\begin{aligned}
&\norm{v \partial_x \omega}_{L^2}^2 \leq \int_0^\alpha \left( \norm{v}_{L_z^\infty}^2 \int_0^1 (\partial_x \omega)^2 dz \right) dx 
\leq \norm{\partial_x \omega}_{L_x^\infty L_z^2}^2 \norm{v}_{L_x^2 L_z^\infty}^2, \\
&\norm{\partial_x \omega}_{L_x^\infty L_z^2}^2 \leq 2 \int_0^1 \left( \int_0^\alpha |\partial_x \omega  \partial_{xx} \omega| dx \right) dz 
\leq 2 \norm{\partial_x \omega}_{L^2} \norm{\partial_{xx} \omega}_{L^2}, \\
&\norm{v}_{L_x^2 L_z^\infty}^2 \leq 2 \int_0^\alpha \left( \int_0^1 |\partial_z v  \partial_{zz} v| dz \right) dx 
\leq 2 \norm{\partial_z v}_{L^2} \norm{v}_{L^2}.
\end{aligned}
\end{align*}
This, combined with Young’s inequality, yields the estimate for $K_1$
as follows
\begin{align}\label{kkkk-1}
\begin{aligned}
K_1 &:= C_\epsilon \norm{v \partial_x \omega}_{L^2}^2 
\leq \epsilon \norm{\partial_{xx} \omega}_{L^2}^2 
+ C_\epsilon \norm{\partial_z v}_{L^2}^2 \norm{v}_{L^2}^2 \norm{\partial_x \omega}_{L^2}^2.
\end{aligned}
\end{align}
Note that
\begin{align*}
\begin{aligned}
&\norm{ \left( \int_0^z \partial_x v(t,x,\xi) d\xi \right) \partial_z \omega }_{L^2}^2 \\
&\quad \leq \int_0^1 \int_0^\alpha \left( \norm{\partial_z \omega}_{L_x^\infty}^2 \left( \int_0^z \partial_x v(\xi) d\xi \right)^2 \right) dx dz \\
&\quad \leq C \norm{\partial_z \omega}_{L_z^2 L_x^\infty}^2 \norm{\partial_x v}_{L_x^2 L_z^\infty}^2, \\
&\norm{\partial_z \omega}_{L_z^2 L_x^\infty}^2 \leq 2 \int_D |\partial_z \omega  \partial_{zx} \omega| dx dz 
\leq 2 \norm{\partial_z \omega}_{L^2} \norm{\partial_{xz} \omega}_{L^2}, \\
&\norm{\partial_x v}_{L_x^2 L_z^\infty}^2 \leq 2 \norm{\partial_x v}_{L^2} \norm{\partial_{xz} v}_{L^2}.
\end{aligned}
\end{align*}
This, combined with Young’s inequality, yields the estimate for $K_2$
as follows
\begin{align}\label{kkkk-2}
\begin{aligned}
K_2 &:= C_\epsilon \norm{ \left( \int_0^z \partial_x v(t,x,\xi) d\xi \right) \partial_z \omega }_{L^2}^2 
\\&\leq \epsilon \norm{\partial_{xz} \omega}_{L^2}^2 
+ C_\epsilon \norm{\partial_z \omega}_{L^2}^2 \norm{\partial_x v}_{L^2}^2 \norm{\partial_{xz} v}_{L^2}^2.
\end{aligned}
\end{align}
For $K_3$, one has 
\begin{align}\label{kkkk-3}
\begin{aligned}
K_3 &:= C_\epsilon \norm{\omega^2}_{L^2}^2 
\leq C \norm{\partial_x v}_{L^2}^2 \norm{\partial_{xx} v}_{L^2} \norm{\partial_z \omega}_{L^2}.
\end{aligned}
\end{align}

Making use of the following estimates
\begin{align*}
\begin{aligned}
&\norm{ \left( \int_0^z \partial_{xx} v(t,x,\xi) d\xi \right) \partial_z v }_{L^2}^2 \\
&\quad \leq \int_0^1 \int_0^\alpha \left( \norm{\partial_z \omega}_{L_x^\infty}^2 \left( \int_0^z \partial_{xx} v(\xi) d\xi \right)^2 \right) dx dz \\
&\quad \leq C \norm{\partial_z \omega}_{L_z^2 L_x^\infty}^2 \norm{\partial_{xx} v}_{L_x^2 L_z^\infty}^2, \\
&\norm{\partial_z \omega}_{L_z^2 L_x^\infty}^2 \leq 2 \int_D |\partial_z \omega  \partial_{zx} \omega| dx dz 
\leq 2 \norm{\partial_z \omega}_{L^2} \norm{\partial_{xz} \omega}_{L^2}, \\
&\norm{\partial_x \omega}_{L_x^2 L_z^\infty}^2 \leq 2 \norm{\partial_x \omega}_{L^2} \norm{\partial_{xz} \omega}_{L^2}, 
\end{aligned}
\end{align*}
 and applying Young’s inequality,  the estimate for $K_2$ is as follows
 \begin{align}\label{kkkk-4}
\begin{aligned}
K_4 &:= C_\epsilon \norm{ \left( \int_0^z \partial_{xx} v(t,x,\xi) d\xi \right) \partial_z v }_{L^2}^2 
\\&\leq \epsilon \norm{\partial_{xz} \omega}_{L^2}^2 
+ C_\epsilon \norm{\partial_z \omega}_{L^2}^2 \norm{\partial_x \omega}_{L^2}^2 \norm{\partial_{xz} \omega}_{L^2}^2.
\end{aligned}
\end{align}

Multiplying the second equation in the system \eqref{three-proof--8-1-1} by $\partial_{xx} \theta$, integrating by parts and applying Young’s inequality, we obtain
\begin{align}\label{nablatheta}
\begin{aligned}
&\frac{d}{dt} \int_D |\partial_x \theta|^2 dx dz 
+ 2\norm{\partial_{xx} \theta}_{L^2}^2 
+ 2\kappa_a \norm{\partial_{zx} \theta}_{L^2}^2 \\
&\leq \epsilon \left( \norm{\partial_{xx} \theta}_{L^2}^2 + \norm{\partial_{zz} \theta}_{L^2}^2 \right) 
- 2\text{R} \left( \int_0^z \partial_x \omega(t,x,\xi) d\xi, \partial_{xx} \theta \right) \\
&\quad + C_\epsilon \norm{v \partial_x \theta}_{L^2}^2 
 + C_\epsilon \norm{ \left( \int_0^z \partial_x v(t,x,\xi) d\xi \right) \partial_z \theta }_{L^2}^2 \\& + C_\epsilon \norm{\omega \theta}_{L^2}^2+ C_\epsilon \norm{ \left( \int_0^z \partial_x \omega(t,x,\xi) d\xi \right) \partial_z \Theta }_{L^2}^2 \\
&\leq \epsilon \left( \norm{\partial_{xx} \theta}_{L^2}^2 + \norm{\partial_{zz} \theta}_{L^2}^2 \right) 
- 2\text{R} \left( \int_0^z \partial_x \omega(t,x,\xi) d\xi, \partial_{xx} \theta \right) \\
&\quad + J_1 + J_2 + J_3 + J_4.
\end{aligned}
\end{align}
We now estimate the terms $J_j$ for $j = 1, \cdots, 4$. 
We observe that
\begin{align*}
\begin{aligned}
&\norm{v \partial_x \theta}_{L^2}^2 \leq \int_0^\alpha \left( \norm{v}_{L_z^\infty}^2 \int_0^1 (\partial_x \theta)^2 dz \right) dx 
\leq \norm{\partial_x \theta}_{L_x^\infty L_z^2}^2 \norm{v}_{L_x^2 L_z^\infty}^2, \\
&\norm{\partial_x \theta}_{L_x^\infty L_z^2}^2 \leq 2 \int_0^1 \left( \int_0^\alpha |\partial_x \theta  \partial_{xx} \theta| dx \right) dz 
\leq 2 \norm{\partial_x \theta}_{L^2} \norm{\partial_{xx} \theta}_{L^2}, \\
&\norm{v}_{L_x^2 L_z^\infty}^2 \leq 2 \int_0^\alpha \left( \int_0^1 |\partial_z v  \partial_{zz} v| dz \right) dx 
\leq 2 \norm{\partial_z v}_{L^2} \norm{v}_{L^2},
\end{aligned}
\end{align*}
This, combined with Young’s inequality, yields the estimate for $J_1$
as follows
\begin{align}\label{jjjj-1}
\begin{aligned}
J_1 &:= C_\epsilon \norm{v \partial_x \theta}_{L^2}^2 
\leq \epsilon \norm{\partial_{xx} \theta}_{L^2}^2 
+ C_\epsilon \norm{\partial_z v}_{L^2}^2 \norm{v}_{L^2}^2 \norm{\partial_x \theta}_{L^2}^2.
\end{aligned}
\end{align}

Making use of 
\[
\begin{aligned}
&\norm{ \left( \int_0^z \partial_x v(t,x,\xi) d\xi \right) \partial_z \theta }_{L^2}^2 \\
&\quad \leq \int_0^1 \int_0^\alpha \left( \norm{\partial_z \theta}_{L_x^\infty}^2 \left( \int_0^z \partial_x v(\xi) d\xi \right)^2 \right) dx dz \\
&\quad \leq C \norm{\partial_z \theta}_{L_z^2 L_x^\infty}^2 \norm{\partial_x v}_{L_x^2 L_z^\infty}^2, \\
&\norm{\partial_z \theta}_{L_z^2 L_x^\infty}^2 \leq 2 \int_D |\partial_z \theta  \partial_{zx} \theta| dx dz 
\leq 2 \norm{\partial_z \theta}_{L^2} \norm{\partial_{xz} \theta}_{L^2}, \\
&\norm{\partial_x v}_{L_x^2 L_z^\infty}^2 \leq 2 \norm{\partial_x v}_{L^2} \norm{\partial_{xz} v}_{L^2},
\end{aligned}
\]
 and applying Young’s inequality,  the estimate for $J_2$ is as follows
\begin{align}\label{jjjj-2}
\begin{aligned}
J_2 &:= C_\epsilon \norm{ \left( \int_0^z \partial_x v(t,x,\xi) d\xi \right) \partial_z \theta }_{L^2}^2 
\\&\leq \epsilon \norm{\partial_{xz} \theta}_{L^2}^2 
+ C_\epsilon \norm{\partial_z \theta}_{L^2}^2 \norm{\partial_x v}_{L^2}^2 \norm{\partial_{xz} v}_{L^2}^2.
\end{aligned}
\end{align}

For $J_3$, we get
\begin{align}\label{jjjj-3}
\begin{aligned}
J_3 &:= C_\epsilon \norm{\omega \theta}_{L^2}^2 
\leq C_\epsilon \norm{\omega}_{L^2} \norm{\partial_x \omega}_{L^2} \norm{\theta}_{L^2} \norm{\partial_z \theta}_{L^2}.
\end{aligned}
\end{align}
We note that
\begin{align*}
&\begin{aligned}
&\norm{ \left( \int_0^z \partial_x \omega(t,x,\xi) d\xi \right) \partial_z \Theta }_{L^2}^2 \\
&\quad \leq \int_0^1 \int_0^\alpha \left( \norm{\partial_z \Theta}_{L_x^\infty}^2 \left( \int_0^z \partial_x \omega(\xi) d\xi \right)^2 \right) dx dz \\
&\quad \leq C \norm{\partial_z \Theta}_{L_z^2 L_x^\infty}^2 \norm{\partial_x \omega}_{L_x^\infty L_z^2}^2, \\
&\norm{\partial_z \Theta}_{L_x^2 L_z^\infty}^2 \leq 2 \int_D |\partial_z \Theta  \partial_{zz} \Theta| dx dz 
\leq 2 \norm{\partial_z \Theta}_{L^2} \norm{\partial_{zz} \Theta}_{L^2}, \\
&\norm{\partial_x \omega}_{L_x^\infty L_z^2}^2 \leq 2 \norm{\partial_x \omega}_{L^2} \norm{\partial_{xx} \omega}_{L^2}.
\end{aligned}
\end{align*}
This, combined with Young’s inequality, yields the estimate for $J_1$
as follows
\begin{align}\label{jjjj-4}
\begin{aligned}
J_4 &:= C_\epsilon \norm{ \left( \int_0^z \partial_x \omega(t,x,\xi) d\xi \right) \partial_z \Theta }_{L^2}^2 
\\&\leq \epsilon \norm{\partial_{xx} \omega}_{L^2}^2 
+ C_\epsilon \norm{\partial_x \omega}_{L^2}^2 \norm{\partial_z \Theta}_{L^2}^2 \norm{\partial_{zz} \Theta}_{L^2}^2,
\end{aligned}
\end{align}

Note that the cross terms exhibit the following cancellation:
\[
2\text{R} \left( \int_0^z \partial_x \theta(t,x,\xi) d\xi, \partial_{xx} \omega \right)
- 2\text{R} \left( \int_0^z \partial_x \omega(t,x,\xi) d\xi, \partial_{xx} \theta \right) = 0.
\]

Adding \eqref{nablaomega} and \eqref{nablatheta}, and using \eqref{kkkk-1}-\eqref{kkkk-4} and \eqref{jjjj-1}-\eqref{jjjj-4} with $\epsilon$ chosen sufficiently small, we obtain
\begin{align}\label{h2dtestimate-all-2}
\begin{aligned}
&\frac{d}{dt} \left( \int_D |\partial_x \omega|^2 dx dz + \int_D |\partial_x \theta|^2 dx dz \right) \\
&+ \text{Pr}_x \norm{\partial_{xx} \omega}_{L^2}^2 
+ \text{Pr}_z \norm{\partial_{zx} \omega}_{L^2}^2 
+ \norm{\partial_{xx} \theta}_{L^2}^2 
+ \kappa_a \norm{\partial_{zx} \theta}_{L^2}^2 \\
&\leq C \norm{\partial_z v}_{L^2}^2 \norm{v}_{L^2}^2 \norm{\partial_x \omega}_{L^2}^2 
+ C \norm{\partial_z \omega}_{L^2}^2 \norm{\partial_x v}_{L^2}^2 \norm{\partial_{xz} v}_{L^2}^2 \\
&\quad + C \norm{\partial_x v}_{L^2}^2 \norm{\partial_{xx} v}_{L^2} \norm{\partial_z \omega}_{L^2} 
+ C \norm{\partial_z \omega}_{L^2}^2 \norm{\partial_x \omega}_{L^2}^2 \norm{\partial_{xz} \omega}_{L^2}^2\\
&\quad +C \norm{\partial_z v}_{L^2}^2 \norm{v}_{L^2}^2 \norm{\partial_x \theta}_{L^2}^2+C
\norm{\partial_z \theta}_{L^2}^2 \norm{\partial_x v}_{L^2}^2 \norm{\partial_{xz} v}_{L^2}^2\\
&\quad +C
 \norm{\omega}_{L^2} \norm{\partial_x \omega}_{L^2} \norm{\theta}_{L^2} \norm{\partial_z \theta}_{L^2} 
 +C\norm{\partial_x \omega}_{L^2}^2 \norm{\partial_z \Theta}_{L^2}^2 \norm{\partial_{zz} \Theta}_{L^2}^2.
\end{aligned}
\end{align}
Furthermore, based on Lemma~\ref{v-l2-l2}--Lemma \ref{lemma-grad-1-6}, integrating \eqref{h2dtestimate-all-2} with respect to time establishes \eqref{infty-vxx-delta-17-1}. Finally, combining \eqref{h2dtestimate-all-2} and \eqref{infty-vxx-delta-17-1} and employing arguments similar to those used in the proof of \eqref{decy-vxz-vzz-16}, we obtain the exponential decay estimate \eqref{decy-vxz-vzz-17-1-1}.
\end{proof}

\begin{lemma}\label{lemma-grad-1-8}
Let $T_1 > T_0$ and consider any smooth solution $(v, \Theta)$ to system \eqref{ondimensional-equations-2} with initial data $(v_0, \Theta_0) \in H^2(D)$. Then there exists a constant $\lambda > 0$ such that
\begin{align}\label{decy-vxz-vzz-17-1}
\norm{\partial_x q}_{L^2(D)}^2 + \norm{\partial_t (v, \Theta)}_{L^2(D)}^2 \leq C C_0 e^{-2\lambda t} \norm{(v_0, \Theta_0)}_{H^2(D)}^2, \quad \forall t \geq 0,
\end{align}
where the constant $C_0$ depends on $\norm{v_0}_{H^1(D)}$ and $\norm{\Theta_0}_{L^2(D)}$.
\end{lemma}
\begin{proof}
Multiplying the equation \eqref{ondimensional-equations-2}$_1$ by $\partial_x q$ in $L^2(D)$, we obtain
\begin{align}\label{xqq}
\begin{aligned}
\norm{\partial_x q}_{L^2}^2 
&= \int_D \left( \text{Pr}_x \partial_{xx} v + \text{Pr}_z \partial_{zz} v \right) \partial_x q  dx dz - \int_D \partial_t v  \partial_x q  dx dz \\
&\quad - \int_D \left( v \partial_x v - \left( \int_0^z \partial_x v(t,x,\xi) d\xi \right) \partial_z v \right) \partial_x q  dx dz \\
&\quad - \text{R} \int_D \left( \partial_x \left( \int_0^z \Theta(t,x,\xi) d\xi \right) \right) \partial_x q  dx dz.
\end{aligned}
\end{align}

Note that the time derivative term vanishes due to the incompressibility condition:
\begin{align*}
\int_D \partial_t v  \partial_x q  dx dz 
&= \int_D \left( \partial_z \left( \int_0^z v_t(t,x,\xi) d\xi \right) \right) \partial_x q  dx dz \\
&= - \int_D \left( \int_0^z v_t(t,x,\xi) d\xi \right) \partial_z (\partial_x q)  dx dz = 0.
\end{align*}
Thus, we have the estimate
\begin{align*}
\norm{\partial_x q}_{L^2}^2 
&\leq \frac{1}{2} \norm{\partial_x q}_{L^2}^2 
+ C \left( \norm{v}_{H^2}^2 + \norm{\Theta}_{L^2}^2 \right) \\
&\quad + C \norm{v}_{L^2} \norm{\partial_z v}_{L^2} \norm{\partial_x v}_{L^2} \norm{\partial_{xx} v}_{L^2} \\
&\quad + C_\epsilon \norm{\partial_z v}_{L^2} \norm{\partial_{xx} v}_{L^2} \norm{\partial_{zz} v}_{L^2} \norm{\partial_x v}_{L^2}.
\end{align*}
Combining this with the estimates in Lemmas~\ref{v-l2-l2}--\ref{lemma-grad-1-7}, there exists a constant $\lambda > 0$ such that
\begin{align}\label{decy-vxz-vzz-17-q}
\norm{\partial_x q}_{L^2(D)}^2 \leq C C_0 e^{-2\lambda t} \norm{(v_0, \Theta_0)}_{H^2(D)}^2,
\end{align}
where the constant $C_0$ depends on $\norm{v_0}_{H^1(D)}$.

Furthermore, from the system \eqref{ondimensional-equations-2}, we estimate the time derivatives:
\begin{align*}
\norm{\partial_t (v, \Theta)}_{L^2(D)} 
&\leq C \norm{(v, \Theta)}_{H^2(D)}  + \norm{ v \partial_x v - \left( \int_0^z \partial_x v(t,x,\xi) d\xi \right) \partial_z v }_{L^2(D)} \\
&\quad + \norm{ v \partial_x \Theta - \left( \int_0^z \partial_x v(t,x,\xi) d\xi \right) \partial_z \Theta }_{L^2(D)} \\
&\leq C \norm{(v, \Theta)}_{H^2(D)} \leq C C_0 e^{-2\lambda t} \norm{(v_0, \Theta_0)}_{H^2(D)}^2.
\end{align*}
Combining \eqref{decy-vxz-vzz-17-q} with the time derivative estimate completes the proof of \eqref{decy-vxz-vzz-17-1}.
\end{proof}

\subsection{Case $T_1=T_0$}
For case $T_1 = T_0$, Theorem \ref{global-stability} can be derived by
the following 
Lemma~\ref{lemma-grad-1-1-cases2}--Lemma \ref{lemma-grad-1-8-cases2}.

\subsubsection{$L^2$-Estimates and Global Stability}
\begin{lemma}\label{lemma-grad-1-1-cases2}
Let $T_1 = T_0$ and consider any smooth solution $(v, \Theta)$ to the system \eqref{ondimensional-equations-2} with initial data $(v_0, \Theta_0) \in L^2(D)$. Then the following regularity estimates hold:
\begin{align}\label{u21-1-2}
\begin{aligned}
&v \in L^\infty\big( (0,\infty); L^2(D) \big) \cap L^2\big( (0,\infty); H^1(D) \big), \\
&\Theta \in L^\infty\big( (0,\infty); L^2(D) \big) \cap L^2\big( (0,\infty); H^1(D) \big),
\end{aligned}
\end{align}
and there exists a constant $\lambda > 0$ such that
\begin{align}\label{case-2-decay-theta-v-2}
\norm{v(t)}_{L^2(D)}^2 + \norm{\Theta(t)}_{L^2(D)}^2 \leq e^{-2\lambda t} \left( \norm{v_0}_{L^2(D)}^2 + \norm{\Theta_0}_{L^2(D)}^2 \right), \quad \forall t \geq 0.
\end{align}
\end{lemma}
\begin{proof}
The mathematical structure of the system \eqref{ondimensional-equations-2} in the regime $T_1 > T_0$ is fundamentally different from that in the case $T_1 = T_0$. Consequently, the proof of the present lemma requires an approach distinct from the one used for Lemma~\ref{v-l2-l2}. First, we establish the exponential decay of $\Theta$. Multiplying the equation \eqref{ondimensional-equations-2}$_2$ by $\Theta$ and integrating by parts over $D$, we obtain
\begin{align}\label{t1=t0-10}
\frac{d}{dt} \int_D \Theta^2 dx dz = -2 \int_D |\partial_x \Theta|^2 dx dz - 2\kappa_a \int_D |\partial_z \Theta|^2 dx dz.
\end{align}
The Poincaré inequality then yields the existence of $\lambda > 0$ such that
\begin{align}\label{case2-Theta-l2}
\norm{\Theta(t)}_{L^2(D)}^2 \leq e^{-\lambda t} \norm{\Theta_0}_{L^2(D)}^2, \quad \forall t \geq 0.
\end{align}

Next, we prove the exponential decay of $v$. Multiplying equation \eqref{ondimensional-equations-2}$_1$ by $v$ and integrating by parts gives
\begin{align}\label{t1=t0-3}
\begin{aligned}
\frac{d}{dt} \int_D v^2 dx dz 
&= -2\text{Pr}_x \int_D |\partial_x v|^2 dx dz - 2\text{Pr}_z \int_D |\partial_z v|^2 dx dz \\
&\quad - 2 \int_D \left( \partial_x \left( \int_0^z \Theta(t,x,\xi) d\xi \right) + \partial_x q(t,x) \right) v dx dz \\
&= -2\text{Pr}_x \int_D |\partial_x v|^2 dx dz - 2\text{Pr}_z \int_D |\partial_z v|^2 dx dz \\
&\quad - 2 \int_D \left( \Theta \int_0^z \partial_x v(t,x,\xi) d\xi \right) dx dz \\
&\leq -\text{Pr}_x \int_D |\partial_x v|^2 dx dz - 2\text{Pr}_z \int_D |\partial_z v|^2 dx dz \\
&\quad + C \int_D |\Theta|^2 dx dz.
\end{aligned}
\end{align}
Here, the pressure term vanishes due to the incompressibility condition and boundary conditions.

Combining \eqref{t1=t0-3} with the decay estimate \eqref{case2-Theta-l2} and applying the generalized Gronwall inequality (Lemma~\ref{general-inequality1}), we obtain
\begin{align}\label{t1=t0-1}
\norm{v(t)}_{L^2(D)}^2 \leq C e^{-\lambda t} \left( \norm{v_0}_{L^2(D)}^2 + \norm{\Theta_0}_{L^2(D)}^2 \right), \quad \forall t \geq 0.
\end{align}

The combination of \eqref{case2-Theta-l2} and \eqref{t1=t0-1} establishes the decay estimate \eqref{case-2-decay-theta-v-2}. The regularity properties \eqref{u21-1-2} follow from the time integration of the corresponding differential inequalities.
\end{proof}

\subsubsection{$H^1$-Estimates and Global Stability}

\begin{lemma}\label{lemma-grad-1-2-cases2}
Let $T_1 = T_0$ and consider any smooth solution $(v, \Theta)$ to the system \eqref{ondimensional-equations-2} with initial data $(v_0, \Theta_0) \in L^2(D)$ and $\partial_z v_0 \in L^2(\Omega)$. Then the following regularity estimates hold:
\begin{align}\label{infty-2}
\partial_z v \in L^\infty\big( (0,\infty); L^2(D) \big), \quad
\nabla \partial_z v \in L^2\big( (0,\infty); L^2(D) \big),
\end{align}
and there exist constants $C > 0$ and $\lambda > 0$ such that
\begin{align}\label{infty-decay-2}
\norm{\partial_z v(t)}_{L^2(D)}^2 \leq C e^{-2\lambda t} \left( \norm{(v_0, \Theta_0)}_{L^2(D)}^2 + \norm{\partial_z v_0}_{L^2(D)}^2 \right), \quad \forall t \geq 0.
\end{align}
\end{lemma}
\begin{proof}
This result can be proved using the same approach as in the proof of 
Lemma~\ref{lemma-grad-1-2}.
\end{proof}

\begin{lemma}\label{lemma-grad-1-3-cases2}
Let $T_1 = T_0$ and consider any smooth solution $(v, \Theta)$ to the system \eqref{ondimensional-equations-2} with initial data $(v_0, \Theta_0) \in H^2(D)$ and $\nabla v_0 \in L^2(\Omega)$. Then the horizontal derivative satisfies
\begin{align}\label{case-2-v_x-infty-2}
\partial_x v \in L^\infty\big( (0,\infty); L^2(D) \big) \cap L^2\big( (0,\infty); H^1(D) \big).
\end{align}
\end{lemma}
\begin{proof}
The proof follows from arguments similar to those used in the proof of Lemma~\ref{lemma-grad-1-3}.
\end{proof}

\begin{lemma}\label{lemma-grad-1-4-cases2}
Let $T_1 = T_0$ and consider any smooth solution $(v, \Theta)$ to the system \eqref{ondimensional-equations-2} with initial data $(v_0, \Theta_0) \in H^1(D)$. Then the vertical derivative of temperature satisfies
\begin{align}\label{infty-thetaz-2}
\partial_z \Theta \in L^\infty\big( (0,\infty); L^2(D) \big) \cap L^2\big( (0,\infty); H^1(D) \big),
\end{align}
and there exists a constant $\lambda > 0$ such that
\begin{align}\label{decaythetaz-2}
\norm{\partial_z \Theta(t)}_{L^2(D)}^2 \leq C C_0 e^{-2\lambda t} \norm{(v_0, \Theta_0)}_{H^1(D)}^2, \quad \forall t \geq 0,
\end{align}
where the constant $C_0$ depends on $\norm{v_0}_{H^1(D)}$ and $\norm{\Theta_0}_{L^2(D)}$.
\end{lemma}
\begin{proof}
This result can be proved using the same approach as in the proof of 
Lemma~\ref{lemma-grad-1-4}.
\end{proof}

\begin{lemma}\label{lemma-grad-1-5-cases2}
Let $T_1 = T_0$ and consider any smooth solution $(v, \Theta)$ to the system \eqref{ondimensional-equations-2} with initial data $(v_0, \Theta_0) \in H^1(D)$. Then the horizontal derivatives satisfy the regularity estimate
\begin{align}\label{cases-2-theta-v-infty-xx-33}
(\partial_x v, \partial_x \Theta) \in L^\infty\big( (0,\infty); L^2(D) \big) \cap L^2\big( (0,\infty); H^1(D) \big).
\end{align}
Moreover, there exist constants $\lambda > 0$ and $C \geq 1$ such that
\begin{align}\label{cases-2-theta-v-decay-x}
\norm{\partial_x v(t)}_{L^2(D)}^2 + \norm{\partial_x \Theta(t)}_{L^2(D)}^2 \leq C C_0 e^{-2\lambda t} \norm{(v_0, \Theta_0)}_{H^1(D)}^2, \quad \forall t \geq 0,
\end{align}
where the constant $C_0$ depends on $\norm{v_0}_{H^1(D)}$ and $\norm{\Theta_0}_{L^2(D)}$.
\end{lemma}

\begin{proof}
The proof of this lemma is different from that of Lemma~\ref{lemma-grad-1-5}
due to $T_1=T_0$. We firstly show the decay of $\norm{\partial_x \Theta}_{L^2}^2$. Multiplying the equation $\eqref{ondimensional-equations-2}_2$  by $\Theta_{xx}$, integrating by parts, we get
\begin{align}\label{theta-11}
\begin{aligned}
&\frac{d}{dt}\norm{\partial_x \Theta}_{L^2}^2
+ 2\int_{D}\abs{\partial_{xx}\Theta}^2
+2\kappa_{a} \int_{D}\abs{\partial_{xz}\Theta}^2\\
&=2\int_{D}
\left(v \partial_x \Theta-\left(\int_0^z\partial_xv(t,x,\xi)\,d\xi \right)\partial_z \Theta
\right)\partial_{xx} \Theta\,dx\,dz=:L_1+L_2.
\end{aligned}
\end{align}
For $L_1$, we have
\begin{align}
\begin{aligned}
L_1&=2\int_{D}
(v \partial_x \Theta)\partial_{xx}\Theta\,dx\,dz
\\&\leq 2\norm{\partial_{xx} \Theta}_{L^2}
\norm{v\partial_x\Theta}_{L^2}
\leq\epsilon \norm{\partial_{xx} \Theta}_{L^2}^2+
C_{\epsilon}\norm{v\partial_x\Theta}_{L^2}^2\\
&\leq\epsilon \norm{\partial_{xx} \Theta}_{L^2}^2+
\norm{v}_{L^2}
 \norm{
\partial_zv}_{L^2}
\norm{
\partial_x\Theta}_{L^2}
\norm{
\partial_{xx}\Theta}_{L^2}
\end{aligned}
\end{align}
where we have used Young’s inequality and the following estimates
    \begin{align}
&\begin{aligned}
&\norm{v\partial_x\Theta}_{L^2}^2
\leq
\norm{\partial_x\Theta}_{L^{\infty}_{x}L^2_{z}}^2
\norm{v}_{L^2_{x}L^{\infty}_{z}}^2,\\
&\norm{\partial_x\Theta}
_{L^{\infty}_{x}L^2_{z}}^2
\leq 2\int_{D}\abs{\partial_x\Theta
\partial_{xx}\Theta}\,dx\,dz
\leq 2\norm{
\partial_x\Theta}_{L^2}
\norm{
\partial_{xx}\Theta}_{L^2},\\
&\norm{v}_{L^2_{x}L^{\infty}_{z}}^2\leq 
2\norm{
\partial_zv}_{L^2}
\norm{v}_{L^2}.
\end{aligned}
\end{align}
For $L_2$, we have
    \begin{align}
&\begin{aligned}
L_2&=-2\int_{D}
\left(\int_0^z\partial_xv(t,x,\xi)\,d\xi \right)\partial_z \Theta
\partial_{xx} \Theta\,dx\,dz\\&\leq
 2\norm{\partial_{xx} \Theta}_{L^2}\norm{\left(\int_0^z\partial_xv(t,x,\xi)\,d\xi \right)\partial_z\Theta}_{L^2}\\
&\leq  
\epsilon \norm{\partial_{xx} \Theta}_{L^2}^2+
 C_{\epsilon}
 \norm{
\partial_xv}_{L^2}
\norm{
\partial_{xx}v}_{L^2}
\norm{
\partial_z\Theta}_{L^2}
\norm{
\partial_{zz}\Theta}_{L^2}\\
&\leq 
\epsilon \norm{\partial_{xx} \Theta}_{L^2}^2
+\epsilon\norm{\partial_{xx}v}_{L^2}^2+
C_{\epsilon}
\norm{
\partial_z\Theta}_{L^2}^2
 \norm{
\partial_xv}_{L^2}^2
\norm{
\partial_{zz}\Theta}_{L^2}^2,
\end{aligned}
\end{align}
where we have used Young’s inequality and the following estimates
    \begin{align}
&\begin{aligned}
&\begin{aligned}
 \norm{
\left(\int_0^z\partial_xv(t,x,\xi)\,d\xi \right)\partial_z\Theta}_{L^2}^2
&\leq \int_0^1
\left( \left(\int_0^z\norm{\partial_xv}_{L_x^{\infty}}(\xi)\,d\xi
\right)^2
\int_{0}^{\alpha}\abs{\partial_z\Theta}^2\,dx\right)\,dz\\
&\leq C \norm{\partial_z\Theta}^2_{L_z^{\infty}L_x^{2}}
\norm{\partial_xv}_{L_z^{2}L_x^{\infty}}^2,
\end{aligned}\\
&\norm{\partial_z\Theta}
_{L^{\infty}_{z}L^2_{x}}^2
\leq 2\int_{D}\abs{\partial_z\Theta
\partial_{zz}\Theta}\,dx\,dz
\leq 2\norm{
\partial_z\Theta}_{L^2}
\norm{
\partial_{zz}\Theta}_{L^2},\\
&\norm{\partial_xv}_{L_z^{2}L_x^{\infty}}^2\leq 2\norm{
\partial_xv}_{L^2}
\norm{
\partial_{xx}v}_{L^2},
\end{aligned}
\end{align}

Taking small enough $\epsilon$,there exists a $\lambda>0$ such that
\begin{align}\label{lambda-1}
\begin{aligned}
&\frac{d}{dt}\norm{\partial_x \Theta}_{L^2}^2+
 \int_{D}\abs{\partial_{xx}\Theta}^2
+\kappa_{a} \int_{D}\abs{\partial_{xz}\Theta}^2\leq 
e^{-\lambda t}
R(t),
\end{aligned}
\end{align}
where $R(t)$ is given by
\begin{align}
\begin{aligned}
 R(t):=&e^{\lambda t}\norm{v}_{L^2}^2
 \norm{
\partial_zv}_{L^2}^2
\norm{
\partial_x\Theta}_{L^2}^2\\
&+C_{\epsilon}
e^{\lambda t}
\norm{
\partial_z\Theta}_{L^2}^2
 \norm{
\partial_xv}_{L^2}^2
\norm{
\partial_{zz}\Theta}_{L^2}^2\in L^1(0,\infty).
\end{aligned}
\end{align}
This, together with Lemma \ref{general-inequality1}, yields that 
there exist $\lambda>0$ and $C\geq 1$ such that
\begin{align}
\norm{\partial_x \Theta}_{L^2}^2\leq 
Ce^{-2\lambda t}
\left(\norm{\partial_x \Theta_0}_{L^2}^2
+\norm{\Theta_0}_{L^2}^2
+\norm{v_0}_{L^2}^2\right).
\end{align}

Multiplying equation $\eqref{ondimensional-equations-2}_1$ by $v_{xx}$, after an integration by part, we get
\begin{align}\label{v-11}
\begin{aligned}
&\frac{d}{dt}
\norm{\partial_x v}_{L^2}^2
+2\text{Pr}_x \int_{D}\abs{\partial_{xx}v}^2
+2\text{Pr}_z \int_{D}\abs{\partial_{xz}v}^2
\\&=2\int_{D}
\left(v \partial_x v-\left(\int_0^z\partial_xv(t,x,\xi)\,d\xi \right)\partial_z v
\right)\partial_{xx} v\,dx\,dz\\
&\quad-2\int_{D}\left(\partial_x \left(
\int_0^z\Theta(t,x,\xi)\,d\xi\right)+\partial_xq(t,x)\right)
\partial_z\left(\int_0^zv_{xx}
(t,x,\xi)\,d\xi\right)
\,dx\,dz
\\&:=S_1+S_2+S_3.
\end{aligned}
\end{align}

For $S_1+S_2$, we have
\begin{align}
\begin{aligned}
S_1+S_2=&
\epsilon \norm{\partial_{xx} v}_{L^2}^2+\epsilon \norm{\partial_{xz} v}_{L^2}^2+C_{\epsilon}
\norm{v}_{L^2}^2
 \norm{
\partial_zv}_{L^2}^2
\norm{
\partial_xv}_{L^2}^2
\\
&+
 C_{\epsilon}
 \norm{
\partial_zv}_{L^2}^2
\norm{
\partial_xv}_{L^2}^4.
\end{aligned}
\end{align}

For $S_3$, we have
\begin{align}
\begin{aligned}
S_3&=-2\int_{D}\left(\text{R}_0\partial_x \left(
\int_0^z\Theta(t,x,\xi)\,d\xi\right)+\partial_xq(t,x)\right)
\partial_z\left(\int_0^zv_{xx}
(t,x,\xi)\,d\xi\right)
\,dx\,dz\\
&=2\int_{D}\partial_x\Theta
\left(\int_0^zv_{xx}
(t,x,\xi)\,d\xi\right)
\,dx\,dz
\\&\leq 2\norm{\partial_{xx} v}_{L^2}
\norm{\partial_x\Theta}_{L^2}
\leq\epsilon \norm{\partial_{xx} v}_{L^2}^2+
C_{\epsilon}\norm{\partial_x\Theta}_{L^2}^2.
\end{aligned}
\end{align}
Taking small enough $\epsilon$, there exists a $\lambda>0$ such that
\begin{align}\label{lambda-2}
\begin{aligned}
&\frac{d}{dt}
\norm{\partial_x v}_{L^2}^2
+\text{Pr}_x \int_{D}\abs{\partial_{xx}v}^2
+\text{Pr}_z \int_{D}\abs{\partial_{xz}v}^2\leq 
e^{-\lambda t}
r(t),
\end{aligned}
\end{align}
where $r(t)$ is given by
\begin{align}
\begin{aligned}
 r(t):=&e^{\lambda t}
C_{\epsilon}
\norm{v}_{L^2}^2
 \norm{
\partial_zv}_{L^2}^2
\norm{
\partial_xv}_{L^2}^2
\\
&+e^{\lambda t}
 C_{\epsilon}
 \norm{
\partial_zv}_{L^2}^2
\norm{
\partial_xv}_{L^2}^4
+C_{\epsilon}e^{\lambda t}\norm{\partial_x\Theta}_{L^2}^2\in L^1(0,\infty).
\end{aligned}
\end{align}
This, together with Lemma \ref{general-inequality1}, yields that 
there exist $\lambda>0$ and $C\geq 1$ such that
\begin{align*}
\norm{\partial_x v}_{L^2}^2\leq
CC_0
e^{-2\lambda t}
\norm{(v_0,\Theta_0)}_{H^1(D)}^2.
\end{align*}
by which, \eqref{theta-11} and \eqref{v-11}, we have \eqref{cases-2-theta-v-decay-x}.
  \end{proof}
  
\subsubsection{$H^2$-Estimates and Global Stability}

\begin{lemma}\label{cases-2-lemma-grad-1-6}
Let $T_1 =T_0$ and consider any smooth solution $(v, \Theta)$ to the system \eqref{ondimensional-equations-2} with initial data $(v_0, \Theta_0) \in H^2(D)$. Then the following regularity estimates hold:
\begin{align}\label{vh2-infty-16-case-2}
\partial_z(\nabla v, \nabla \Theta) \in L^\infty\big( (0,\infty); L^2(D) \big), \quad
\Delta(\partial_z v, \partial_z \Theta) \in L^2\big( (0,\infty); L^2(D) \big).
\end{align}
Moreover, there exists a constant $\lambda > 0$ such that
\begin{align}\label{decy-vxz-vzz-16-case-2}
\norm{\partial_z (\nabla v, \nabla \Theta)(t)}_{L^2(D)}^2 \leq C C_0 e^{-2\lambda t} \norm{(v_0, \Theta_0)}_{H^2(D)}^2, \quad \forall t \geq 0,
\end{align}
where the constant $C_0$ depends on $\norm{v_0}_{H^1(D)}$ and $\norm{\Theta_0}_{L^2(D)}$.
\end{lemma}
\begin{proof}
The proof of this lemma differs from that of Lemma~\ref{lemma-grad-1-5} due to the condition $T_1 = T_0$. In this case, we introduce the variable $\eta := \partial_z \Theta$. From the equation \eqref{ondimensional-equations-2}$_2$, we find that $\eta$ satisfies
\begin{align}\label{thetaz-eta-cases}
\begin{aligned}
\partial_t \eta + v \partial_x \eta - \left( \int_0^z \partial_x v(t,x,\xi)  d\xi \right) \partial_z \eta 
&= \partial_{xx} \eta + \kappa_a \partial_{zz} \eta  - \sigma \partial_x \Theta + \partial_x v  \eta.
\end{aligned}
\end{align}

Multiplying the equation \eqref{thetaz-eta-cases} successively by $\partial_{xx}\eta$
and $\partial_{zz}\eta$, then integrating by parts and applying Young’s inequality, we obtain
 \begin{align*}
\begin{aligned}
&\frac{d}{dt}
\int_{D}\abs{\nabla \eta}^2
\,dx\,dz
+2\norm{\partial_{xx}\eta}_{L^2}^2
+2\kappa_a \norm{\partial_{zz}\eta}_{L^2}^2
+2(1+\kappa_a) \norm{\partial_{xz}\eta}_{L^2}^2\\
&\leq \epsilon 
\left(\norm{\partial_{xx}\eta}_{L^2}^2
+
\norm{\partial_{zz}\eta}_{L^2}^2\right)
+2\norm{v \partial_x \eta}_{L^2}\left(\norm{\partial_{xx}\eta}_{L^2}
+
\norm{\partial_{zz}\eta}_{L^2}\right)
\\&
+2\norm{\left(\int_0^z\partial_xv(t,x,\xi)\,d\xi \right)\partial_z \eta}
_{L^2}\left(\norm{\partial_{xx}\eta}_{L^2}
+
\norm{\partial_{zz}\eta}_{L^2}\right)\\
&\quad+2\norm{\sigma \partial_x \Theta}_{L^2}\left(\norm{\partial_{xx}\eta}_{L^2}
+
\norm{\partial_{zz}\eta}_{L^2}\right)\\
&\quad+2\norm{\eta \partial_x v}_{L^2}\left(\norm{\partial_{xx}\eta}_{L^2}
+
\norm{\partial_{zz}\eta}_{L^2}\right)\\
&\leq 
 \epsilon 
\left(\norm{\partial_{xx}\eta}_{L^2}^2
+\norm{\partial_{xz}\eta}_{L^2}^2+
\norm{\partial_{zz}\eta}_{L^2}^2\right)
+C_{\epsilon}\norm{
\partial_x\eta}_{L^2}^2
\norm{
\partial_zv}_{L^2}^2
\norm{v}_{L^2}^2\\
&\quad +C_{\epsilon}\norm{
\partial_z\eta}_{L^2}^2
\norm{
\partial_xv}_{L^2}\norm{
\partial_{zx}v}_{L^2}+
C_{\epsilon}
\norm{
\partial_z\eta}_{L^2}^2
\norm{
\partial_xv}_{L^2}^2\norm{
\partial_{zx}v}_{L^2}^2\\
&\quad+C_{\epsilon}
\norm{
\partial_{zz}v}_{L^2}
\norm{\partial_{z}v}_{L^2}
\norm{
\partial_x\Theta}_{L^2}
\norm{
\partial_{xx}\Theta}_{L^2}
\\&\quad+C_{\epsilon}\norm{
\partial_xv}_{L^2}
\norm{
\partial_{xx}v}_{L^2}
\norm{
\partial_z\eta }_{L^2}
\norm{\eta }_{L^2}.
\end{aligned}
\end{align*}
Choosing $\epsilon$ sufficiently small, there exists $\lambda_1 > 0$ such that
\begin{align}\label{nablaeta-cases-2}
\begin{aligned}
&\frac{d}{dt} \int_D |\nabla \eta|^2 dx dz 
+ \lambda_1 \left( \norm{\partial_{xx} \eta}_{L^2}^2 + \norm{\partial_{zz} \eta}_{L^2}^2 + \norm{\partial_{xz} \eta}_{L^2}^2 \right) \\
&\leq \norm{\nabla \eta}_{L^2}^2 g(t) + h(t),
\end{aligned}
\end{align}
where $g(t), h(t) \in L^1([0,\infty))$ are given by
\[
\begin{aligned}
g(t) &:= C \norm{\partial_z v}_{L^2}^2 \norm{v}_{L^2}^2 + C \norm{\partial_x v}_{L^2} \norm{\partial_{zx} v}_{L^2} \\
&\quad + C \norm{\partial_x v}_{L^2}^2 \norm{\partial_{zx} v}_{L^2}^2 + C \norm{\eta}_{L^2}^2, \\
h(t) &:= C \norm{\partial_x v}_{L^2}^2 \norm{\partial_{xx} v}_{L^2}^2 + C \norm{\partial_{zz} v}_{L^2}^2 \norm{\partial_z v}_{L^2}^2 \\
&\quad + C \norm{\partial_x \Theta}_{L^2}^2 \norm{\partial_{xx} \Theta}_{L^2}^2.
\end{aligned}
\]

An application of Gronwall's inequality then yields
\[
\nabla \eta \in L^\infty\big( (0,\infty); L^2(D) \big).
\]
\
Moreover, there exists $\lambda_2 > 0$ such that
\[
\norm{\nabla \eta}_{L^2}^2 g(t) + h(t) = e^{-\lambda_2 t} D_1(t), \quad \text{with } D_1(t) \in L^1([0,\infty)).
\]
Combining this with Lemma~\ref{general-inequality1}, we conclude that there exist $\lambda > 0$ and $C \geq 1$ such that
\begin{align*}
\norm{\nabla \eta(t)}_{L^2(D)}^2 \leq C C_0 e^{-2\lambda t} \norm{(v_0, \Theta_0)}_{H^1(D)}^2.
\end{align*}

Integrating \eqref{nablaeta-cases-2} with respect to time yields
\[
\Delta \eta \in L^2\big( (0,\infty); L^2(D) \big).
\]
Proceeding similarly to the derivation of \eqref{nablasigma}, we multiply the equation \eqref{three-proof--8} successively by $\partial_{xx} \sigma$ and $\partial_{zz} \sigma$, then integrate by parts to obtain
\begin{align}\label{nablasigma-cases-2}
\begin{aligned}
&\frac{d}{dt} \int_D |\nabla \sigma|^2 dx dz 
+ \text{Pr}_x \norm{\partial_{xx} \sigma}_{L^2}^2 
+ \text{Pr}_z \norm{\partial_{zz} \sigma}_{L^2}^2 \\
&\leq C \norm{\partial_z v}_{L^2}^2 \norm{v}_{L^2}^2 \norm{\partial_x \sigma}_{L^2}^2 
+ C \norm{\partial_x v}_{L^2}^2 \norm{\partial_{xx} v}_{L^2}^2 \norm{\partial_z \sigma}_{L^2}^2 + C \norm{\partial_x \Theta}_{L^2}^2.
\end{aligned}
\end{align}
From this inequality, we deduce
\begin{align*}
\norm{\nabla \sigma(t)}_{L^2(D)}^2 \leq C C_0 e^{-2\lambda t} \norm{(v_0, \Theta_0)}_{H^1(D)}^2, \quad
\Delta \sigma \in L^2\big( (0,\infty); L^2(D) \big).
\end{align*}
\end{proof}
\begin{lemma}\label{cases-2-lemma-grad-1-7}
Let $T_1 =T_0$ and consider any smooth solution $(v, \Theta)$ to the system \eqref{ondimensional-equations-2} with initial data $(v_0, \Theta_0) \in H^2(D)$. Then the following regularity estimates hold:
\begin{align}\label{vh2-infty-17-case-2}
\partial_x(\nabla v, \nabla \Theta) \in L^\infty\big( (0,\infty); L^2(D) \big), \quad
\Delta(\partial_x v, \partial_x \Theta) \in L^2\big( (0,\infty); L^2(D) \big).
\end{align}
Moreover, there exists a constant $\lambda > 0$ such that
\begin{align}\label{decy-vxz-vzz-17-case-2}
\norm{\partial_x (\nabla v, \nabla \Theta)(t)}_{L^2(D)}^2 \leq C C_0 e^{-2\lambda t} \norm{(v_0, \Theta_0)}_{H^2(D)}^2, \quad \forall t \geq 0,
\end{align}
where the constant $C_0$ depends on $\norm{v_0}_{H^1(D)}$ and $\norm{\Theta_0}_{L^2(D)}$.
\end{lemma}
\begin{proof}
This result can be proved using the same approach as in the proof of 
Lemma~\ref{lemma-grad-1-7}.
\end{proof}

\begin{lemma}\label{lemma-grad-1-8-cases2}
Let $T_1 = T_0$ and consider any smooth solution $(v, \Theta)$ to the system \eqref{ondimensional-equations-2} with initial data $(v_0, \Theta_0) \in H^2(D)$. Then there exists a constant $\lambda > 0$ such that
\begin{align}\label{decy-vxz-vzz-17}
\norm{\partial_x q(t)}_{L^2(D)}^2 + \norm{\partial_t (v, \Theta)(t)}_{L^2(D)}^2 \leq C C_0 e^{-2\lambda t} \norm{(v_0, \Theta_0)}_{H^2(D)}^2, \quad \forall t \geq 0,
\end{align}
where the constant $C_0$ depends on $\norm{v_0}_{H^1(D)}$ and $\norm{\Theta_0}_{L^2(D)}$.
\end{lemma}
\begin{proof}
This result can be proved using the same approach as in the proof of 
Lemma~\ref{lemma-grad-1-8}.
\end{proof}

\subsection{Case $T_1 < T_0$}

For case $0<T_0-T_1<T_c$, Theorem \ref{global-stability} can be derived by
the following 
Lemma~\ref{lemma-grad-1-1-cases3}--Lemma \ref{lemma-grad-1-8-cases3}.

\subsubsection{Critical Temperature Difference $T_c = T_0 - T_1$}

The critical temperature difference $T_c = T_0 - T_1$ signifies the threshold at which the steady-state solution \eqref{steady-state} loses linear stability. Since the square of the Rayleigh number $\mathrm{R}$ is directly proportional to $T_c$, we determine $T_c$ by first finding the critical Rayleigh number $\mathrm{R}_c$. This critical value $\mathrm{R}_c$ is characterized by the following zero eigenvalue problem:
\begin{align}\label{critical-eigenvalue}
\begin{cases}
\begin{aligned}
&\mathrm{Pr}_x \, \partial_{xx} v + \mathrm{Pr}_z \, \partial_{zz} v
- \mathrm{R} \, \partial_x \left( \int_0^z \Theta(t,x,\xi) \, d\xi \right)
- \partial_x q(t,x) = 0,
\end{aligned} \\
\begin{aligned}
&\partial_{xx} \Theta + \kappa_{a} \, \partial_{zz} \Theta
- \mathrm{R} \, \partial_x \left( \int_0^z v(t,x,\xi) \, d\xi \right) = 0,
\end{aligned} \\
 \partial_z v|_{z=0} = \partial_z v|_{z=1} = 0, \quad
\Theta|_{z=0} = \Theta|_{z=1} = 0, \\
 v, \Theta \text{ are periodic in $x$ with period $\alpha$}, \\
 \int_0^1 v(t,x,\xi) \, d\xi = 0.
\end{cases}
\end{align}
Specifically, $\mathrm{R}_c$ is the smallest $\mathrm{R}$ for which the problem \eqref{critical-eigenvalue} admits nontrivial solutions.

This eigenvalue problem \eqref{critical-eigenvalue} possesses a variational structure with respect to the parameter $\mathrm{R}$. Consequently, the critical Rayleigh number $\mathrm{R}_c$ can be characterized as the minimum value of the corresponding objective functional. Specifically, we consider the variational problem:
\begin{align}\label{ve12}
\mathrm{R}_c = \inf_{(v, \Theta) \in \mathcal{X}(D)} \frac{E_1(v, \Theta)}{2 E_2(v, \Theta)}.
\end{align}
Here, the functionals $E_1$ and $E_2$ are defined by:
\begin{align}\label{e12}
\begin{aligned}
E_1(v, \Theta) &= \int_D \left( \mathrm{Pr}_x |\partial_x v|^2 + \mathrm{Pr}_z |\partial_z v|^2 + |\partial_x \Theta|^2 + \kappa_a |\partial_z \Theta|^2 \right) dx\,dz, \\
E_2(v, \Theta) &= -\int_D v \, \partial_x \left( \int_0^z \Theta(t,x,\xi) \, d\xi \right) dx\,dz,
\end{aligned}
\end{align}
and the admissible function space $\mathcal{X}(D)$ is given by:
\begin{align*}
\begin{aligned}
\mathcal{X}(D) = \bigg\{ (v, \Theta) \in H^1(D) \,\bigg|\, & \Theta|_{z=0} = \Theta|_{z=1} = 0,\ \int_0^1 v \, dz = 0, \\
& v, \Theta \text{ are periodic in $x$ with period $\alpha$} \bigg\}.
\end{aligned}
\end{align*}

\begin{lemma}\label{cri-eigenvalue}
There exists a unique pair $(v, \Theta) \in H^2(D) \cap \mathcal{X}(D)$ such that
\[
\frac{E_1(v, \Theta)}{2 E_2(v, \Theta)} = \mathrm{R}_c = \inf_{(u, \theta) \in \mathcal{X}(D)} \frac{E_1(u, \theta)}{2 E_2(u, \theta)},
\]
and the triple $(v, \Theta, \mathrm{R}_c)$ solves the problem \eqref{critical-eigenvalue}.
\end{lemma}

\begin{proof}
Define the constrained set
\[
\mathcal{J} = \left\{ (v, \Theta) \in \mathcal{X}(D) \,\middle|\, E_2(v, \Theta) = 1 \right\}.
\]
Then the variational problem \eqref{ve12} is equivalent to
\[
\inf_{(v, \Theta) \in \mathcal{X}(D)} \frac{E_1(v, \Theta)}{E_2(v, \Theta)} = \inf_{(v, \Theta) \in \mathcal{J}} E_1(v, \Theta).
\]

Since $E_1(v, \Theta)$ is weakly lower semicontinuous and $E_2(v, \Theta)$ is weakly continuous on $H^1(D)$, the functional $E_1$ attains its minimum on $\mathcal{J}$. Let $(v, \Theta)$ be a minimizer. To show that it satisfies \eqref{critical-eigenvalue}, consider the first variation at $(v, \Theta)$: for any test function $(u, \theta) \in \mathcal{X}(D)$,
\[
\begin{aligned}
&\int_D \left( \mathrm{Pr}_x \, \partial_x v \, \partial_x u + \mathrm{Pr}_z \, \partial_z v \, \partial_z u + \partial_x \Theta \, \partial_x \theta + \kappa_a \, \partial_z \Theta \, \partial_z \theta \right) dx\,dz \\
&=- \mathrm{R}_c \int_D v \, \partial_x \left( \int_0^z \theta(t,x,\xi) \, d\xi \right) dx\,dz 
- \mathrm{R}_c \int_D u \, \partial_x \left( \int_0^z \Theta(t,x,\xi) \, d\xi \right) dx\,dz.
\end{aligned}
\]
This identity confirms that $(v, \Theta)$ is a weak solution of \eqref{critical-eigenvalue}. An application of standard elliptic regularity theory then yields the higher regularity $(v, \Theta) \in H^2(D)$, which guarantees the existence of a pressure function $q(x)$ such that the triple $(v, \Theta, q)$ is a strong solution of \eqref{critical-eigenvalue}. 
\end{proof}
\subsubsection{Exact Value of $\mathrm{R}_c$}

To determine the exact value of the critical Rayleigh number $\mathrm{R}_c$, we employ a spectral method. Consider the Fourier expansions for the perturbation fields:
\begin{align}\label{T-e}
\begin{aligned}
&v_z = \sum_{m\in \mathbb{Z},n \in \mathbb{Z}^+} (-n\pi) v_{m,n} e^{i2\pi m x / \alpha} \sin(n\pi z),\quad  v_{-m,n}=\overline{v_{m,n}},\\
&\Theta = \sum_{m\in \mathbb{Z},n \in \mathbb{Z}^+} \Theta_{m,n} e^{i2\pi m x / \alpha} \sin(n\pi z),\quad
 \Theta_{-m,n}=\overline{ \Theta_{m,n}}
\end{aligned}
\end{align}
Substituting these expansions into the linearized system
\begin{align}\label{perturbation-33}
\begin{cases}
\mathrm{Pr}_x \, \partial_{xx} v_z + \mathrm{Pr}_z \, \partial_{zz} v_z - \mathrm{R} \, \partial_x \Theta = 0, \\
\partial_{xx} \Theta + \kappa_a \, \partial_{zz} \Theta - \mathrm{R} \, \partial_x v_z = 0, \\
v_z|_{z=0} = v_z|_{z=1} = 0, \quad \Theta|_{z=0} = \Theta|_{z=1} = 0, \\
v_z, \Theta \text{ are periodic in $x$ with period $\alpha$},
\end{cases}
\end{align}
yields, for each Fourier mode $(m, n)$, the algebraic system:
\begin{align}\label{algebraic-system}
\begin{cases}
n\pi \left( \mathrm{Pr}_x \frac{4\pi^2 m^2}{\alpha^2} + \mathrm{Pr}_z n^2 \pi^2 \right) v_{m,n} 
- \mathrm{R} \, \frac{i2\pi m}{\alpha} \Theta_{m,n} = 0, \\
\mathrm{R} \, \frac{i2\pi m}{\alpha} v_{m,n} 
+ \left( \frac{4\pi^2 m^2}{\alpha^2} + \kappa_a n^2 \pi^2 \right) \Theta_{m,n} = 0.
\end{cases}
\end{align}

For $m=0$, the system \eqref{algebraic-system} only has the trivial solution
$v_{0,n}=\Theta_{0,n}=0$. Hence, we focus on the case $m\neq 0$.
For nontrivial solutions $(v_{m,n}, \Theta_{m,n}) \neq (0,0)$ to exist, the determinant of system \eqref{algebraic-system} must vanish. This solvability condition leads to:
\begin{align*}
n^2 \pi^2 \left( \mathrm{Pr}_x \frac{4\pi^2 m^2}{\alpha^2} + \mathrm{Pr}_z n^2 \pi^2 \right)
\left( \frac{4\pi^2 m^2}{\alpha^2} + \kappa_a n^2 \pi^2 \right)
- \mathrm{R}^2 \frac{4\pi^2 m^2}{\alpha^2} = 0.
\end{align*}

Solving for $\mathrm{R}^2$ yields the dispersion relation:
\begin{align}\label{RRR}
\mathrm{R}^2 = n^2 \pi^2 \left( \mathrm{Pr}_x \frac{4\pi^2 m^2}{\alpha^2} + \mathrm{Pr}_z n^2 \pi^2 \right)
\left( \frac{4\pi^2 m^2}{\alpha^2} + \kappa_a n^2 \pi^2 \right) \bigg/\left( \frac{\alpha^2}{4\pi^2 m^2} \right).
\end{align}

The critical Rayleigh number $\mathrm{R}_c$ corresponds to the minimum value of $\mathrm{R}$ over all $(m,n)\in \mathbb{Z}\times \mathbb{Z}^+$. The minimum is typically achieved for $n=1$. Setting $n=1$ in \eqref{RRR}, we obtain:
\begin{align*}
\begin{aligned}
\mathrm{R}_c^2 &= \min_{m \in \mathbb{Z}^+} \pi^2 \left( \mathrm{Pr}_x \frac{4\pi^2 m^2}{\alpha^2} + \mathrm{Pr}_z \pi^2 \right)
\left( \frac{4\pi^2 m^2}{\alpha^2} + \kappa_a \pi^2 \right) \bigg/\left( \frac{\alpha^2}{4\pi^2 m^2} \right) \\
&= \min_{m \in \mathbb{Z}^+} \pi^4 \left( \mathrm{Pr}_x \frac{4 m^2}{\alpha^2} + \mathrm{Pr}_z \right)
\left( 1 + \frac{\kappa_a \alpha^2}{4 m^2} \right).
\end{aligned}
\end{align*}

To find the minimizing wavenumber $m_c$, we rewrite the expression:
\begin{align}\label{Rc-minimization}
\begin{aligned}
\mathrm{R}_c^2 &= \pi^4 \mathrm{Pr}_z \left( 1 + \frac{\kappa_a}{\mathrm{Pr}_a} \right)
+ \pi^4 \mathrm{Pr}_z \min_{m \in \mathbb{Z}^+} \left( \frac{4 m^2}{\mathrm{Pr}_a \alpha^2} + \frac{\kappa_a \alpha^2}{4 m^2} \right),
\end{aligned}
\end{align}
where we have introduced $\mathrm{Pr}_a = \mathrm{Pr}_z / \mathrm{Pr}_x$ and the diffusivity ratio $\kappa_a = \kappa_z / \kappa_x$. The critical wavenumber $m_c \in \mathbb{Z}^+$ is the integer that minimizes the expression in parentheses.

\begin{figure}[h]
  \centering
  \includegraphics[width=.6\textwidth,height=.4\textwidth]{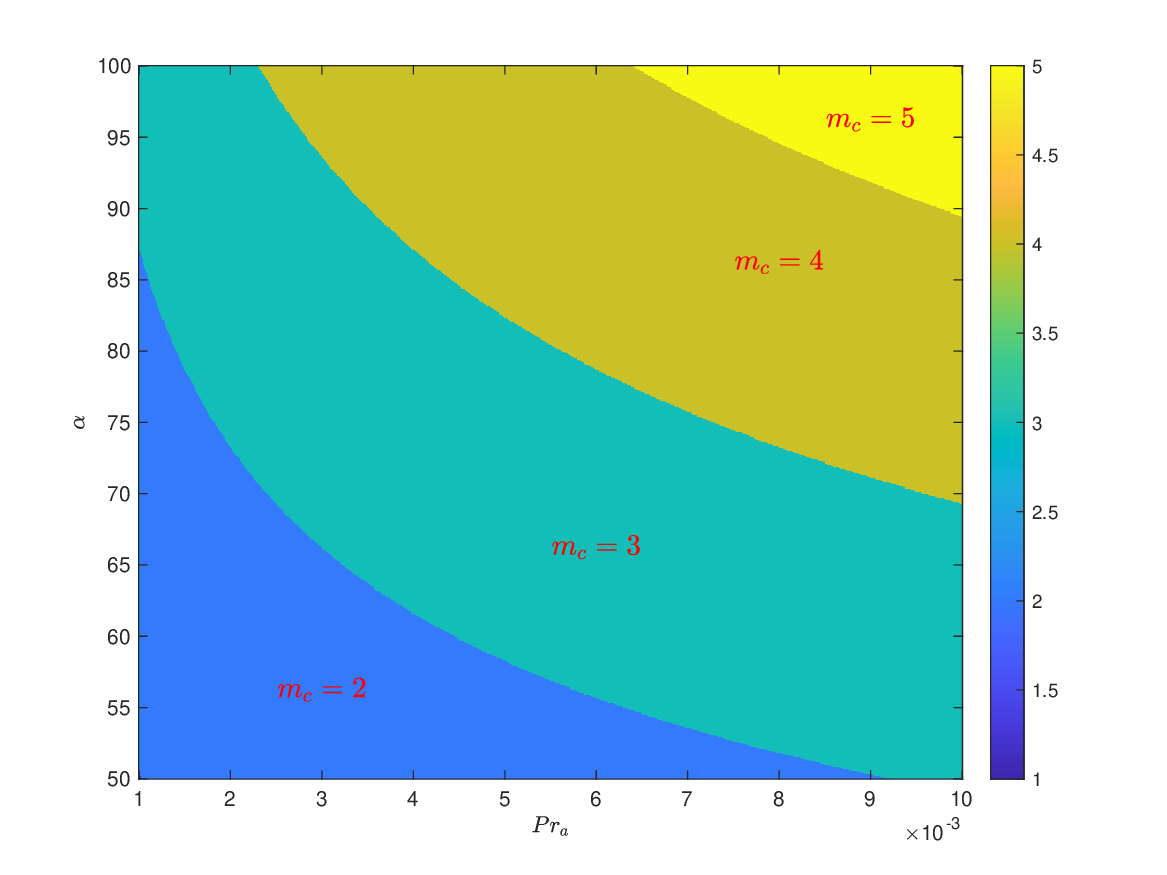}\\
  \caption{The cell number $m_c$ with $\kappa_a=0.01$ and $\alpha
  \in [1,40]$ and $\text{Pr}_a\in [0.01,1]$.
  }\label{Critical_p11}
\end{figure}

Using the expression
$\text{R} = \frac{H \sqrt{H \rho_0 g \beta |T_0 - T_1|}}{\kappa_x}$
 defined in \eqref{parameters-2} and  the dimensionless parameters in  \eqref{parameters}, the critical temperature difference $T_c$ is given by
\begin{align}\label{difference}
T_c = \frac{\pi^4 \nu_z \kappa_x}{H^3 \rho_0 g \beta}
\left( \frac{\kappa_z\nu_x}{\nu_z\kappa_x} + 1 \right)
+ \frac{\pi^4 \nu_z \kappa_x}{H^3 \rho_0 g \beta}
\left( \frac{4 m_c^2}{\text{Pr}_a \alpha^2} + \frac{\kappa_a \alpha^2}{4 m_c^2} \right).
\end{align}
We define the function
\[
F(\alpha, \mathrm{Pr}_{a}, \kappa_a, m) = \frac{4 m^2}{\mathrm{Pr}_a \alpha^2} + \frac{\kappa_a \alpha^2}{4 m^2},
\]
so that the critical wavenumber $m_c$ corresponds to the integer $m \in \mathbb{Z}^+$ that minimizes $F$.

It is not hard to see that $m_c$
is completely determined by $\alpha, \mathrm{Pr}_{a}$ and $\kappa_a$. For the fixed parameters $\kappa_a = 0.01$, $\alpha
  \in [50,100]$ and $\text{Pr}_a\in [0.001,0.01]$, the value of $m_c$ is determined numerically by minimizing $F$; the results are presented in Fig.~\ref{Critical_p11}.
Fig.~\ref{Critical_p11} reveals that $m_c$ decreases as both $\alpha$ and $\mathrm{Pr}_a$ decrease.   We can also see from \eqref{difference} that
 $T_c$ is proportional to the inverse of $\rho_0\beta$.

\subsubsection{Principle of Exchange Stability}\label{pescondition}
In the preceding section, we determined the critical Rayleigh number $\mathrm{R}_c$, or equivalently, the critical temperature difference $T_c = T_0 - T_1$. Having established this threshold for linear instability, we now turn to the principle of exchange of stabilities at $T_c$. To this end, we consider the eigenvalue problem associated with the linearization of the system \eqref{ondimensional-equations-2}:
\begin{align}\label{perturbation-31}
\begin{cases}
\begin{aligned}
&\text{Pr}_x\partial_{xx}v+
\text{Pr}_z\partial_{zz}v-\text{R}\partial_x \left(
\int_0^z\Theta(t,x,\xi)\,d\xi\right)-\partial_xq(t,x)=\beta v,
\end{aligned}\\
\begin{aligned}
&\partial_{xx} \Theta+
\kappa_{a}\partial_{zz} \Theta
-\text{R}\partial_x\left(\int_0^zv(t,x,\xi)\,d\xi \right)=\beta \Theta,
\end{aligned}\\
v_z|_{z=0}= v_z|_{z=1}=0, \quad
\Theta|_{z=0}=0,\quad \Theta|_{z=1}=0,\\
v, \Theta~\text{ are periodic in $x$ with period $\alpha$},\\
\int_0^1v(t,x,\xi)\,d\xi=0.
\end{cases}
\end{align}

To analyze this eigenvalue problem, we introduce the following function spaces:
\begin{align}
\begin{aligned}
\mathcal{Y}(D) &= \left\{ (v, \Theta) \in L^2(D) \,\middle|\,
v, \Theta \text{ are periodic in $x$ with period $\alpha$} \right\}, \\
\mathcal{X}_0(D) &= \left\{ (v, \Theta) \in \mathcal{Y}(D) \,\middle|\,
\int_0^1 v(x,z) \, dz = 0 \right\}, \\
\mathcal{X}_1(D) &= \left\{ (v, \Theta) \in \mathcal{X}_0(D) \cap H^1(D) \,\middle|\,
\Theta|_{z=0} = \Theta|_{z=1} = 0 \right\}, \\
\mathcal{X}_2(D) &= \left\{ (v, \Theta) \in \mathcal{X}_0(D) \cap H^2(D) \,\middle|\,
\partial_z v|_{z=0} = \partial_z v|_{z=1} = 0 \right\}.
\end{aligned}
\end{align}
Note that the sets $\left\{ e^{i m x / \alpha} \sin(n\pi z) \right\}_{m \in \mathbb{Z}, n \geq 1}$ and $\left\{ e^{i m x / \alpha} \cos(n\pi z) \right\}_{m \in \mathbb{Z}, n \geq 0}$ form two distinct orthonormal bases for $\mathcal{Y}(D)$. They are also
 two distinct orthonormal bases for $\mathcal{X}_1(D) $ and
 $\mathcal{X}_2(D)$.

We define the bounded linear projection operator $\mathcal{P}: L^2(D) \to L^2(D)$ by
\begin{align}
\mathcal{P} v = v - \int_0^1 v(x,z) \, dz,
\end{align}
which enforces the zero vertical mean condition. The eigenvalue problem can then be expressed in terms of a linear operator $\mathcal{L}_{\mathrm{R}}: \mathcal{X}_2(D) \to \mathcal{X}_0(D)$ defined as follows:
\begin{align}\label{linear-operator}
\begin{aligned}
\mathcal{L}_{\mathrm{R}} \psi &= \mathcal{A} \psi + \mathrm{R} \, \mathcal{B} \psi, \quad \psi = (v, \Theta), \\
\mathcal{A} \psi &=
\begin{pmatrix}
\mathcal{P} \left( \mathrm{Pr}_x \, \partial_{xx} v + \mathrm{Pr}_z \, \partial_{zz} v - \partial_x q \right) \\
\partial_{xx} \Theta + \kappa_{a} \, \partial_{zz} \Theta
\end{pmatrix}, \\
\mathcal{B} \psi &=
\begin{pmatrix}
- \mathcal{P} \, \partial_x \left( \int_0^z \Theta(t,x,\xi) \, d\xi \right) \\
\partial_x \left( \int_0^z v(t,x,\xi) \, d\xi \right)
\end{pmatrix}.
\end{aligned}
\end{align}

\begin{lemma}\label{eigen-all}
For the eigenvalue problem \eqref{perturbation-31}, the following properties hold:
\begin{enumerate}
    \item[(1)] Every eigenvalue $\beta$ is real.
    \item[(2)] There exists a constant $\beta_0 > 0$ such that every eigenvalue $\sigma$ of the modified problem
    \begin{align}\label{perturbation-313}
    \mathcal{L}_{\mathrm{R}} \psi - \beta_0 \psi = \sigma \psi
    \end{align}
    is negative. And, $(\mathcal{L}_{\mathrm{R}}-\beta_0I)^{-1}:
    \mathcal{X}_0(D)\to \mathcal{X}_0(D)$ is a compact operator.
\end{enumerate}
\end{lemma}

\begin{proof}
We first note that the operator $\mathcal{L}_{\mathrm{R}}$ satisfies $(\mathcal{L}_{\mathrm{R}} \psi, \psi) = (\psi, \mathcal{L}_{\mathrm{R}} \psi)$, which implies that $\mathcal{L}_{\mathrm{R}}$ is self-adjoint. Consequently, all eigenvalues $\beta$ of problem \eqref{perturbation-31} are real.
Furthermore, there exists a sufficiently large $\beta_0 > 0$ such that the quadratic form satisfies
\[
(\mathcal{L}_{\mathrm{R}} \psi, \psi) - \beta_0 (\psi, \psi) < 0 \quad \text{for all } \psi \neq 0.
\]
This inequality implies that all eigenvalues $\sigma$ of the shifted operator in \eqref{perturbation-313} are negative. It is not hard to show that 
$(\mathcal{L}_{\mathrm{R}}-\beta_0I)^{-1}:
    \mathcal{X}_0(D)\to \mathcal{X}_0(D)$ is a compact operator.
\end{proof}

\begin{remark}\label{remark51}
It follows from Lemma \ref{eigen-all} that the eigenvalue problem \eqref{perturbation-31} admits a countable sequence of real eigenvalues, each of finite multiplicity.
\end{remark}

Let us now show the existence of the largest eigenvalue for the problem 
\eqref{perturbation-31}. To this end, we define the first eigenvalue $\beta_1$ as 
follows:
\begin{align}\label{ve12-3}
-\beta_1 = \inf_{(u, \theta) \in \mathcal{X}_1(D)} \frac{E_1(u,\theta)-2\text{R}E_2(u, \theta)}{E_3(u, \theta)},
\end{align}
where $E_1$ and $E_2$ are given in \eqref{e12}, and the functionals $E_3$ are defined by:
\begin{align}\label{e13-2}
\begin{aligned}
E_1(u, \theta) &= \int_D \left( u^2 +  \theta^2 \right) dx\,dz.
\end{aligned}
\end{align}

\begin{lemma}
There exists a unique pair $(v, \Theta) \in H^2(D) \cap \mathcal{X}_1(D)$ such that
\[
-\beta_1 = \frac{E_1(v_1, \Theta_1) - 2\mathrm{R} E_2(v_1, \Theta_1)}{E_3(v_1, \Theta_1)},
\]
and the triple $(v_1, \Theta_1, \beta_1)$ satisfies the eigenvalue problem \eqref{perturbation-31}.
\end{lemma}
The proof of this lemma is analogous to that of Lemma~\ref{cri-eigenvalue}.

For any solution $(v, \Theta, \beta)$ of the problem \eqref{perturbation-31},
based on the definition of $-\beta_1$,
$(v, \Theta, \beta)$ satisfies 
\[
\beta =\frac{E_1(v, \Theta) - 2\mathrm{R} E_2(v, \Theta)}{E_3(v, \Theta)}
\leq -\frac{E_1(v_1, \Theta_1) - 2\mathrm{R} E_2(v_1, \Theta_1)}{E_3(v_1, \Theta_1)}
=\beta_1.
\]
This infers that $\beta_1$ is the largest eigenvalue of  the problem \eqref{perturbation-31}.

We now proceed to solve the eigenvalue problem \eqref{perturbation-31} exactly.
Substituting the Fourier expansions \eqref{T-e} into \eqref{perturbation-31} yields the algebraic system for each mode $(m, n)$:
\begin{align}\label{all-eign}
\begin{cases}
n\pi \left( \mathrm{Pr}_x \dfrac{4\pi^2 m^2}{\alpha^2} + \mathrm{Pr}_z n^2 \pi^2 + \beta \right) v_{m,n}
- \mathrm{R} \, \dfrac{i 2\pi m}{\alpha} \Theta_{m,n} = 0, \\[2ex]
\mathrm{R} \, \dfrac{i 2\pi m}{\alpha} v_{m,n} +
n\pi \left( \dfrac{4\pi^2 m^2}{\alpha^2} + \kappa_{a} n^2 \pi^2 + \beta \right) \Theta_{m,n} = 0.
\end{cases}
\end{align}
The system \eqref{all-eign} admits $(v_{m,n}, \Theta_{m,n}) \neq (0,0)$ if and only if its determinant vanishes:
\begin{align*}
n^2 \pi^2 \left( \mathrm{Pr}_x \frac{4\pi^2 m^2}{\alpha^2} + \mathrm{Pr}_z n^2 \pi^2 + \beta \right)
\left( \frac{4\pi^2 m^2}{\alpha^2} + \kappa_{a} n^2 \pi^2 + \beta \right)
= \mathrm{R}^2 \frac{4\pi^2 m^2}{\alpha^2}.
\end{align*}

Solving this characteristic equation for $\beta$, we obtain the eigenvalues:
\begin{align}\label{all-eigenvalue}
\beta = \beta_{m,n}^{\pm} = -A_{m,n} \pm B_{m,n},
\end{align}
where
\begin{align*}
A_{m,n} &= \frac{ (\mathrm{Pr}_x + 1) \dfrac{4\pi^2 m^2}{\alpha^2} + (\mathrm{Pr}_z + \kappa_a) n^2 \pi^2 }{2}, \\
B_{m,n} &= \sqrt{ A_{m,n}^2 - C_{m,n} }, \\
C_{m,n} &= \left( \mathrm{Pr}_x \frac{4\pi^2 m^2}{\alpha^2} + \mathrm{Pr}_z n^2 \pi^2 \right)
\left( \frac{4\pi^2 m^2}{\alpha^2} + \kappa_a n^2 \pi^2 \right)
- \mathrm{R}^2 \frac{4\pi^2 m^2}{\alpha^2 n^2 \pi^2}.
\end{align*}

Let the eigenvector corresponding to the eigenvalue $\beta_{m,n}^{\pm}$ in the system \eqref{all-eign} be denoted by
\begin{align}\label{eigenvector-0}
\begin{pmatrix}
v_{m,n}^{\pm} \\
\Theta_{m,n}^{\pm}
\end{pmatrix},
\end{align}
where the pair $(v_{m,n}^{\pm}, \Theta_{m,n}^{\pm})$ is a nontrivial solution of \eqref{all-eign} with $\beta = \beta_{m,n}^{\pm}$.

Owing to the fact that all coefficients in the problem \eqref{perturbation-31} are real, both the real and imaginary parts of any complex-valued solution are themselves real solutions. Consequently, for each eigenvalue $\beta_{m,n}^{\pm}$, we obtain two linearly independent real eigenvectors:
\begin{align}\label{eigenvector}
\begin{aligned}
\psi_{m,n}^{\pm,1} &= \Re \begin{pmatrix}
v_{m,n}^{\pm} \, e^{i 2\pi m x / \alpha} \cos(n\pi z) \\
\Theta_{m,n}^{\pm} \, e^{i 2\pi m x / \alpha} \sin(n\pi z)
\end{pmatrix}, \\
\psi_{m,n}^{\pm,2} &= \Im \begin{pmatrix}
v_{m,n}^{\pm} \, e^{i 2\pi m x / \alpha} \cos(n\pi z) \\
\Theta_{m,n}^{\pm} \, e^{i 2\pi m x / \alpha} \sin(n\pi z)
\end{pmatrix}.
\end{aligned}
\end{align}

Note that $\beta_{0,n}^{\pm} < 0$ for all $n$. For $m \neq 0$, we have $\beta_{m,n}^{-} < 0$, while $\beta_{m,n}^{+} = 0$ if and only if $C_{m,n} = 0$. The condition $C_{m,n} = 0$ is equivalent to
\[
\left( \mathrm{Pr}_x \frac{4\pi^2 m^2}{\alpha^2} + \mathrm{Pr}_z n^2 \pi^2 \right)
\left( \frac{4\pi^2 m^2}{\alpha^2} + \kappa_{a} n^2 \pi^2 \right)
= \mathrm{R}^2 \frac{4\pi^2 m^2}{\alpha^2 n^2 \pi^2},
\]
which recovers the relation \eqref{RRR} used for the determination of the critical Rayleigh number.

Therefore, when $\mathrm{R}$ is near the critical value $\mathrm{R}_c$, the first eigenvalue to cross zero is $\beta_{m_c,1}^{+}$, where $m_c$ is the critical wavenumber. This leads to the following lemma:

\begin{lemma}\label{pes-lemma}
For the eigenvalues $\beta_{m,n}^{\pm}$, the following conclusions hold:
\begin{align}\label{pes-condition}
\begin{cases}
\beta_{m,1}^{+}
\begin{cases}
>0, & \text{if } \mathrm{R} > \mathrm{R}_c, \\
=0, & \text{if } \mathrm{R} = \mathrm{R}_c, \\
<0, & \text{if } \mathrm{R} < \mathrm{R}_c,
\end{cases}
\quad \text{for } (m,1) = (m_c,1), \\[2ex]
\beta_{m,n}^{\pm} < 0, \quad \text{for } \mathrm{R} \leq \mathrm{R}_c,\ (m,n) \neq (m_c,1).
\end{cases}
\end{align}
\end{lemma}

In the vicinity of $\mathrm{R} = \mathrm{R}_c$, the first two eigenvectors $\psi_{m_c,1}^{+,1}$ and $\psi_{m_c,1}^{+,2}$ take the form:
\begin{align}\label{first-eigenvector}
\begin{aligned}
\psi_{m_c,1}^{+,1} &=
v_{m_c,1}^{+,1}
\begin{pmatrix}
\cos \left( \dfrac{2\pi m_c x}{\alpha} \right) \cos (\pi z) \\[1.5ex]
A_{m_c} \sin \left( \dfrac{2\pi m_c x}{\alpha} \right) \sin (\pi z)
\end{pmatrix}, \\[3ex]
\psi_{m_c,1}^{+,2} &=
v_{m_c,1}^{+,1}
\begin{pmatrix}
\sin \left( \dfrac{2\pi m_c x}{\alpha} \right) \cos (\pi z) \\[1.5ex]
-A_{m_c} \cos \left( \dfrac{2\pi m_c x}{\alpha} \right) \sin (\pi z)
\end{pmatrix},
\end{aligned}
\end{align}
where $v_{m_c,1}^{+,1}$ is a normalization constant determined by the condition 
\[
\|\psi_{m_c,1}^{+,1}\|_{L^2} = \|\psi_{m_c,1}^{+,2}\|_{L^2} = 1,
\] 
which and the coefficent $A_{m_c}$ are given by
\begin{align}\label{A-V}
\begin{aligned}
A_{m_c}=\frac{2\pi m_c\text{R}}{\alpha \pi\left(
\frac{4\pi^2m_c^2}{\alpha^2}
+\kappa_{a} \pi^2+\beta_{m_c,1}^{+}\right)}
\end{aligned},\quad
\left(v_{m_c,1}^{+,1}\right)^{2}=\left(
\frac{\alpha }{4}+\frac{\alpha  }{4} A_{m_c}^2\right)^{-2}.
\end{align}

\begin{remark}\label{remark52}
$\beta_{m_c,1}^{+}$ is the largest eigenvalue $\beta_1$ defined by 
\eqref{ve12-3}.
\end{remark}

\subsubsection{Global $L^2$-Stability for $0 < T_0 - T_1 < T_c$}

\begin{lemma}\label{lemma-grad-1-1-cases3}
Let $0 < T_0 - T_1 < T_c$. For any smooth solution $(v, \Theta)$ of the system \eqref{ondimensional-equations-2} with initial data $(v_0, \Theta_0) \in L^2(D)$, the following estimates hold:
\begin{align}\label{u21-1}
\begin{aligned}
&v \in L^{\infty}\big( (0,\infty); L^2(D) \big) \cap L^{2} \big( (0,\infty); H^1(D) \big), \\
&\Theta \in L^{\infty}\big( (0,\infty); L^2(D) \big) \cap L^{2} \big( (0,\infty); H^1(D) \big),
\end{aligned}
\end{align}
and there exists a constant $\lambda > 0$ such that
\begin{align}\label{u21-1-cases3}
\norm{v(t)}_{L^2(D)}^2 + \norm{\Theta(t)}_{L^2(D)}^2
\leq e^{-2\lambda t} \left( \norm{v_0}_{L^2(D)}^2 + \norm{\Theta_0}_{L^2(D)}^2 \right).
\end{align}
\end{lemma}

\begin{proof}
Under the condition $0 < T_0 - T_1 < T_c$, we multiply the first equation in \eqref{ondimensional-equations-2} by $v$ and the second by $\Theta$. Adding the resulting equations and integrating over $D$ yields:
\begin{align}\label{two-proof--1-cases3}
\begin{aligned}
\frac{d}{dt} \int_{D} \left( v^2 + \Theta^2 \right) dx\,dz
=& -2 \mathrm{Pr}_x \int_{D} \abs{\partial_x v}^2 dx\,dz
   -2 \mathrm{Pr}_z \int_{D} \abs{\partial_z v}^2 dx\,dz \\
  & -2 \int_{D} \abs{\partial_x \Theta}^2 dx\,dz
   -2 \kappa_{a} \int_{D} \abs{\partial_z \Theta}^2 dx\,dz \\
  & -4 \mathrm{R} \int_{D} \left( \Theta \int_0^z \partial_x v(t,x,\xi) \, d\xi \right) dx\,dz \\
\leq& \; 2 \beta_1 \int_{D} \left( v^2 + \Theta^2 \right) dx\,dz,
\end{aligned}
\end{align}
where the inequality follows from the principle of exchange of stabilities (see \eqref{pes-condition}), and $\beta_1$ is the maximum eigenvalue, which is negative.
Integrating inequality \eqref{two-proof--1-cases3} in time, we obtain 
\eqref{u21-1} and
the exponential decay estimate \eqref{u21-1-cases3} with $\lambda = -\beta_1 > 0$
when $0<T_0-T_1<T_c$.
\end{proof}

\subsubsection{Global $H^1$-Stability for $T_0 - T_1 < T_c$}

\begin{lemma}\label{lemma-grad-1-cases3-h1}
Let $0 < T_0 - T_1 < T_c$. For any smooth solution $(v, \Theta)$ of the system \eqref{ondimensional-equations-2} with initial data $(v_0, \Theta_0) \in H^1(D) \times L^2(D)$, the following estimates hold:
\begin{align}\label{infty}
\partial_z v \in L^{\infty}\big( (0,\infty); L^2(D) \big), \quad
\nabla \partial_z v \in L^2\big( (0,\infty); L^2(D) \big),
\end{align}
Moreover, there exists a constant $\lambda > 0$ such that
\begin{align}\label{cases-3-vz}
\norm{\partial_z v(t)}_{L^2}^2 \leq C e^{-2\lambda t} \left( \norm{(v_0, \Theta_0)}_{L^2(D)}^2 + \norm{\partial_z v_0}_{L^2}^2 \right).
\end{align}
\end{lemma}
\begin{proof}
The proof follows using the approach developed in the proof of Lemma~\ref{lemma-grad-1-2}.
\end{proof}

\begin{lemma}\label{lemma-grad-1-2-cases3}
Let $0 < T_0 - T_1 < T_c$. For any smooth solution $(v, \Theta)$ of the system \eqref{ondimensional-equations-2} with initial data $(v_0, \Theta_0) \in H^1(D) \times L^2(D)$, we have
\begin{align}\label{infty-cases3-3}
\partial_x v \in L^{\infty}\big( (0,\infty); L^2(D) \big) \cap L^2\big( (0,\infty); H^1(D) \big).
\end{align}
\end{lemma}
\begin{proof}
The proof follows using the approach developed in the proof of Lemma~\ref{lemma-grad-1-3}.
\end{proof}

\begin{lemma}\label{lemma-grad-1-3-cases}
Let $0 < T_0 - T_1 < T_c$. For any smooth solution $(v, \Theta)$ of the system \eqref{ondimensional-equations-2} with initial data $(v_0, \Theta_0) \in H^1(D) \times L^2(D)$, the following estimates hold:
\begin{align}\label{infty-thetaz}
\partial_z \Theta \in L^{\infty}\big( (0,\infty); L^2(D) \big) \cap L^2\big( (0,\infty); H^1(D) \big).
\end{align}
Moreover, there exists a constant $\lambda > 0$ such that
\begin{align}\label{decaythetaz-cases3}
\norm{\partial_z \Theta(t)}_{L^2}^2 \leq C C_0 e^{-2\lambda t} \norm{(v_0, \Theta_0)}_{H^1(D)}^2,
\end{align}
where the constant $C_0$ depends on $\norm{v_0}_{H^1(D)}$ and $\norm{\Theta_0}_{L^2(D)}$.
\end{lemma}
\begin{proof}
The proof follows using the approach developed in the proof of Lemma~\ref{lemma-grad-1-4}.
\end{proof}

\begin{lemma}\label{lemma-grad-1-4-cases3}
Let $0 < T_0 - T_1 < T_c$. For any smooth solution $(v, \Theta)$ of the system \eqref{ondimensional-equations-2} with initial data $(v_0, \Theta_0) \in H^1(D) \times L^2(D)$, the following estimates hold:
\begin{align}\label{infty-cases3}
(\partial_x v, \partial_x \Theta) \in L^{\infty}\big( (0,\infty); L^2(D) \big) \cap L^2\big( (0,\infty); H^1(D) \big).
\end{align}
Moreover, there exists a constant $\lambda > 0$ such that
\begin{align}\label{decayvxthetax-cases-3}
\norm{\partial_x v(t)}_{L^2}^2 + \norm{\partial_x \Theta(t)}_{L^2}^2 \leq C C_0 e^{-2\lambda t} \norm{(v_0, \Theta_0)}_{H^1(D)}^2,
\end{align}
where the constant $C_0$ depends on $\norm{v_0}_{H^1(D)}$ and $\norm{\Theta_0}_{L^2(D)}$.
\end{lemma}

\begin{proof}
For the case $0 < T_0 - T_1 < T_c$, we multiply the $\Theta$-equation and $v$-equation in \eqref{ondimensional-equations-2} by $-\partial_{xx} \Theta$ and $-\partial_{xx} v$, respectively. After integration by parts, we obtain:
\begin{align}\label{decayusing}
\begin{aligned}
&\frac{d}{dt} \left( \norm{\partial_x v}_{L^2}^2 + \norm{\partial_x \Theta}_{L^2}^2 \right)
+ 2\mathrm{Pr}_x \int_D |\partial_{xx} v|^2 + 2\mathrm{Pr}_z \int_D |\partial_{xz} v|^2 \\
&\quad + 2 \int_D |\partial_{xx} \Theta|^2 + 2\kappa_a \int_D |\partial_{xz} \Theta|^2 \\
&= -4\mathrm{R} \int_D \partial_x \Theta \left( \int_0^z \partial_{xx} v(t,x,\xi) \, d\xi \right) dx\,dz \\
&\quad + 2 \int_D \left( v \partial_x v - \left( \int_0^z \partial_x v(t,x,\xi) \, d\xi \right) \partial_z v \right) \partial_{xx} v \, dx\,dz \\
&\quad + 2 \int_D \left( v \partial_x \Theta - \left( \int_0^z \partial_x v(t,x,\xi) \, d\xi \right) \partial_z \Theta \right) \partial_{xx} \Theta \, dx\,dz \\
&:= J_0 + H_1 + H_2 + S_1 + S_2.
\end{aligned}
\end{align}

For the term $J_0$, we estimate:
\begin{align}
J_0 \leq C \norm{\partial_x \Theta}_{L^2} \norm{\partial_{xx} v}_{L^2}
\leq \epsilon \norm{\partial_{xx} v}_{L^2}^2 + C_{\epsilon} \norm{\partial_x \Theta}_{L^2}^2.
\end{align}
Following a similar approach as in the proof of Lemma \ref{lemma-grad-1-5}, we deduce from \eqref{S1-S2} that:
\begin{align}
\begin{aligned}
H_1 + H_2 + S_1 + S_2 &= \epsilon \norm{\partial_{xx} v}_{L^2}^2 + \epsilon \norm{\partial_{xx} \Theta}_{L^2}^2 
+ \epsilon \norm{\partial_{zz} \Theta}_{L^2}^2 + \epsilon \norm{\partial_{zz} v}_{L^2}^2 \\
&\quad + C_{\epsilon} \norm{v}_{L^2}^2 \norm{\partial_z v}_{L^2}^2 \norm{\partial_x v}_{L^2}^2 \\
&\quad + C_{\epsilon} \norm{v}_{L^2} \norm{\partial_z v}_{L^2} \norm{\partial_x \Theta}_{L^2} \norm{\partial_{xx} \Theta}_{L^2} \\
&\quad + C_{\epsilon} \norm{\partial_z v}_{L^2}^2 \norm{\partial_{zz} v}_{L^2}^2 \norm{\partial_x v}_{L^2}^2 \\
&\quad + C_{\epsilon} \norm{\partial_z \Theta}_{L^2}^2 \norm{\partial_x v}_{L^2}^2 \norm{\partial_{zz} \Theta}_{L^2}^2.
\end{aligned}
\end{align}

Taking $\epsilon > 0$ sufficiently small, we arrive at:
\begin{align}\label{vthetagradient}
\begin{aligned}
&\frac{d}{dt} \left( \norm{\partial_x v}_{L^2}^2 + \norm{\partial_x \Theta}_{L^2}^2 \right)
+ \mathrm{Pr}_x \int_D |\partial_{xx} v|^2 + \mathrm{Pr}_z \int_D |\partial_{xz} v|^2 \\
&\quad + \int_D |\partial_{xx} \Theta|^2 + \kappa_a \int_D |\partial_{xz} \Theta|^2 \leq r(t),
\end{aligned}
\end{align}
where $r(t)$ is defined by:
\begin{align}
\begin{aligned}
r(t) := &\; C \norm{\partial_x \Theta}_{L^2}^2 
+ C \norm{v}_{L^2}^2 \norm{\partial_z v}_{L^2}^2 \norm{\partial_x v}_{L^2}^2 \\
& + C \norm{v}_{L^2}^2 \norm{\partial_z v}_{L^2}^2 \norm{\partial_x \Theta}_{L^2}^2 \\
& + C \norm{\partial_z v}_{L^2}^2 \norm{\partial_{zz} v}_{L^2}^2 \norm{\partial_x v}_{L^2}^2 \\
& + C \norm{\partial_z \Theta}_{L^2}^2 \norm{\partial_x v}_{L^2}^2 \norm{\partial_{zz} \Theta}_{L^2}^2.
\end{aligned}
\end{align}
The estimate \eqref{infty-cases3} is then obtained by integrating \eqref{vthetagradient} in time and using the fact that $\int_0^{\infty} r(t) \, dt < \infty$.
Returning to \eqref{decayusing}, we also have:
\begin{align}\label{decayusing-2}
\begin{aligned}
&\frac{d}{dt} \left( \norm{\partial_x v}_{L^2}^2 + \norm{\partial_x \Theta}_{L^2}^2 \right) \\
&= -2\mathrm{Pr}_x \int_D |\partial_{xx} v|^2 - 2\mathrm{Pr}_z \int_D |\partial_{xz} v|^2 
- 2 \int_D |\partial_{xx} \Theta|^2 - 2\kappa_a \int_D |\partial_{xz} \Theta|^2 \\
&\quad - 4\mathrm{R} \int_D \partial_x \Theta \left( \int_0^z \partial_{xx} v(t,x,\xi) \, d\xi \right) dx\,dz \\
&\quad + 2 \int_D \left( v \partial_x v - \left( \int_0^z \partial_x v(t,x,\xi) \, d\xi \right) \partial_z v \right) \partial_{xx} v \, dx\,dz \\
&\quad + 2 \int_D \left( v \partial_x \Theta - \left( \int_0^z \partial_x v(t,x,\xi) \, d\xi \right) \partial_z \Theta \right) \partial_{xx} \Theta \, dx\,dz \\
&\leq \beta \left( \norm{\partial_x v}_{L^2}^2 + \norm{\partial_x \Theta}_{L^2}^2 \right) + H_1(t) + H_2(t) + S_1(t) + S_2(t).
\end{aligned}
\end{align}

Let $\eta$ be a positive parameter satisfying $0 < \eta < \lambda$, with $\lambda$ being the stability constant as introduced in \eqref{cases-3-vz} and \eqref{decaythetaz-cases3}. 
We now estimate the time integrals of the nonlinear terms $H_1$, $H_2$, $S_1$, and $S_2$ with an exponential weight $e^{\eta \tau}$.
For $H_1$, we have
\begin{align}
\begin{aligned}
&\int_0^t e^{\eta \tau} H_1(\tau) \, d\tau 
\\&= 2 \int_0^t e^{\eta \tau} \left( \int_D (v \partial_x v) \partial_{xx} v \, dx\,dz \right) d\tau \\
&\leq C \left( \int_0^t \norm{\partial_{xx} v}_{L^2}^2 \, d\tau \right)^{\frac{1}{2}}
   \left( \int_0^t e^{2\eta \tau} \norm{v}_{L^2} \norm{\partial_z v}_{L^2} \norm{\partial_x v}_{L^2} \norm{\partial_{xx} v}_{L^2} \, d\tau \right)^{\frac{1}{2}} \\
&\leq C_0 \left( \int_0^t \norm{\partial_{xx} v}_{L^2}^2 \, d\tau \right)^{\frac{1}{2}}
   \left( \int_0^t e^{(-\lambda + 2\eta) \tau} \norm{\partial_x v}_{L^2} \norm{\partial_{xx} v}_{L^2} \, d\tau \right)^{\frac{1}{2}},
\end{aligned}
\end{align}
where we have used the previously established decay estimates for $\norm{v}_{L^2}$ and $\norm{\partial_z v}_{L^2}$.

For $H_2$, one can get
\begin{align}
\begin{aligned}
&\int_0^t e^{\eta \tau} H_2(\tau) \, d\tau 
\\&= 2 \int_0^t e^{\eta \tau} \left( \int_D \left( \int_0^z \partial_x v(t,x,\xi) \, d\xi \right) \partial_z v \, \partial_{xx} v \, dx\,dz \right) d\tau \\
&\leq C \left( \int_0^t \norm{\partial_{xx} v}_{L^2}^2 \, d\tau \right)^{\frac{1}{2}}
   \left( \int_0^t e^{2\eta \tau} \norm{ \left( \int_0^z \partial_x v(t,x,\xi) \, d\xi \right) \partial_z v }_{L^2}^2 \, d\tau \right)^{\frac{1}{2}} \\
&\leq C \left( \int_0^t \norm{\partial_{xx} v}_{L^2}^2 \, d\tau \right)^{\frac{1}{2}}
   \left( \int_0^t e^{2\eta \tau} \norm{\partial_z v}_{L^2} \norm{\partial_x v}_{L^2} \norm{\partial_{xx} v}_{L^2} \norm{\partial_{xz} v}_{L^2} \, d\tau \right)^{\frac{1}{2}} \\
&\leq C \left( \int_0^t \norm{\partial_{xx} v}_{L^2}^2 \, d\tau \right)^{\frac{1}{2}}
   \left( \int_0^t e^{(-2\lambda + 2\eta) \tau} \norm{\partial_x v}_{L^2} \norm{\partial_{xx} v}_{L^2} \norm{\partial_{xz} v}_{L^2} \, d\tau \right)^{\frac{1}{2}}.
\end{aligned}
\end{align}

For the terms $S_1$ and $S_2$, we get
\begin{align}
\begin{aligned}
&\int_0^t e^{\eta \tau} S_1(\tau) \, d\tau 
\\&\leq C \left( \int_0^t \norm{\partial_{xx} \Theta}_{L^2}^2 \, d\tau \right)^{\frac{1}{2}}
   \left( \int_0^t e^{2\eta \tau} \norm{v \partial_x \Theta}_{L^2}^2 \, d\tau \right)^{\frac{1}{2}} \\
&\leq C \left( \int_0^t \norm{\partial_{xx} \Theta}_{L^2}^2 \, d\tau \right)^{\frac{1}{2}}
   \left( \int_0^t e^{2\eta \tau} \norm{v}_{L^2} \norm{\partial_z v}_{L^2} \norm{\partial_x \Theta}_{L^2} \norm{\partial_{xx} \Theta}_{L^2} \, d\tau \right)^{\frac{1}{2}} \\
&\leq C_0 \left( \int_0^t \norm{\partial_{xx} \Theta}_{L^2}^2 \, d\tau \right)^{\frac{1}{2}}
   \left( \int_0^t e^{(-2\lambda + 2\eta) \tau} \norm{\partial_x \Theta}_{L^2} \norm{\partial_{xx} \Theta}_{L^2} \, d\tau \right)^{\frac{1}{2}},
\end{aligned}
\end{align}
and
\begin{align}
\begin{aligned}
&\int_0^t e^{\eta \tau} S_2(\tau) \, d\tau 
\\&\leq C \left( \int_0^t \norm{\partial_{xx} \Theta}_{L^2}^2 \, d\tau \right)^{\frac{1}{2}}
   \left( \int_0^t e^{2\eta \tau} \norm{ \left( \int_0^z \partial_x v(t,x,\xi) \, d\xi \right) \partial_z \Theta }_{L^2}^2 \, d\tau \right)^{\frac{1}{2}} \\
&\leq C \left( \int_0^t \norm{\partial_{xx} \Theta}_{L^2}^2 \, d\tau \right)^{\frac{1}{2}}
   \left( \int_0^t e^{2\eta \tau} \norm{\partial_x v}_{L^2} \norm{\partial_{xx} v}_{L^2} \norm{\partial_z \Theta}_{L^2} \norm{\partial_{zz} \Theta}_{L^2} \, d\tau \right)^{\frac{1}{2}} \\
&\leq C \left( \int_0^t \norm{\partial_{xx} \Theta}_{L^2}^2 \, d\tau \right)^{\frac{1}{2}}
   \left( \int_0^t e^{(-\lambda + 2\eta) \tau} \norm{\partial_x v}_{L^2} \norm{\partial_{xx} v}_{L^2} \norm{\partial_{zz} \Theta}_{L^2} \, d\tau \right)^{\frac{1}{2}}.
\end{aligned}
\end{align}

Returning to inequality \eqref{decayusing-2}, we derive the differential form:
\begin{align}\label{decayusing-3}
\begin{aligned}
\frac{d}{dt} \left( \norm{\partial_x v}_{L^2}^2 + \norm{\partial_x \Theta}_{L^2}^2 \right)
- \beta \left( \norm{\partial_x v}_{L^2}^2 + \norm{\partial_x \Theta}_{L^2}^2 \right)
\leq e^{-\eta t} M(t),
\end{aligned}
\end{align}
where the function $M(t) = e^{\eta t} \left( H_1(t) + H_2(t) + S_1(t) + S_2(t) \right)$ belongs to $L^1(0,\infty)$.
Finally, applying Lemma \ref{general-inequality1} to \eqref{decayusing-3} yields the decay estimate \eqref{decayvxthetax-cases-3} holds.

\end{proof}

\subsubsection{Global $H^2$-Stability for $T_0 - T_1 < T_c$}

\begin{lemma}\label{lemma-grad-1-6-cases3}
Let $0 < T_0 - T_1 < T_c$. For any smooth solution $(v, \Theta)$ of the system \eqref{ondimensional-equations-2} with initial data $(v_0, \Theta_0) \in H^2(D)$, the following estimates hold:
\begin{align}\label{vh2-infty-16-cases-3}
\partial_z(\nabla v, \nabla \Theta) \in L^{\infty}\big( (0,\infty); L^2(D) \big), \quad
\Delta(\partial_z v, \partial_z \Theta) \in L^2\big( (0,\infty); L^2(D) \big).
\end{align}
Moreover, there exists a constant $\lambda > 0$ such that
\begin{align}\label{decy-vxz-vzz-16-cases}
\norm{\partial_z (\nabla v, \nabla \Theta)(t)}_{L^2}^2 \leq C C_0 e^{-2\lambda t} \norm{(v_0, \Theta_0)}_{H^2(D)}^2,
\end{align}
where the constant $C_0$ depends on $\norm{v_0}_{H^1(D)}$ and $\norm{\Theta_0}_{L^2(D)}$.
\end{lemma}
\begin{proof}
The proof follows using the approach developed in the proof of Lemma~\ref{lemma-grad-1-6}.
\end{proof}

\begin{lemma}\label{lemma-grad-1-7-cases3}
Let $0 < T_0 - T_1 < T_c$. For any smooth solution $(v, \Theta)$ of the system \eqref{ondimensional-equations-2} with initial data $(v_0, \Theta_0) \in H^2(D)$, the following estimates hold:
\begin{align}\label{infty-vxx-delta-17-cases3}
(\partial_{xx} v, \partial_{xx} \Theta) \in L^{\infty}\big( (0,\infty); L^2(D) \big), \quad
\Delta(\partial_x v, \partial_x \Theta) \in L^2\big( (0,\infty); L^2(D) \big).
\end{align}
Moreover, there exists a constant $\lambda > 0$ such that
\begin{align}\label{decy-vxx-thetaxx-17-cases-3}
\norm{(\partial_{xx} v, \partial_{xx} \Theta)(t)}_{L^2}^2 \leq C C_0 e^{-2\lambda t} \norm{(v_0, \Theta_0)}_{H^2(D)}^2,
\end{align}
where the constant $C_0$ depends on $\norm{v_0}_{H^1(D)}$ and $\norm{\Theta_0}_{L^2(D)}$.
\end{lemma}
\begin{proof}
The proof follows using the approach developed in the proof of Lemma~\ref{lemma-grad-1-7}.
\end{proof}

\begin{lemma}\label{lemma-grad-1-8-cases3}
Let $0 < T_0 - T_1 < T_c$. For any smooth solution $(v, \Theta)$ of the system \eqref{ondimensional-equations-2} with initial data $(v_0, \Theta_0) \in H^2(D)$, there exists a constant $\lambda > 0$ such that
\begin{align}\label{decy-pressure-temporal-17}
\norm{\partial_x q(t)}_{L^2}^2 + \norm{\partial_t (v, \Theta)(t)}_{L^2}^2 \leq C C_0 e^{-2\lambda t} \norm{(v_0, \Theta_0)}_{H^2(D)}^2,
\end{align}
where the constant $C_0$ depends on $\norm{v_0}_{H^1(D)}$ and $\norm{\Theta_0}_{L^2(D)}$.
\end{lemma}
\begin{proof}
The proof follows using the approach developed in the proof of Lemma~\ref{lemma-grad-1-8}.
\end{proof}

\section{Proof of Theorem \ref{critical-stability}}\label{proof-theorem22}

Theorem~\ref{critical-stability} follows from Corollary~\ref{colarry}, which is established by Lemmas~\ref{attractor5} and~\ref{attractor6}.

\subsection{Existence of Global Attractor} 

\begin{lemma}\label{attractor1}
Let $(v, \Theta)$ be a smooth solution of the system \eqref{ondimensional-equations-2} with initial data $(v_0, \Theta_0) \in L^2(D) \times L^2(D)$. Then there exists a constant $\lambda > 0$ such that the following estimates hold:
\begin{align}\label{l2bouned-1}
&\norm{v(t)}_{L^2}^2 + \norm{\Theta(t)}_{L^2}^2 \leq C_0 e^{-\lambda t} + C,
\\&\label{l2bouned-2}
\int_t^{t+1} \left( \norm{\nabla v(\tau)}_{L^2}^2 + \norm{\nabla \Theta(\tau)}_{L^2}^2 \right) d\tau \leq C_0 e^{-\lambda t} + C,
\end{align}
where $C_0$ depends on $\norm{v_0}_{L^2}$ and $\norm{\Theta_0}_{L^2}$, while $C > 0$ is a constant independent of the initial data.
\end{lemma}

\begin{proof}
Consider the case $T_0 >T_1$. Define the perturbation variable
\[
W = \Theta - \mathrm{R} z.
\]
It follows from \eqref{ondimensional-equations-2} that $W$ satisfies
\begin{align*}
\begin{cases}
\begin{aligned}
\partial_t W = &\partial_{xx} W + \kappa_a \partial_{zz} W 
- v \partial_x W + \left( \int_0^z \partial_x v(t,x,\xi) \, d\xi \right) \partial_z W,
\end{aligned} \\
W|_{z=0} = 0, \quad W|_{z=1} = -\mathrm{R}, \\
W \text{ is periodic in $x$ with period $\alpha$}.
\end{cases}
\end{align*}

For the positive part $W_+$ and the negative part $(W + \mathrm{R})_-$, we derive the energy identities:
\begin{align*}
\begin{aligned}
&\begin{aligned}
&\frac{d}{dt} \int_D |W_+|^2 \, dx\,dz 
+ 2 \int_D |\partial_x W_+|^2 \, dx\,dz 
\\&+ 2\kappa_a \int_D |\partial_z W_+|^2 \, dx\,dz = 0, 
\end{aligned}
\\
&\begin{aligned}
&\frac{d}{dt} \int_D |(W + \mathrm{R})_-|^2 \, dx\,dz 
+ 2 \int_D |\partial_x (W + \mathrm{R})_-|^2 \, dx\,dz 
\\&+ 2\kappa_a \int_D |\partial_z (W + \mathrm{R})_-|^2 \, dx\,dz = 0.
\end{aligned}
\end{aligned}
\end{align*}
These imply the existence of $\lambda > 0$ such that
\begin{align}
\begin{aligned}
&\norm{W_+(t)}_{L^2}^2 \leq e^{-\lambda t} \norm{W_+(0)}_{L^2}^2, \\
&\norm{(W + \mathrm{R})_-(t)}_{L^2}^2 \leq e^{-\lambda t} \norm{(W + \mathrm{R})_-(0)}_{L^2}^2.
\end{aligned}
\end{align}

Now, decompose $W$ as
\[
W = W_+ - (W + \mathrm{R})_- + Q.
\]
From this decomposition, we deduce that
\[
-\mathrm{R} \leq Q \leq 0.
\]
Since $\Theta = W_+ - (W + \mathrm{R})_- + Q + \mathrm{R} z$, it follows that
\begin{align}\label{thetal2}
\norm{\Theta(t)}_{L^2}^2 \leq C_0 e^{-\lambda t} + C,
\end{align}
where the constant $C$ depends only on $\mathrm{R}$ and $|D|$.

For the velocity field $v$ under the condition $T_0 >T_1 $, we have
\begin{align*}
\begin{aligned}
&\frac{d}{dt} \int_D v^2 \, dx\,dz 
+ \mathrm{Pr}_x \int_D |\partial_x v|^2 \, dx\,dz 
\\&+ 2\mathrm{Pr}_z \int_D |\partial_z v|^2 \, dx\,dz 
\leq C \int_D \Theta^2 \, dx\,dz.
\end{aligned}
\end{align*}
Applying Lemma \ref{general-inequality1} together with estimate \eqref{thetal2}, we conclude that there exists $\lambda > 0$ such that
\begin{align*}
\norm{v(t)}_{L^2}^2 + \norm{\Theta(t)}_{L^2}^2 \leq C_0 e^{-\lambda t} + C.
\end{align*}

To establish \eqref{l2bouned-2}, consider the total energy balance:
\begin{align}\label{two-proof--1-cases3-attractor}
\begin{aligned}
\frac{d}{dt} \int_D \left( v^2 + \Theta^2 \right) dx\,dz 
=& -\mathrm{Pr}_x \int_D |\partial_x v|^2 \, dx\,dz 
- 2\mathrm{Pr}_z \int_D |\partial_z v|^2 \, dx\,dz \\
& - 2 \int_D |\partial_x \Theta|^2 \, dx\,dz 
- 2\kappa_a \int_D |\partial_z \Theta|^2 \, dx\,dz \\
& + C \int_D \Theta^2 \, dx\,dz.
\end{aligned}
\end{align}
Integrating this inequality from $t$ to $t+1$ and using the previously obtained bounds yields:
\begin{align*}
\begin{aligned}
&\mathrm{Pr}_x \int_t^{t+1} \norm{\partial_x v(\tau)}_{L^2}^2 \, d\tau 
+ 2\mathrm{Pr}_z \int_t^{t+1} \norm{\partial_z v(\tau)}_{L^2}^2 \, d\tau \\
&\quad + \int_t^{t+1} \norm{\partial_x \Theta(\tau)}_{L^2}^2 \, d\tau 
+ 2\kappa_a \int_t^{t+1} \norm{\partial_z \Theta(\tau)}_{L^2}^2 \, d\tau \\
&\leq \norm{v(t)}_{L^2}^2 + \norm{\Theta(t)}_{L^2}^2 
+ \frac{C_0}{\lambda} \left( e^{-\lambda t} - e^{-\lambda (t+1)} \right) + C \\
&\leq C_0 e^{-\lambda t} + C.
\end{aligned}
\end{align*}
This completes the proof of the lemma.
\end{proof}

\begin{lemma}\label{attractor2}
Let $(v, \Theta)$ be a smooth solution of the system \eqref{ondimensional-equations-2} with initial data $(v_0, \Theta_0) \in H^1(D) \times L^2(D)$. Then there exists a constant $\lambda > 0$ such that
\begin{align}\label{l2bouned}
\norm{\partial_z v(t)}_{L^2}^2 \leq C_0 e^{-\lambda t} + C,
\end{align}
where $C_0$ depends on $\norm{v_0}_{H^1}$ and $\norm{\Theta_0}_{L^2}$, while $C > 0$ is independent of the initial data.
Furthermore, for any $\epsilon > 0$ and $M > 0$, there exists $t_0(M, \epsilon) > 0$ such that for all $t \geq t_0(M, \epsilon)$,
\begin{align}\label{l2bouned-23-v}
\int_t^{t+1} \norm{\partial_z \nabla v(\tau)}_{L^2}^2 \, d\tau \leq C_{\epsilon},
\end{align}
where $C_{\epsilon}$ is a constant depending on $\epsilon$ but independent of $M$.
\end{lemma}

\begin{proof}
The proof follows using the approach developed in the proof of Lemma~\ref{lemma-grad-1-2}.
\end{proof}

\begin{lemma}\label{attractor3}
Let $M > 0$ and $\epsilon > 0$ be given. For any smooth solution $(v, \Theta)$ of the system \eqref{ondimensional-equations-2} with initial data $(v_0, \Theta_0) \in H^1(D) \times L^2(D)$ satisfying
\[
\norm{v_0}_{H^1}^2 + \norm{\Theta_0}_{L^2}^2 \leq M,
\]
there exists $t_0(M, \epsilon) > 0$ such that for all $t > t_0(M, \epsilon)$,
\begin{align}\label{l2bouned-22}
\norm{\partial_z \Theta(t)}_{L^2}^2 \leq C_{\epsilon},
\end{align}
and
\begin{align}\label{l2bouned-23}
\int_t^{t+1} \norm{\partial_z \nabla \Theta(\tau)}_{L^2}^2 \, d\tau \leq C_{\epsilon},
\end{align}
where $C_{\epsilon}$ is a constant depending on $\epsilon$ but independent of $M$.
\end{lemma}

\begin{proof}
Following the method used in the proof of Lemma~\ref{lemma-grad-1-4}, there exists $\lambda > 0$ such that
\begin{align*}
\begin{aligned}
\frac{d}{dt} \norm{\partial_z \Theta}_{L^2}^2
+ \lambda \norm{\partial_z \Theta}_{L^2}^2
\leq &\; C \norm{v}_{L^2}^2
+ C \norm{v}_{L^2}^2 \norm{\partial_x \Theta}_{L^2}^2 \norm{\partial_z v}_{L^2}^2 \\
& + C \norm{v}_{L^2}^2 \norm{\partial_z v}_{L^2}^2 \norm{\partial_z \Theta}_{L^2}^2.
\end{aligned}
\end{align*}
By Lemmas~\ref{attractor1} and~\ref{attractor2}, for any $\epsilon > 0$, there exists $t_0(M, \epsilon) > 0$ such that for all $t \geq t_0(M, \epsilon)$,
\begin{align*}
\begin{aligned}
\norm{v(t)}_{L^2}^2 \leq C + \epsilon, \quad
\norm{\partial_z v(t)}_{L^2}^2 \leq C + \epsilon.
\end{aligned}
\end{align*}
Hence, for $t \geq t_0(M, \epsilon)$, we obtain
\begin{align*}
\begin{aligned}
\frac{d}{dt} \norm{\partial_z \Theta}_{L^2}^2
+ \lambda \norm{\partial_z \Theta}_{L^2}^2
\leq &\; C(C + \epsilon)
+ C(C + \epsilon)^2 \norm{\partial_x \Theta}_{L^2}^2 \\
& + C(C + \epsilon)^2 \norm{\partial_z \Theta}_{L^2}^2.
\end{aligned}
\end{align*}

Let us define
\[
f_1(t) = C(C + \epsilon)^2, \quad 
f_2(t) = C(C + \epsilon) + C(C + \epsilon)^2 \norm{\partial_x \Theta}_{L^2}^2, \quad 
h(t) = \norm{\partial_z \Theta}_{L^2}^2.
\]
Then the inequality becomes
\begin{align*}
\frac{dh}{dt} \leq f_1(t) h + f_2(t).
\end{align*}

For $t \leq s \leq t + 1$, we deduce that
\begin{align}\label{ke-inequality}
\begin{aligned}
h(t+1) &\leq h(s) e^{C(C + \epsilon)} + e^{C(C + \epsilon)} \int_s^{t+1} f_2(\tau) \, d\tau \\
&\leq h(s) e^{C(C + \epsilon)} + e^{C(C + \epsilon)} \left( C(C + \epsilon) + C(C + \epsilon)^3 \right),
\end{aligned}
\end{align}
where we have used the estimate
\[
\int_t^{t+1} f_2(\tau) \, d\tau \leq C(C + \epsilon) + C(C + \epsilon)^3, \quad t \geq t_0(M, \epsilon).
\]

Integrating inequality \eqref{ke-inequality} with respect to $s$ from $t$ to $t+1$ yields
\begin{align*}
\begin{aligned}
h(t+1) &\leq e^{C(C + \epsilon)} \int_t^{t+1} h(s) \, ds 
+ e^{C(C + \epsilon)} \left( C(C + \epsilon) + C(C + \epsilon)^3 \right) \\
&\leq e^{C(C + \epsilon)} C(C + \epsilon) 
+ e^{C(C + \epsilon)} \left( C(C + \epsilon) + C(C + \epsilon)^3 \right),
\end{aligned}
\end{align*}
where we have used the bound
\[
\int_t^{t+1} h(\tau) \, d\tau \leq C + \epsilon, \quad t \geq t_0(M, \epsilon).
\]

Finally, from the energy estimate
\begin{align*}
\begin{aligned}
\frac{d}{dt} \norm{\partial_z \Theta}_{L^2}^2
+ \kappa_a \int_D |\partial_{zz} \Theta|^2 
+ \int_D |\partial_{xz} \Theta|^2 \leq 
C(C + \epsilon) + C(C + \epsilon)^2 \norm{\nabla \Theta}_{L^2}^2,
\end{aligned}
\end{align*}
we obtain
\begin{align*}
\begin{aligned}
&\kappa_a \int_t^{t+1} \norm{\partial_{zz} \Theta(\tau)}_{L^2}^2 \, d\tau 
+ \int_t^{t+1} \norm{\partial_{xz} \Theta(\tau)}_{L^2}^2 \, d\tau \\
&\leq C_{\epsilon} + C(C + \epsilon) + C(C + \epsilon) \int_t^{t+1} \norm{\nabla \Theta(\tau)}_{L^2}^2 \, d\tau \\
&\leq C_{\epsilon} + C(C + \epsilon) + C(C + \epsilon)^2 = C_{\epsilon}.
\end{aligned}
\end{align*}
\end{proof}

\begin{lemma}\label{attractor4}
Let $M > 0$ and $\epsilon > 0$ be given. For any smooth solution $(v, \Theta)$ of the system \eqref{ondimensional-equations-2} with initial data $(v_0, \Theta_0) \in H^1(D) \times L^2(D)$ satisfying
\[
\norm{v_0}_{H^1}^2 + \norm{\Theta_0}_{L^2}^2 \leq M,
\]
there exists $t_0(M, \epsilon) > 0$ such that for all $t > t_0(M, \epsilon)$,
\begin{align}\label{l2bouned-24}
\norm{\partial_x \Theta(t)}_{L^2}^2 + \norm{\partial_x v(t)}_{L^2}^2 \leq C_{\epsilon},
\end{align}
and
\begin{align}\label{l2bouned-23}
\int_t^{t+1} \norm{\partial_x (\nabla v(\tau), \nabla \Theta(\tau))}_{L^2}^2 \, d\tau \leq C_{\epsilon},
\end{align}
where $C_{\epsilon}$ is a constant depending on $\epsilon$ but independent of $M$.
\end{lemma}
\begin{proof}
Taking $\epsilon > 0$ sufficiently small, we obtain the energy inequality:
\begin{align*}
\begin{aligned}
&\frac{d}{dt} \left( \norm{\partial_x v}_{L^2}^2 + \norm{\partial_x \Theta}_{L^2}^2 \right)
+ \mathrm{Pr}_x \int_D |\partial_{xx} v|^2 \, dx\,dz
+ \mathrm{Pr}_z \int_D |\partial_{xz} v|^2 \, dx\,dz \\
&+ \int_D |\partial_{xx} \Theta|^2 \, dx\,dz
+ \kappa_a \int_D |\partial_{xz} \Theta|^2 \, dx\,dz \leq r(t),
\end{aligned}
\end{align*}
where the remainder term $r(t)$ is given by
\begin{align*}
\begin{aligned}
r(t) := &\ C \left( \norm{v}_{L^2}^2 \norm{\partial_z v}_{L^2}^2 + 1 \right)
\left( \norm{\partial_x v}_{L^2}^2 + \norm{\partial_x \Theta}_{L^2}^2 \right) \\
&+ C \left( \norm{\partial_z v}_{L^2}^2 \norm{\partial_x v}_{L^2}^2 
+ \norm{v}_{L^2}^2 \norm{\partial_z v}_{L^2}^2 \right)
\left( \norm{\partial_x v}_{L^2}^2 + \norm{\partial_x \Theta}_{L^2}^2 \right) \\
&+ C \norm{\partial_z \Theta}_{L^2}^2 \norm{\partial_{zz} \Theta}_{L^2}^2
\left( \norm{\partial_x v}_{L^2}^2 + \norm{\partial_x \Theta}_{L^2}^2 \right).
\end{aligned}
\end{align*}

Let us define
\[
h(t) = \norm{\partial_x v}_{L^2}^2 + \norm{\partial_x \Theta}_{L^2}^2,
\]
and
\begin{align*}
\begin{aligned}
f_1(t) = &C \left( \norm{v}_{L^2}^2 \norm{\partial_z v}_{L^2}^2 + 1 \right)
+ C \norm{\partial_z \Theta}_{L^2}^2 \norm{\partial_{zz} \Theta}_{L^2}^2.
\\&+ C \left( \norm{\partial_z v}_{L^2}^2 \norm{\partial_x v}_{L^2}^2 
+ \norm{v}_{L^2}^2 \norm{\partial_z v}_{L^2}^2 \right).
\end{aligned}
\end{align*}
Then inequality \eqref{vthetagradient-2} implies
\begin{align}\label{vthetagradient-2}
\frac{dh}{dt} \leq f_1(t) h.
\end{align}

By Lemma~\ref{attractor1}--Lemma \ref{attractor3}, there exists $t_0(M, \epsilon) > 0$ such that for all $t \geq t_0(M, \epsilon)$,
\begin{align}
\begin{aligned}
\int_t^{t+1} h(\tau) \, d\tau &\leq C_{\epsilon}, \\
\int_t^{t+1} f_1(\tau) \, d\tau &\leq C_{\epsilon}.
\end{aligned}
\end{align}
Applying the uniform Gronwall inequality, we conclude that
\begin{align*}
\norm{\partial_x \Theta(t)}_{L^2}^2 + \norm{\partial_x v(t)}_{L^2}^2 \leq C_{\epsilon}, \quad \text{for all } t \geq t_0(M, \epsilon).
\end{align*}
Finally, integrating  \eqref{vthetagradient-2} from $t$ to $t+1$ and using  Lemma~\ref{attractor1}--Lemma \ref{attractor3}, we obtain
\begin{align*}
&\int_t^{t+1} \norm{\partial_x (\nabla v(\tau), \nabla \Theta(\tau))}_{L^2}^2 \, d\tau \\
&\leq C \int_t^{t+1} r(\tau) \, d\tau + h(t) \leq C_{\epsilon}, \quad \text{for all } t \geq t_0(M, \epsilon).
\end{align*}
\end{proof}

\begin{lemma}\label{attractor5}
Let $M > 0$ and $\epsilon > 0$ be given. For any smooth solution $(v, \Theta)$ of the system \eqref{ondimensional-equations-2} with initial data $(v_0, \Theta_0) \in H^2(D)$ satisfying
\[
\norm{v_0}_{H^2}^2 + \norm{\Theta_0}_{H^2}^2 \leq M,
\]
there exists $t_0(M, \epsilon) > 0$ such that for all $t > t_0(M, \epsilon)$,
\begin{align}\label{l2bouned-24}
\norm{\partial_z \nabla \Theta(t)}_{L^2}^2 + \norm{\partial_z \nabla v(t)}_{L^2}^2 \leq C_{\epsilon},
\end{align}
where $C_{\epsilon}$ is a constant depending on $\epsilon$ but independent of $M$.
\end{lemma}

\begin{proof}
Let $(\sigma, \eta) = \partial_z (v, \Theta)$ and define
\[
h(t) = \norm{\nabla(\sigma, \eta)}_{L^2}^2.
\]

Following the approach used to derive \eqref{h2dtestimate-all-1}, we obtain the energy estimate:
\begin{align}\label{h2dtestimate-all-attractor}
\begin{aligned}
&\frac{dh}{dt} + \lambda_1 \left( \norm{\partial_{xx} \sigma}_{L^2}^2
+ \norm{\partial_{zx} \sigma}_{L^2}^2
+ \norm{\partial_{xx} \eta}_{L^2}^2
+ \norm{\partial_{zx} \eta}_{L^2}^2 \right) \\
&\leq f_1(t) h + C \norm{\partial_x \Theta}_{L^2}^2,
\end{aligned}
\end{align}
where 
\[
\begin{aligned}
f_1(t) = &\ C \left( \norm{\partial_z v}_{L^2}^2 \norm{v}_{L^2}^2
+ \norm{\partial_x v}_{L^2}^2 \norm{\partial_{xx} v}_{L^2}^2 \right) \\
&+ C \norm{\partial_z v}_{L^2}^2 \norm{v}_{L^2}^2 
+ C \norm{\partial_x v}_{L^2} \norm{\partial_{zx} v}_{L^2} \\
&+ C \norm{\partial_x v}_{L^2}^2 \norm{\partial_{zx} v}_{L^2}^2 
+ C \norm{\sigma}_{L^2} \norm{\partial_x \Theta}_{L^2} \norm{\partial_{xx} \Theta}_{L^2} \\
&+ C \norm{\partial_x v}_{L^2} \norm{\partial_{xx} v}_{L^2} \norm{\eta}_{L^2}.
\end{aligned}
\]

By Lemma~\ref{attractor1}--Lemma \ref{attractor4}, there exists $t_0(M, \epsilon) > 0$ such that for all $t \geq t_0(M, \epsilon)$,
\[
\int_t^{t+1} \norm{\partial_x \Theta}_{L^2}^2 \, d\tau \leq C_{\epsilon},
\quad \text{and} \quad 
\int_t^{t+1} f_1(\tau) \, d\tau \leq C_{\epsilon}.
\]

Using a similar argument as in \eqref{ke-inequality} and applying the uniform Gronwall inequality to \eqref{h2dtestimate-all-attractor}, we conclude that there exists $t_0(M, \epsilon) > 0$ such that
\[
\norm{\partial_z \nabla \Theta(t)}_{L^2}^2 + \norm{\partial_z \nabla v(t)}_{L^2}^2 \leq C_{\epsilon}, \quad \text{for all } t > t_0(M, \epsilon).
\]
\end{proof}

\begin{lemma}\label{attractor6}
Let $M > 0$ and $\epsilon > 0$ be given. For any smooth solution $(v, \Theta)$ of the system \eqref{ondimensional-equations-2} with initial data $(v_0, \Theta_0) \in H^2(D)$ satisfying
\[
\norm{v_0}_{H^2}^2 + \norm{\Theta_0}_{H^2}^2 \leq M,
\]
there exists $t_0(M, \epsilon) > 0$ such that for all $t > t_0(M, \epsilon)$,
\begin{align}\label{l2bouned-24}
\norm{\partial_x \nabla \Theta(t)}_{L^2}^2 + \norm{\partial_x \nabla v(t)}_{L^2}^2 \leq C_{\epsilon},
\end{align}
where $C_{\epsilon}$ is a constant depending on $\epsilon$ but independent of $M$.
\end{lemma}

\begin{proof}
The proof follows the proofs of Lemma~\ref{lemma-grad-1-7} and Lemma~\ref{attractor5}.
\end{proof}

\subsection{Global Asymptotic Stability}

\begin{lemma}\label{lemma-grad-1-cases3}
Let $0 < T_0 - T_1 = T_c$ and $\mathrm{R} = \mathrm{R}_c$. For any smooth solution $(v, \Theta)$ of the system \eqref{ondimensional-equations-2} with initial data $(v_0, \Theta_0) \in L^2(D) \times L^2(D)$, we have
\begin{align}
\lim_{t \to \infty} \left( \norm{v(t)}_{L^2}^2 + \norm{\Theta(t)}_{L^2}^2 \right) = 0.
\end{align}
\end{lemma}

\begin{proof}
By Lemma \ref{eigen-all} and Remark \ref{remark51}, the eigenvalue problem \eqref{perturbation-31} possesses a countable infinity of eigenvalues given by \eqref{all-eigenvalue}. We relabel all eigenvalues and eigenvectors as $\{(\beta_k, (v_k, \Theta_k))\}_{k \in \mathbb{N}}$, where the eigenvalues (counting multiplicity) are ordered in decreasing real parts:
\begin{align*}
-\infty \leftarrow \cdots \leq \beta_3 \leq \beta_2 \leq \beta_1.
\end{align*}

The first eigenvalue has multiplicity two, and $\{(v_k, \Theta_k)\}_{k \in \mathbb{N}}$ forms an orthogonal basis of $\mathcal{X}_0(D)$. Moreover, the eigenvalues satisfy:
\begin{align*}
\begin{aligned}
&\beta_1 = \beta_2 = 0, \quad \text{as}~T_0-T_1=T_c , \\
&\beta_k < 0, \quad \text{as}~T_0-T_1=T_c . 
\end{aligned}
\end{align*}
Define the subspaces:
\begin{align}
E_0 &= \mathrm{Span}\{(v_1, \Theta_1), (v_2, \Theta_2)\}, \label{e01} \\
E_0^{\perp} &= \left\{ \psi \in \mathcal{X}_0(D) \,\middle|\, (\psi, \phi)_{L^2} = 0, ~ \forall \phi \in E_0 \right\}. \label{e02}
\end{align}
For a strong solution $(v, \Theta)$ of \eqref{ondimensional-equations-2}, consider the expansion:
\begin{align}
\begin{pmatrix} v \\ \Theta \end{pmatrix} = \sum_{k=1}^{\infty} \alpha_k \begin{pmatrix} v_k \\ \Theta_k \end{pmatrix}.
\end{align}

Testing \eqref{ondimensional-equations-2} with $(v, \Theta)$ and integrating by parts over $D$ yields:
\begin{align*}
\frac{1}{2} \frac{d}{dt} \left( \norm{v}_{L^2}^2 + \norm{\Theta}_{L^2}^2 \right)
= \sum_{m=3}^{\infty} \beta_m \alpha_m^2 \left( \norm{v_m}_{L^2}^2 + \norm{\Theta_m}_{L^2}^2 \right) \leq 0.
\end{align*}
Integrating from $0$ to $t$ gives:
\begin{align}\label{deady-1}
\begin{aligned}
\norm{v(t)}_{L^2}^2 + \norm{\Theta(t)}_{L^2}^2
&\leq \norm{v_0}_{L^2}^2 + \norm{\Theta_0}_{L^2}^2 \\
&\quad - 2|\lambda_3| \sum_{m=3}^{\infty} \alpha_m^2 \int_0^t \left( \norm{v_m(\tau)}_{L^2}^2 + \norm{\Theta_m(\tau)}_{L^2}^2 \right) d\tau \\
&= \norm{v_0}_{L^2}^2 + \norm{\Theta_0}_{L^2}^2 \\
&\quad - 2|\beta_3| \int_0^t \left( \norm{\tilde{v}(\tau)}_{L^2}^2 + \norm{\tilde{\Theta}(\tau)}_{L^2}^2 \right) d\tau,
\end{aligned}
\end{align}
where we decompose $(v, \Theta) = (u, \eta) + (\tilde{v}, \tilde{\Theta}) \in \mathcal{X}_0(D) = E_0 \oplus E_0^{\perp}$ with
\begin{align*}
(u, \eta)^T &= \alpha_1 (v_1, \Theta_1) + \alpha_2 (v_2, \Theta_2), \\
(\tilde{v}, \tilde{\Theta}) &= \sum_{m=3}^{\infty} \alpha_m (v_m, \Theta_m).
\end{align*}

For any $(v_0, \Theta_0) \in \mathcal{X}_2(D)$, the solution $(v(t; v_0, \Theta_0), \Theta(t; v_0, \Theta_0))$ of \eqref{ondimensional-equations-2} is non-increasing in time:$
\norm{v(t_2)}_{L^2}^2 + \norm{\Theta(t_2)}_{L^2}^2 \leq \norm{v(t_1)}_{L^2}^2 + \norm{\Theta(t_1)}_{L^2}^2, \quad \forall t_1 < t_2$.
Hence, the limit 
\[
\lim_{t \to \infty} \left( \norm{v(t)}_{L^2}^2 + \norm{\Theta(t)}_{L^2}^2 \right) = \delta \leq \norm{v_0}_{L^2}^2 + \norm{\Theta_0}_{L^2}^2
\]
exists. The $\omega$-limit set $\omega(v_0, \Theta_0)$, being invariant, satisfies
\[
\omega(v_0, \Theta_0) \subset S_\delta = \left\{ \phi \in \mathcal{X}_0(D) \,\middle|\, \norm{\phi}_{L^2} = \delta \right\}.
\]
For any $\phi \in \omega(v_0, \Theta_0)$, we have
\begin{align}\label{decay-2}
(v(t; \phi), \Theta(t; \phi)) \subset \omega(v_0, \Theta_0) \subset S_\delta, \quad \forall t \geq 0.
\end{align}
If $\phi = \phi_0 + \phi_1$ with $\phi_0 \in E_0$, $\phi_1 \in E_0^\perp$, and $\phi_1 \neq 0$, then \eqref{deady-1} implies
\[
\norm{v(t; \phi)}_{L^2}^2 + \norm{\Theta(t; \phi)}_{L^2}^2 < \norm{\phi}_{L^2}^2 = \delta, \quad \forall t > 0,
\]
contradicting \eqref{decay-2}. Therefore,
\[
\omega(v_0, \Theta_0) \subset E_0, \quad \forall (v_0, \Theta_0) \in \mathcal{X}_2(D).
\]

We now show that $\omega(v_0, \Theta_0) = \{(0,0)\}$ for all $(v_0, \Theta_0)$. Suppose there exists $\phi_s \neq (0,0)$ in $\omega(v_0, \Theta_0)\subset E_0$. Then the solution $(v, \Theta)$ with initial data $\phi_s$ satisfies the nonlinear system:
\begin{align}\label{perturbation-3-special}
\begin{cases}
\begin{aligned}
\partial_t v = -v \partial_x v + \left( \int_0^z \partial_x v(t,x,\xi) \, d\xi \right) \partial_z v,
\end{aligned} \\
\begin{aligned}
\partial_t \Theta = -v \partial_x \Theta + \left( \int_0^z \partial_x v(t,x,\xi) \, d\xi \right) \partial_z \Theta,
\end{aligned} \\
\partial_z v|_{z=0} = \partial_z v|_{z=1} = 0, \quad \Theta|_{z=0} = 0, \quad \Theta|_{z=1} = 0, \\
v, \Theta \text{ are periodic in $x$ with period $2\pi L$}, \\
\int_0^1 v(t,x,\xi) \, d\xi = 0.
\end{cases}
\end{align}
Note that for any $\gamma > 0$, the rescaled functions 
\[
(\gamma v(\gamma t, x, z), \gamma \Theta(\gamma t, x, z))
\]
 also satisfy \eqref{perturbation-3-special}. This scaling property contradicts the uniform bounds established in Lemma \ref{attractor4}. Hence, $\omega(v_0, \Theta_0) = \{(0,0)\}$ for all $(v_0, \Theta_0)$.
 This infers $\lim_{t \to \infty} \left( \norm{v(t)}_{L^2}^2 + \norm{\Theta(t)}_{L^2}^2 \right) = 0$.
\end{proof}

\begin{corollary}\label{colarry}
Let $0 < T_0 - T_1 = T_c$ and $\mathrm{R} = \mathrm{R}_c$. For any smooth solution $(v, \Theta)$ of system \eqref{ondimensional-equations-2} with initial data $(v_0, \Theta_0) \in H^1(D)$, we have
\begin{align}
\lim_{t \to \infty} \left( \norm{v(t)}_{H^1}^2 + \norm{\Theta(t)}_{H^1}^2 \right) = 0.
\end{align}
\end{corollary}

\begin{proof}
By Lemma \ref{attractor5} and Lemma \ref{attractor6}, the system \eqref{ondimensional-equations-2} possesses a global attractor in $H^1(D)$. This global attractor coincides with the one in $L^2(D)$. Since we have shown that the global attractor in $L^2(D)$ for $0 < T_0 - T_1 = T_c$ is $\{(0,0)\}$, it follows that
\[
\lim_{t \to \infty} \left( \norm{v(t)}_{H^1}^2 + \norm{\Theta(t)}_{H^1}^2 \right) = 0.
\]
\end{proof}

\section{Nonlinear instability}
\subsection{Nonlinear estimates}
\begin{lemma}\label{estimes-nonlinear-instability}
Let $T > 0$ be a given time and $(v_0, \Theta_0) \in H^2(D)$ be initial data. Then there exists a unique strong solution $(v, \Theta) \in C([0,T]; H^2(D))$ to the problem \eqref{ondimensional-equations-2}. Moreover, there exists $\delta_0 \in (0,1]$ such that if
\[
\norm{v(t)}_{H^2}^2 + \norm{\Theta(t)}_{H^2}^2 \leq \delta_0 \quad \text{for all } t \in [0,T],
\]
then the following estimate holds:
\begin{align}\label{nonlinear0718}
\begin{aligned}
&\norm{v(t)}_{H^2}^2 + \norm{\Theta(t)}_{H^2}^2 
+ \norm{\partial_x q(t)}_{L^2}^2 
+ \norm{\partial_t v(t)}_{L^2}^2 
+ \norm{\partial_t \Theta(t)}_{L^2}^2 \\
&+ \int_0^t \left( \norm{\nabla v(s)}_{L^2}^2 
+ \norm{\nabla \Theta(s)}_{L^2}^2 
+ \norm{\partial_s v(s)}_{H^1}^2 
+ \norm{\partial_s \Theta(s)}_{H^1}^2 \right) ds \\
&\leq C^* \left( \norm{v_0}_{H^2}^2 + \norm{\Theta_0}_{H^2}^2 \right)
+ C^* \int_0^t \left( \norm{v(\tau)}_{L^2}^2 + \norm{\Theta(\tau)}_{L^2}^2 \right) d\tau,
\end{aligned}
\end{align}
where the constant $C^* > 0$ is independent of $T$.
\end{lemma}
\begin{proof}
From the equations \eqref{ondimensional-equations-2}, we derive the following energy estimates:
\begin{align}\label{nonlinear-instability}
&\begin{aligned}
&\frac{1}{2}\frac{d}{dt}\left( \norm{v}_{L^2}^2 + \norm{\Theta}_{L^2}^2 \right)
+ \frac{1}{2}\left( \mathrm{Pr}_x \norm{\partial_x v}_{L^2}^2 + \mathrm{Pr}_z \norm{\partial_z v}_{L^2}^2 \right. \\
&\left. + \norm{\partial_x \Theta}_{L^2}^2 + \kappa_a \norm{\partial_z \Theta}_{L^2}^2 \right)
\leq C\left( \norm{v}_{L^2}^2 + \norm{\Theta}_{L^2}^2 \right),
\end{aligned} \\
\label{nonlinear-instability-2}
&\begin{aligned}
&\frac{3}{4}\left( \norm{\partial_t v}_{L^2}^2 + \norm{\partial_t \Theta}_{L^2}^2 \right)
+ \frac{1}{2}\frac{d}{dt} \left( \mathrm{Pr}_x \norm{\partial_x v}_{L^2}^2 + \mathrm{Pr}_z \norm{\partial_z v}_{L^2}^2 \right. \\
&\left. + \norm{\partial_x \Theta}_{L^2}^2 + \kappa_a \norm{\partial_z \Theta}_{L^2}^2 \right)
\leq C\left( \norm{\nabla v}_{L^2}^2 + \norm{\nabla \Theta}_{L^2}^2 \right) \\
&+ \int_D \left( -v \partial_x v + \left( \int_0^z \partial_x v(t,x,\xi) \, d\xi \right) \partial_z v \right) \partial_t v \, dx\,dz \\
&+ \int_D \left( -v \partial_x \Theta + \left( \int_0^z \partial_x v(t,x,\xi) \, d\xi \right) \partial_z \Theta \right) \partial_t \Theta \, dx\,dz \\
&= C\left( \norm{\nabla v}_{L^2}^2 + \norm{\nabla \Theta}_{L^2}^2 \right) + T_1 + T_2,
\end{aligned}
\end{align}
where
\begin{align*}
\begin{aligned}
T_1 &:= \int_D \left( -v \partial_x v + \left( \int_0^z \partial_x v(t,x,\xi) \, d\xi \right) \partial_z v \right) \partial_t v \, dx\,dz = T_{11} + T_{12}, \\
T_2 &:= \int_D \left( -v \partial_x \Theta + \left( \int_0^z \partial_x v(t,x,\xi) \, d\xi \right) \partial_z \Theta \right) \partial_t \Theta \, dx\,dz = T_{21} + T_{22}.
\end{aligned}
\end{align*}
We estimate the terms $T_{ij}$ ($i,j=1,2$) as follows:
\begin{align*}
T_{11} &\leq \norm{v \partial_x v}_{L^2} \norm{\partial_t v}_{L^2}
\leq \norm{v}_{L^4} \norm{\partial_x v}_{L^4} \norm{\partial_t v}_{L^2} \\
&\leq \epsilon \norm{\partial_t v}_{L^2}^2 + C_{\epsilon} \norm{v}_{L^2}^2 \norm{v}_{H^2}^2, \\
T_{12} &\leq \norm{ \left( \int_0^z \partial_x v(t,x,\xi) \, d\xi \right) \partial_z v }_{L^2} \norm{\partial_t v}_{L^2} \\
&\leq \epsilon \norm{\partial_t v}_{L^2}^2 + C_{\epsilon} \norm{\partial_z v}_{L^2} \norm{\partial_{xz} v}_{L^2} \norm{\partial_x v}_{L^2}^2 \\
&\leq \epsilon \norm{\partial_t v}_{L^2}^2 + C_{\epsilon} \norm{v}_{H^1}^2 \norm{v}_{H^2}^2, \\
T_{21} &\leq \norm{v \partial_x \Theta}_{L^2} \norm{\partial_t \Theta}_{L^2}
\leq \norm{v}_{L^4} \norm{\partial_x \Theta}_{L^4} \norm{\partial_t \Theta}_{L^2} \\
&\leq \epsilon \norm{\partial_t \Theta}_{L^2}^2 + C_{\epsilon} \norm{v}_{H^1}^2 \norm{\Theta}_{H^2}^2, \\
T_{22} &\leq \norm{ \left( \int_0^z \partial_x v(t,x,\xi) \, d\xi \right) \partial_z \Theta }_{L^2} \norm{\partial_t \Theta}_{L^2} \\
&\leq \epsilon \norm{\partial_t \Theta}_{L^2}^2 + C_{\epsilon} \norm{\partial_x v}_{L^2} \norm{\partial_{xx} v}_{L^2} \norm{\partial_z \Theta}_{L^2} \norm{\partial_{zz} \Theta}_{L^2} \\
&\leq \epsilon \norm{\partial_t \Theta}_{L^2}^2 + C_{\epsilon} \left( \norm{v}_{H^1}^2 \norm{v}_{H^2}^2 + \norm{\Theta}_{H^1}^2 \norm{\Theta}_{H^2}^2 \right).
\end{align*}
Combining these estimates yields:
\begin{align}\label{nonlinear-instability-21}
\begin{aligned}
&\frac{1}{2}\left( \norm{\partial_t v}_{L^2}^2 + \norm{\partial_t \Theta}_{L^2}^2 \right)
+ \frac{1}{2}\frac{d}{dt} \left( \mathrm{Pr}_x \norm{\partial_x v}_{L^2}^2 + \mathrm{Pr}_z \norm{\partial_z v}_{L^2}^2 \right. \\
&\left. + \norm{\partial_x \Theta}_{L^2}^2 + \kappa_a \norm{\partial_z \Theta}_{L^2}^2 \right)
\leq C\left( \norm{\nabla v}_{L^2}^2 + \norm{\nabla \Theta}_{L^2}^2 \right) \\
&+ C\left( \norm{v}_{L^2}^2 \norm{v}_{H^2}^2 + \norm{v}_{H^1}^2 \norm{\Theta}_{H^2}^2 \right)
+ C\left( \norm{v}_{H^1}^2 \norm{v}_{H^2}^2 + \norm{\Theta}_{H^1}^2 \norm{\Theta}_{H^2}^2 \right).
\end{aligned}
\end{align}

Differentiating \eqref{ondimensional-equations-2} with respect to $t$, we obtain:
\begin{align}\label{nonlinear-instability-tt}
\begin{aligned}
&\frac{1}{2}\frac{d}{dt}\left( \norm{\partial_t v}_{L^2}^2 + \norm{\partial_t \Theta}_{L^2}^2 \right)
+ \left( \frac{3\mathrm{Pr}_x}{4} \norm{\partial_x \partial_t v}_{L^2}^2 + \mathrm{Pr}_z \norm{\partial_z \partial_t v}_{L^2}^2 \right. \\
&\left. + \frac{3}{4} \norm{\partial_x \partial_t \Theta}_{L^2}^2 + \kappa_a \norm{\partial_z \partial_t \Theta}_{L^2}^2 \right)
\leq C\left( \norm{\partial_t v}_{L^2}^2 + \norm{\partial_t \Theta}_{L^2}^2 \right) \\
&+ \int_D \left( -\partial_t v \partial_x v + \left( \int_0^z \partial_x \partial_t v(t,x,\xi) \, d\xi \right) \partial_z v \right) \partial_t v \, dx\,dz \\
&+ \int_D \left( -\partial_t v \partial_x \Theta + \left( \int_0^z \partial_x \partial_t v(t,x,\xi) \, d\xi \right) \partial_z \Theta \right) \partial_t \Theta \, dx\,dz \\
&= C\left( \norm{\partial_t v}_{L^2}^2 + \norm{\partial_t \Theta}_{L^2}^2 \right) + U_1 + U_2,
\end{aligned}
\end{align}
where
\begin{align*}
U_1 &= \int_D \left( -\partial_t v \partial_x v + \left( \int_0^z \partial_x \partial_t v(t,x,\xi) \, d\xi \right) \partial_z v \right) \partial_t v \, dx\,dz = U_{11} + U_{12}, \\
U_2 &= \int_D \left( -\partial_t v \partial_x \Theta + \left( \int_0^z \partial_x \partial_t v(t,x,\xi) \, d\xi \right) \partial_z \Theta \right) \partial_t \Theta \, dx\,dz = U_{21} + U_{22}.
\end{align*}
We estimate $U_{ij}$ ($i,j=1,2$) as follows:
\begin{align*}
U_{11} &\leq \norm{\partial_t v \partial_x v}_{L^2} \norm{\partial_t v}_{L^2}
\leq \norm{\partial_x v}_{L^2} \norm{\partial_t v}_{L^4}^2 \\
&\leq \epsilon \norm{\partial_t \nabla v}_{L^2}^2 + C_{\epsilon} \norm{\nabla v}_{L^2}^2 \norm{\partial_t v}_{L^2}^2, \\
U_{12} &\leq \norm{ \left( \int_0^z \partial_x \partial_t v(t,x,\xi) \, d\xi \right) \partial_z v }_{L^2} \norm{\partial_t v}_{L^2} \\
&\leq \norm{\partial_z v}_{L^2}^{1/2} \norm{\partial_{zz} v}_{L^2}^{1/2} \norm{\partial_x \partial_t v}_{L^2} \norm{\partial_t v}_{L^2} \\
&\leq \epsilon \norm{\partial_t \nabla v}_{L^2}^2 + C_{\epsilon} \norm{\partial_z v}_{L^2} \norm{\partial_{zz} v}_{L^2} \norm{\partial_t v}_{L^2}^2 \\
&\leq \epsilon \norm{\partial_t \nabla v}_{L^2}^2 + C_{\epsilon} \norm{v}_{H^2}^2 \norm{\partial_t v}_{L^2}^2, \\
U_{21} &\leq \norm{\partial_t v \partial_x \Theta}_{L^2} \norm{\partial_t \Theta}_{L^2}
\leq \norm{\partial_x \Theta}_{L^2} \norm{\partial_t v}_{L^4} \norm{\partial_t \Theta}_{L^4} \\
&\leq \epsilon \left( \norm{\partial_t \nabla v}_{L^2}^2 + \norm{\partial_t \nabla \Theta}_{L^2}^2 \right)
+ C_{\epsilon} \norm{\nabla \Theta}_{L^2}^2 \left( \norm{\partial_t v}_{L^2}^2 + \norm{\partial_t \Theta}_{L^2}^2 \right), \\
U_{22} &\leq \norm{ \left( \int_0^z \partial_x \partial_t v(t,x,\xi) \, d\xi \right) \partial_z \Theta }_{L^2} \norm{\partial_t \Theta}_{L^2} \\
&\leq C \norm{\partial_x \partial_t v}_{L^2} \norm{\partial_z \Theta}_{L^4} \norm{\partial_t \Theta}_{L^4} \\
&\leq \epsilon \left( \norm{\partial_t \nabla v}_{L^2}^2 + \norm{\partial_t \nabla \Theta}_{L^2}^2 \right)
+ C_{\epsilon} \norm{\partial_z \Theta}_{L^4}^2 \norm{\partial_t \Theta}_{L^4}^2 \\
&\leq \epsilon \left( \norm{\partial_t \nabla v}_{L^2}^2 + \norm{\partial_t \nabla \Theta}_{L^2}^2 \right)
+ C_{\epsilon} \norm{\partial_t \Theta}_{L^2}^2 \left( \norm{\Theta}_{H^2}^2 + \norm{\Theta}_{H^2}^4 \right).
\end{align*}

Define the energy functionals:
\begin{align*}
E_1(t) &= \norm{v}_{L^2}^2 + \norm{\Theta}_{L^2}^2, \\
E_2(t) &= \norm{\partial_t v}_{L^2}^2 + \norm{\partial_t \Theta}_{L^2}^2, \\
E_3(t) &= \mathrm{Pr}_x \norm{\partial_x v}_{L^2}^2 + \mathrm{Pr}_z \norm{\partial_z v}_{L^2}^2 + \norm{\partial_x \Theta}_{L^2}^2 + \kappa_a \norm{\partial_z \Theta}_{L^2}^2, \\
E_4(t) &= \mathrm{Pr}_x \norm{\partial_x \partial_t v}_{L^2}^2 + \mathrm{Pr}_z \norm{\partial_z \partial_t v}_{L^2}^2 + \norm{\partial_x \partial_t \Theta}_{L^2}^2 + \kappa_a \norm{\partial_z \partial_t \Theta}_{L^2}^2.
\end{align*}

Combining the estimates and choosing appropriate constants $A_1, A_2, A_3 > 0$ with $A_3 = \frac{7}{2} A_2 C$ and $A_1$ sufficiently large, we obtain:
\begin{align*}
&\frac{1}{2}\frac{d}{dt}\left( A_1 E_1 + A_2 E_2 + A_3 E_3 \right)
+ \frac{1}{2}\left( A_1 E_3 + A_2 E_4 + A_3 E_2 \right) \\
&\leq A_1 C E_1 + A_2 C E_2 \left( 1 + \norm{v}_{H^2}^2 + \norm{\Theta}_{H^2}^2 + \norm{v}_{H^2}^4 + \norm{\Theta}_{H^2}^4 \right) \\
&\quad + A_3 C \left( \norm{\nabla v}_{L^2}^2 + \norm{\nabla \Theta}_{L^2}^2 \right)
+ A_3 C \left( \norm{v}_{L^2}^2 \norm{v}_{H^2}^2 + \norm{v}_{L^2}^2 \norm{\Theta}_{H^2}^2 \right) \\
&\quad + A_3 C \left( \norm{\nabla v}_{L^2}^2 + \norm{\nabla \Theta}_{L^2}^2 \right) \left( \norm{v}_{H^2}^2 + \norm{\Theta}_{H^2}^2 \right).
\end{align*}

With the choice of constants, there exists $C^* > 0$ such that:
\begin{align*}
&\frac{1}{2}\frac{d}{dt}\left( A_1 E_1 + A_2 E_2 + A_3 E_3 \right)
+ \frac{1}{2}\left( A_3 C E_3 + A_2 E_4 + A_2 C E_2 \right) \leq A_1 C E_1.
\end{align*}
Integrating in time , we finally obtain:
\begin{align*}
&\norm{v(t)}_{H^1}^2 + \norm{\Theta(t)}_{H^1}^2 + \norm{\partial_t v(t)}_{L^2}^2 + \norm{\partial_t \Theta(t)}_{L^2}^2 \\
&+ \int_0^t \left( \norm{\nabla v(s)}_{L^2}^2 + \norm{\nabla \Theta(s)}_{L^2}^2 + \norm{\partial_s v(s)}_{H^1}^2 + \norm{\partial_s \Theta(s)}_{H^1}^2 \right) ds \\
&\leq C^* \left( \norm{v_0}_{H^2}^2 + \norm{\Theta_0}_{H^2}^2 \right)
+ C^* \int_0^t \left( \norm{v(\tau)}_{L^2}^2 + \norm{\Theta(\tau)}_{L^2}^2 \right) d\tau.
\end{align*}
wher we have used the bound $\norm{\partial_t v}_{L^2}^2 + \norm{\partial_t \Theta}_{L^2}^2 \leq C\left( \norm{v_0}_{H^2}^2 + \norm{\Theta_0}_{H^2}^2 \right)$.
\end{proof}
\subsection{Proof of Theorem \ref{noninstability}}\label{proof-theorem23}
\begin{proof}
We now use the preceding lemma to prove Theorem \ref{noninstability}.
For $T_0-T_1>\Delta T_c$, it follows from Subsection \ref{pescondition} that the linear problem 
\eqref{perturbation-31} has a solution
as follows $(v,\Theta)=(u_1,\Theta_1)e^{\beta_1 t}$,
where $(u_1,\Theta_1)\in C^{\infty}(D)\cap \mathcal{X}_2(D) $ is the first eigefunction given in \ref{first-eigenvector}, $\beta_{1}>0$, and it satisfies $\norm{(u_1,\Theta_1)}_{H^2}=1$, 
$\norm{(u_1,\Theta_1)}_{L^2}=C_1$ and $\norm{(u_1,\Theta_1)}_{L^1}=C_2$.
Choosing $(v_0,\Theta_0):=(v,\Theta)_{t=0}=\delta (u_1,\Theta_1)$,
for $\delta\in (0,\delta_0)$. The problem \eqref{ondimensional-equations-2} has a global solution
$(v^{\delta},\Theta^{\delta})\in  C([0,T];H^2(D)\times H^2(D))$,
its initial data $(v_0,\Theta_0)=\delta (u_1,\Theta_1)$ satisfies
$\norm{(v^{\delta}(0),\Theta^{\delta}(0))}_{H^2}=\delta$.

For any $\delta \in (0,\delta_0)$ satisfying $\delta<\epsilon_0$, let us define 
\begin{align}
T_{\delta}=\lambda_1^{-1}\ln \frac{\epsilon_0}{\delta},i.e., \epsilon_0=\delta e^{\lambda_1 T_{\delta}}
\end{align}
where $\epsilon_0$ independent of $\delta$ is small positive constant to be determined,
and define 
\begin{align*}
&T_{*}=\sup\{t\in (0,\infty)|\norm{(v^{\delta},\Theta^{\delta})}_{H^2}\leq \delta_0\},\\
&T_{**}=\sup\{t\in (0,\infty)|\norm{(v^{\delta},\Theta^{\delta})}_{L^2}\leq 2C_1\delta e^{\beta_1t}\}.
\end{align*}
By the definitions of $T_{*}$ and $T_{**}$, we have $T_{*}>0$ and $T_{**}>0$, and 
\begin{align*}
&\norm{
(v^{\delta}(T_{*}),\Theta^{\delta}(T_{*}))
}_{H^2}=\delta_0,\quad \text{if}\quad T_{*}<\infty,\\
&\norm{(v^{\delta}(T_{**}),\Theta^{\delta}(T_{**}))}_{L^2}=2C_1\delta e^{\beta_1T_{**}},\quad \text{if}\quad T_{**}<\infty.
\end{align*}
Note that for $t\leq \min\{T_{\delta},T_{*},T_{**}\}$, by the estimate \ref{nonlinear0718}, we have
 \begin{align}\label{nonlinear-delta}
\begin{aligned}
&\norm{v^{\delta}}_{H^2}^2
+\norm{\Theta^{\delta}}_{H^2}^2
+\norm{\partial_tv^{\delta}}_{L^2}^2
+\norm{\partial_t\Theta^{\delta}}_{L^2}^2
\\&+\int_0^t
\left(
\norm{\nabla v^{\delta}}_{L^2}^2
+\norm{\nabla \Theta^{\delta}}_{L^2}^2
+\norm{\partial_sv^{\delta}}_{H^1}^2
+\norm{\partial_s\Theta^{\delta}}_{H^1}^2
\right)\,ds
\\&\leq C^*\left(\delta^2
+4C_1^2\delta^2\int_0^t e^{2\beta_1t}\,d\tau\right)
\leq C_3\delta^2e^{2\beta_1t},
\end{aligned}
\end{align}
where $C_3$ is independent of $\delta$.

Let $(u^{\delta},\theta^{\delta})$ be the solution of linear problem 
associated with the system \eqref{ondimensional-equations-2}
 with initial data $\delta (u_1,\Theta_1)$. Consider $(v_e^{\delta},\Theta_e^{\delta})=(v^{\delta},\Theta^{\delta})-(u^{\delta},\theta^{\delta})$, it satisfies the following equations
     \begin{align}\label{main=p-ln}
\begin{cases}
\begin{aligned}
\partial_tv_e^{\delta}=&\text{Pr}_x\partial_{xx}v_e^{\delta}+
\text{Pr}_z\partial_{zz}v_e^{\delta}-\text{R}\partial_x \left(
\int_0^z\Theta_e^{\delta}(t,x,\xi)\,d\xi\right)-\partial_xq(t,x)
\\&-v^{\delta} \partial_x v^{\delta}+\left(\int_0^z\partial_xv^{\delta}(t,x,\xi)\,d\xi \right)\partial_z v^{\delta},
\end{aligned}\\
\begin{aligned}
\partial_t\Theta_e^{\delta}= &\partial_{xx} \Theta_e^{\delta}+
\kappa_{a}\partial_{zz} \Theta_e^{\delta}
-\text{R}\left(\int_0^z\partial_xv_e^{\delta}(t,x,\xi)\,d\xi \right)
\\&-v^{\delta}\partial_x \Theta^{\delta}+\left(\int_0^z\partial_xv^{\delta}(t,x,\xi)
\,d\xi \right)\partial_z \Theta^{\delta},
\end{aligned}\\
\partial_zv_e^{\delta}|_{z=0}= \partial_zv_e^{\delta}|_{z=1}=0, \quad
\Theta_e^{\delta}|_{z=0}=0,\quad \Theta_e^{\delta}|_{z=1}=0,\\
v_e^{\delta}, \Theta_e^{\delta}~\text{ are periodic in $x$ with period $2\pi L$}.\\
\int_0^1v_e^{\delta}(t,x,\xi)\,d\xi=0.
\end{cases}
\end{align}
Taking the $L^2$-inner product of \eqref{main=p-ln} with $(v_e^{\delta},\Theta_e^{\delta})$, we have
  \begin{align}\label{ut-proof-1-ln}
\begin{aligned}
&\frac{1}{2}\frac{d}{dt}
\left(\norm{v_{e}^{\delta}}_{L^2}^2
+\norm{\Theta_{e}^{\delta}}_{L^2}^2
\right)- \beta_1
\left(\norm{v_{e}^{\delta}}_{L^2}^2
+\norm{\Theta_{e}^{\delta}}_{L^2}^2
\right)
 \\&\leq\int_{D}\left(
 -v^{\delta}\partial_x v^{\delta}+\left(\int_0^z\partial_xv^{\delta}(t,x,\xi)
\,d\xi \right)\partial_z v^{\delta}\right)v_{e}^{\delta}\,dx\,dz\\
 &\quad+\int_{D}\left(
 -v^{\delta}\partial_x \Theta^{\delta}+\left(\int_0^z\partial_xv^{\delta}(t,x,\xi)
\,d\xi \right)\partial_z \Theta^{\delta}\right)\Theta_{e}^{\delta}\,dx\,dz
\\&:=L_1+L_2+L_3+L_4.
\end{aligned}
\end{align}
Note that 
  \begin{align}\label{ut-proof-1-ln-1}
  \begin{aligned}
&\begin{aligned}
L_1&=\int_{D}\left(
 -v^{\delta}\partial_x v^{\delta}\right)v_{e}^{\delta}\,dx\,dz
 \leq \norm{v_{e}^{\delta}}_{L^2}
 \norm{v^{\delta}}_{L^4}
 \norm{\partial_xv^{\delta}}_{L^4}
 \\&\leq C \sqrt{\norm{v_{e}^{\delta}}_{L^2}^2
 +\norm{\Theta_{e}^{\delta}}_{L^2}^2}
 \left( \norm{v^{\delta}}_{H^2}^2
 + \norm{\Theta^{\delta}}_{H^2}^2\right),
\end{aligned}\\
&\begin{aligned}
L_2&=\int_{D}\left(\left(\int_0^z\partial_xv^{\delta}(t,x,\xi)
\,d\xi \right)\partial_z v^{\delta}\right)v_{e}^{\delta}\,dx\,dz
\\&\leq \norm{v_{e}^{\delta}}_{L^2}
\norm{\left(\int_0^z\partial_xv^{\delta}(t,x,\xi)
\,d\xi \right)\partial_z v^{\delta}}_{L^2}
\\& \leq C \sqrt{\norm{v_{e}^{\delta}}_{L^2}^2
 +\norm{\Theta_{e}^{\delta}}_{L^2}^2}
 \left( \norm{v^{\delta}}_{H^2}^2
 + \norm{\Theta^{\delta}}_{H^2}^2\right),
\end{aligned}
\end{aligned}
\end{align}
and
  \begin{align}\label{ut-proof-1-ln-2}
\begin{aligned}
&\begin{aligned}
L_3&=\int_{D}\left(
 -v^{\delta}\partial_x \Theta^{\delta}\right)\Theta_{e}^{\delta}\,dx\,dz
\leq \norm{\Theta_{e}^{\delta}}_{L^2}
 \norm{v^{\delta}}_{L^4}
 \norm{\partial_x\Theta^{\delta}}_{L^4}
 \\&\leq C \sqrt{\norm{v_{e}^{\delta}}_{L^2}^2
 +\norm{\Theta_{e}^{\delta}}_{L^2}^2}
 \left( \norm{v^{\delta}}_{H^2}^2
 + \norm{\Theta^{\delta}}_{H^2}^2\right),
\end{aligned}\\
&\begin{aligned}
L_4&=\int_{D}\left(\left(\int_0^z\partial_xv^{\delta}(t,x,\xi)
\,d\xi \right)\partial_z \Theta^{\delta}\right)\Theta_{e}^{\delta}\,dx\,dz
\\&\leq \norm{\Theta_{e}^{\delta}}_{L^2}
\norm{\left(\int_0^z\partial_xv^{\delta}(t,x,\xi)
\,d\xi \right)\partial_z \Theta^{\delta}}_{L^2}
\\& \leq C \sqrt{\norm{\Theta_{e}^{\delta}}_{L^2}^2
 +\norm{\Theta_{e}^{\delta}}_{L^2}^2}
 \left( \norm{v^{\delta}}_{H^2}^2
 + \norm{\Theta^{\delta}}_{H^2}^2\right),
\end{aligned}
\end{aligned}
\end{align}
which and \eqref{ut-proof-1-ln} yields that 
  \begin{align}\label{ut-proof-1-ln-2}
\begin{aligned}
\frac{d}{dt}
\sqrt{\norm{v_{e}^{\delta}}_{L^2}^2
 +\norm{\Theta_{e}^{\delta}}_{L^2}^2}- \beta_1
\sqrt{\norm{v_{e}^{\delta}}_{L^2}^2
 +\norm{\Theta_{e}^{\delta}}_{L^2}^2} \leq 
 C  \left( \norm{v^{\delta}}_{H^2}^2
 + \norm{\Theta^{\delta}}_{H^2}^2\right).
\end{aligned}
\end{align}
Gronwall inequality and \eqref{nonlinear-delta} give that
  \begin{align}\label{ut-proof-1-ln-3}
\begin{aligned}
\sqrt{\norm{v_{e}^{\delta}}_{L^2}^2
 +\norm{\Theta_{e}^{\delta}}_{L^2}^2}
& \leq C e^{ \beta_1t}
 \int_0^t
 \left( \norm{v^{\delta}}_{H^2}^2
 + \norm{\Theta^{\delta}}_{H^2}^2\right)\,ds
 \\&\leq  CC_3\delta^2e^{ \beta_1t}
  \int_0^te^{ 2\beta_1s}\,ds
  \leq C_4\delta^2e^{ 2\beta_1t}.
\end{aligned}
\end{align}

Let us choose $\epsilon_0$ such that $\epsilon_0<
\min\{\frac{\delta_0}{\sqrt{C_3}},\frac{C_1}{C_4},\frac{C_2}{2\abs{\Omega}^{\frac{1}{2}}C_4}
\}$. Then we aim to show $T_{\delta}=\min\{T_{\delta},T_{*},T_{**}\}$.
If $T_{*}=\min\{T_{\delta},T_{*},T_{**}\}$, $T_{\delta}>T_{*}$ and $T_{*}<\infty$, we have
\[
\sqrt{\norm{v^{\delta}(T_{*})}_{H^2}^2
 + \norm{\Theta^{\delta}(T_{*})}_{H^2}^2}
 \leq 
\sqrt{C_3}\delta e^{\lambda_1T_{*}}<
\sqrt{C_3}\epsilon_0\leq \delta_0
\]
which is contradictory to the definition of $T_{*}$.

If $T_{**}=\min\{T_{\delta},T_{*},T_{**}\}$, $T_{\delta}>T_{**}$ and $T_{**}<\infty$, we have
\begin{align*}
&\sqrt{\norm{v^{\delta}(T_{**})}_{L^2}^2
 + \norm{\Theta^{\delta}(T_{**})}_{L^2}^2}
\\&\leq 
\sqrt{\norm{u^{\delta}(T_{**})}_{L^2}^2
 + \norm{\theta^{\delta}(T_{**})}_{L^2}^2}
 +\sqrt{\norm{v_e^{\delta}(T_{**})}_{L^2}^2
 + \norm{\Theta_e^{\delta}(T_{**})}_{L^2}^2}
\\&\leq C_1\delta 
e^{\lambda_1T_{**}}+C_4\delta^2e^{2\lambda_1T_{**}}\\
&\leq 
C_1\delta 
e^{\lambda_1T_{**}}\left(
1+
\frac{C_4}{C_1}\delta e^{\lambda_1T_{**}}
\right)<
C_1\delta 
e^{\lambda_1T_{**}}\left(
1+
\frac{C_4}{C_1}\epsilon_0
\right)\leq 2
C_1\delta 
e^{\lambda_1T_{**}}.
\end{align*}
which is contradictory to the definition of $T_{**}$. By 
$T_{\delta}=\min\{T_{\delta},T_{*},T_{**}\}$, we finally get
\begin{align*}
&
\sqrt{\norm{v^{\delta}(T_{\delta})}_{L^1}
 + \norm{\Theta^{\delta}(T_{\delta})}_{L^1}}
\\&\geq 
\sqrt{\norm{u^{\delta}(T_{\delta})}_{L^1}
 + \norm{\theta^{\delta}(T_{\delta})}_{L^1}}
-\sqrt{\norm{v_e^{\delta}(T_{\delta})}_{L^1}
 + \norm{\Theta_e^{\delta}(T_{\delta})}_{L^1}}
\\&\geq C_2\delta e^{\lambda_1 T_{\delta}}-\abs{D}^{\frac{1}{2}}
\sqrt{\norm{v_e^{\delta}(T_{\delta})}_{L^2}^2
 + \norm{\Theta_e^{\delta}(T_{\delta})}_{L^2}^2}
\\&\geq
C_2\epsilon_0-\abs{D}^{\frac{1}{2}}
C_4\epsilon_0^2>\frac{C_2\epsilon_0}{2}.
\end{align*}
\end{proof}

\section{Dynamic bifurcation}

For $T_0 - T_1 > T_c$, the trivial steady state $(0,0)$ of system \eqref{ondimensional-equations-2} is both linearly and nonlinearly unstable. As $T_0 - T_1$ decreases through the critical value $T_c$, the system is expected to undergo a bifurcation, giving rise to additional stable, nontrivial steady states. This indicates a transition from the trivial state to more complex configurations. In this section, we characterize this dynamic bifurcation and the corresponding transition of \eqref{ondimensional-equations-2} as $T_0 - T_1$ crosses $T_c$.

We begin by reformulating the system \eqref{ondimensional-equations-2} in abstract form for the case $T_0 > T_1$. Define the bilinear operator $\mathcal{N}: \mathcal{X}_2(D)  \to \mathcal{X}_0(D) $ by
\begin{align}
\mathcal{N}(\psi, \tilde{\psi}) = 
\begin{pmatrix}
\mathcal{P} \left( -v \partial_x \tilde{v} + \left( \int_0^z \partial_x v(t,x,\xi) \, d\xi \right) \partial_z \tilde{v} \right) \\
-v \partial_x \tilde{\Theta} + \left( \int_0^z \partial_x v(t,x,\xi) \, d\xi \right) \partial_z \tilde{\Theta}
\end{pmatrix},
\end{align}
where $\psi = (v, \Theta)$ and $\tilde{\psi} = (\tilde{v}, \tilde{\Theta})$. System \eqref{ondimensional-equations-2} can then be expressed as
\begin{align}\label{operator-form}
\begin{cases}
\dfrac{d\psi}{dt} = \mathcal{L}_{\mathrm{R}} \psi + \mathcal{N}(\psi, \psi), \\
\psi|_{t=0} = \psi_0 \in X_1(D),
\end{cases}
\end{align}
with $\mathcal{L}_{\mathrm{R}} = \mathcal{A} + \mathrm{R} \mathcal{B}$ as defined in \eqref{linear-operator}.

According to Lemma \ref{pes-lemma}, the principle of exchange of stabilities, given by \eqref{pes-condition}, applies to this abstract formulation.
Define the center-unstable space:
\begin{align}
\mathcal{X}_c(D) = \left\{ x_1 \psi_{m_c,1}^{+,1} + x_2 \psi_{m_c,1}^{+,2} \,\middle|\, x_1, x_2 \in \mathbb{R} \right\},
\end{align}
and the stable space:
\begin{align}
\mathcal{X}_s(D) = \left\{ \psi \in \mathcal{X}_0(D) \,\middle|\, (\psi, \psi_{m_c,1}^{+,j})_{L^2} = 0, \ j = 1, 2 \right\}.
\end{align}
It follows that
\begin{align}
\mathcal{X}_0(D) = \mathcal{X}_c(D) \oplus \mathcal{X}_s(D).
\end{align}
Similarly, define
\begin{align}
\mathcal{X}_1(D) = \mathcal{X}_c(D) \oplus \tilde{\mathcal{X}}_s(D),
\end{align}
where
\[
\tilde{\mathcal{X}}_s(D) = \left\{ \psi \in \mathcal{X}_1(D) \,\middle|\, (\psi, \psi_{m_c,1}^{+,j})_{L^2} = 0, \ j = 1, 2 \right\}.
\]
We now leverage the abstract formulation \eqref{operator-form} of the system \eqref{ondimensional-equations-2} to prove Theorem~\ref{reduction}.

\subsection{Proof of Theorem \ref{reduction}}\label{proof-theorem24}
\begin{proof}

According to the exchange of stabilities condition \eqref{pes-condition}, we have $\beta_{m_c,1}^{+} (\mathrm{R}_c) = 0$ and $\frac{d\beta_{m_c,1}^{+}}{d\mathrm{R}}\big|_{\mathrm{R} = \mathrm{R}_c} >0$. Consequently, $(\mathrm{R}_c, \psi) = (\mathrm{R}_c, 0)$ represents a candidate bifurcation point for system \eqref{operator-form} \cite{Ma2019}. The authors of \cite{Han2021} established a general framework for analyzing stability and bifurcation in fluid systems governed by equations with a generic fourth-second order structure. Accordingly, we apply this framework to examine the dynamic bifurcation phenomena in the primitive equations \eqref{ondimensional-equations-2} at the critical parameter value $\mathrm{R} = \mathrm{R}_c$.

The center manifold function $\Psi$ is, by definition, a local mapping from $\mathcal{X}_c(D)$ into the stable subspace $\tilde{\mathcal{X}}_{s}(D) \subset \mathcal{X}_s(D) := (\mathcal{X}_c(D))^\perp$, whose graph is tangent to $\mathcal{X}_c(D)$ at the origin.
We decompose the solution $\psi$ into its center-unstable and stable components:
\begin{align}\label{decomp}
\psi = \psi_c + \psi_s, \quad \psi_c \in \mathcal{X}_c(D), \quad \psi_s \in \mathcal{X}_s(D),
\end{align}
where $\psi_c = P_c(\psi)$ and $\psi_s = P_s(\psi)$, with $P_c$ and $P_s$ representing the projections of the phase space $\mathcal{X}_0(D)$ onto the center-unstable subspace $\mathcal{X}_c(D)$ and the stable subspace $\mathcal{X}_s(D)$, respectively.

We express the center-unstable component of the solution as
\begin{align}\label{stable-space}
  \psi_c(t) = x_1(t) \psi_{m_c,1}^{+,1} + x_2(t) \psi_{m_c,1}^{+,2}.
\end{align}
Projecting system \eqref{operator-form} onto the center-unstable and stable subspaces yields
\begin{align}\label{equeee1}
\begin{aligned}
\frac{d\psi_c}{dt} &= \mathcal{L}_{\mathrm{R}} \psi_c + P_c \mathcal{N}(\psi_c + \psi_s, \psi_c + \psi_s), \\
\frac{d\psi_s}{dt} &= \mathcal{L}_{\mathrm{R}} \psi_s + P_s \mathcal{N}(\psi_c + \psi_s, \psi_c + \psi_s).
\end{aligned}
\end{align}

Taking the inner product of the first equation with the eigenvectors $\psi_{m_c,1}^{+,j}$ for $j = 1, 2$, we obtain the amplitude equations:
\begin{align} \label{abstract-equ2}
\begin{aligned}
\frac{dx_1}{dt} &= \beta_{m_c,1}^{+} x_1 + \left( \mathcal{N}(\psi_c + \psi_s, \psi_c + \psi_s), \psi_{m_c,1}^{+,1} \right), \\
\frac{dx_2}{dt} &= \beta_{m_c,1}^{+} x_2 + \left( \mathcal{N}(\psi_c + \psi_s, \psi_c + \psi_s), \psi_{m_c,1}^{+,2} \right), \\
\frac{d\psi_s}{dt} &= \mathcal{L}_{\mathrm{R}} \psi_s + P_s \mathcal{N}(\psi_c + \psi_s, \psi_c + \psi_s).
\end{aligned}
\end{align}

According to center manifold theory, the dynamics of the full system near criticality are governed by the amplitudes $x_1$ and $x_2$ of the center modes, as described by the first two equations above. Our objective is to derive a closed-form ordinary differential equation for $x_1$ and $x_2$ in \eqref{abstract-equ2} by approximating the center manifold function $\Psi$ and substituting
\begin{align}\label{decomp-2}
\psi_s = \Psi(\psi_c)
\end{align}
into the equations. This approach allows us to study the reduced dynamics on the center manifold.

Note that the time derivative of the center manifold function satisfies
\begin{align} \label{manifold}
\frac{d \Psi(\psi_c)}{dt} = \frac{d\Psi(\psi_c)}{d\psi_c} \frac{d \psi_c}{dt}
= \frac{d\Psi(\psi_c)}{d\psi_c} \left(
\frac{dx_1}{dt} \psi_{m_c,1}^{+,1} + \frac{dx_2}{dt} \psi_{m_c,1}^{+,2}
\right).
\end{align}
Using the expressions for $\dfrac{d\psi_c}{dt}$ and $\dfrac{d\psi_s}{dt}$ from \eqref{equeee1}, we derive the center manifold equation:
\begin{align} \label{comp1}
\begin{aligned}
&\mathcal{L}_{\mathrm{R}} \Psi(\psi_c) + P_s \mathcal{N}(\Psi(\psi_c) + \psi_c, \Psi(\psi_c) + \psi_c) \\
&= \frac{d\Psi(\psi_c)}{d\psi_c} \times \left( \mathcal{L}_{\mathrm{R}} \psi_c + P_c \mathcal{N}(\Psi(\psi_c) + \psi_c, \Psi(\psi_c) + \psi_c) \right).
\end{aligned}
\end{align}

By the tangency of the center manifold function at the origin, we obtain the second-order approximation
\begin{align} \label{tangency}
\begin{aligned} 
  \Psi(\psi_c) &= \Psi_2(\psi_c) + o(|\psi_c|^2)
  = \sigma (\psi_c)^2 + o(|\psi_c|^2) \\
  &= g_{11}x_1^2 + g_{12}x_1x_2 + g_{22}x_2^2 + o(x_1^2 + x_2^2),
\end{aligned}
\end{align}
where $\Psi_2$ denotes the quadratic component of the center manifold expansion.

Due to this tangency condition, the leading-order approximation of equation \eqref{comp1} becomes
\begin{align}
\begin{aligned}
\mathcal{L}_{\mathrm{R}} \Psi_2(\psi_c) + P_s \mathcal{N}(\psi_c, \psi_c)
= 2\sigma \beta_{m_c,1}^{+} (\psi_c)^2.
\end{aligned}
\end{align}
This yields the system of equations governing $g_{11}, g_{12}$ and $g_{22}$:
\begin{align}\label{hig-order1}
\begin{cases}
\mathcal{L}_{\mathrm{R}} g_{11} - 2g_{11} \beta_{m_c,1}^{+}
= -P_s \mathcal{N}\left( \psi_{m_c,1}^{+,1}, \psi_{m_c,1}^{+,1} \right), \\
\mathcal{L}_{\mathrm{R}} g_{12} - 2g_{12} \beta_{m_c,1}^{+}
= -P_s \left( \mathcal{N}\left( \psi_{m_c,1}^{+,1}, \psi_{m_c,1}^{+,2} \right)
+ \mathcal{N}\left( \psi_{m_c,1}^{+,2}, \psi_{m_c,1}^{+,1} \right) \right), \\
\mathcal{L}_{\mathrm{R}} g_{22} - 2g_{22} \beta_{m_c,1}^{+}
= -P_s \mathcal{N}\left( \psi_{m_c,1}^{+,2}, \psi_{m_c,1}^{+,2} \right).
\end{cases}
\end{align}
where the bilinear operator $\mathcal{N}(\psi, \tilde{\psi})$ is defined by
\[
\mathcal{N}(\psi, \tilde{\psi})
= \begin{pmatrix}
\mathcal{P} \left( -v \partial_x \tilde{v} + \left( \int_0^z \partial_x v(t,x,\xi) \, d\xi \right) \partial_z \tilde{v} \right) \\
-v \partial_x \tilde{\Theta} + \left( \int_0^z \partial_x v(t,x,\xi) \, d\xi \right) \partial_z \tilde{\Theta}
\end{pmatrix},
\]
with $\psi = (v, \Theta)$ and $\tilde{\psi} = (\tilde{v}, \tilde{\Theta})$.

Using the eigenvector expression \eqref{first-eigenvector} and performing direct computations, we obtain
\[
\begin{aligned}
&\begin{aligned}
\mathcal{N}\left(\psi_{m_c,1}^{+,1},\psi_{m_c,1}^{+,1}\right)_1&=\mathcal{P}\left(-v\partial_x v
+\left(\int_0^z\partial_xv(t,x,\xi)
\,d\xi \right)\partial_zv\right)\\&=\mathcal{P}\left(
\left(v_{m_c,1}^{+,1}\right)^2
\frac{\pi m_c}{\alpha}
\sin\frac{4\pi m_cx}{\alpha}\cos^2\pi z\right)
\\&\quad+\mathcal{P}\left(
\left(v_{m_c,1}^{+,1}\right)^2
\frac{\pi m_c}{\alpha}
\sin\frac{4\pi m_cx}{\alpha}\sin^2\pi z\right)
\\
&=\mathcal{P}\left(\left(v_{m_c,1}^{+,1}\right)^2
\frac{\pi m_c}{\alpha}\sin\frac{4\pi m_cx}{\alpha}\right)=0,
\end{aligned},\\
&\begin{aligned}
\mathcal{N}\left(\psi_{m_c,1}^{+,1},\psi_{m_c,1}^{+,1}\right)_2&=-v\partial_x \Theta+\left(\int_0^z\partial_xv(t,x,\xi)
\,d\xi \right)\partial_z \Theta\\&=-
\left(v_{m_c,1}^{+,1}\right)^2
\frac{A_{m_c}\pi m_c}{\alpha}
\cos^2\frac{2\pi m_cx}{\alpha}\sin 2\pi z
\\&\quad+\left(-
\left(v_{m_c,1}^{+,1}\right)^2
\frac{A_{m_c}\pi m_c}{\alpha}
\sin^2\frac{2\pi m_cx}{\alpha}\sin 2\pi z\right)
\\
&=-
\left(v_{m_c,1}^{+,1}\right)^2
\frac{A_{m_c}\pi m_c}{\alpha}\sin 2\pi z,
\end{aligned}
\end{aligned}
\]
\[
\begin{aligned}
&\begin{aligned}
\mathcal{N}\left(\psi_{m_c,1}^{+,2},\psi_{m_c,1}^{+,2}\right)_1&=\mathcal{P}\left(-v\partial_x v
+\left(\int_0^z\partial_xv(t,x,\xi)
\,d\xi \right)\partial_zv\right)\\&=\mathcal{P}\left(-
\left(v_{m_c,1}^{+,1}\right)^2
\frac{\pi m_c}{\alpha}
\sin\frac{4\pi m_cx}{\alpha}\cos^2\pi z\right)
\\&\quad-\mathcal{P}\left(
\left(v_{m_c,1}^{+,1}\right)^2
\frac{\pi m_c}{\alpha}
\sin\frac{4\pi m_cx}{\alpha}\sin^2\pi z\right)
\\
&=\mathcal{P}\left(-\left(v_{m_c,1}^{+,1}\right)^2
\frac{\pi m_c}{\alpha}\sin\frac{4\pi m_cx}{\alpha}\right)=0,
\end{aligned}\\
&\begin{aligned}
\mathcal{N}\left(\psi_{m_c,1}^{+,2},\psi_{m_c,1}^{+,2}\right)_2&=-v\partial_x \Theta+\left(\int_0^z\partial_xv(t,x,\xi)
\,d\xi \right)\partial_z \Theta\\&=-
\left(v_{m_c,1}^{+,1}\right)^2
\frac{A_{m_c}\pi m_c}{\alpha}
\sin^2\frac{2\pi m_cx}{\alpha}\sin 2\pi z
\\&\quad+\left(-
\left(v_{m_c,1}^{+,1}\right)^2
\frac{A_{m_c}\pi m_c}{\alpha}
\cos^2\frac{2\pi m_cx}{\alpha}\sin 2\pi z\right)
\\
&=-
\left(v_{m_c,1}^{+,1}\right)^2
\frac{A_{m_c}\pi m_c}{\alpha}\sin 2\pi z,
\end{aligned}
\end{aligned}
\]
\[
\begin{aligned}
&\begin{aligned}
\mathcal{N}\left(\psi_{m_c,1}^{+,1},\psi_{m_c,1}^{+,2}\right)_1&=\mathcal{P}\left(-v\partial_x \tilde{v}
+\left(\int_0^z\partial_xv(t,x,\xi)
\,d\xi \right)\partial_z\tilde{v}\right)\\&=\mathcal{P}\left(-
\left(v_{m_c,1}^{+,1}\right)^2
\frac{2\pi m_c}{\alpha}
\cos^2\frac{2\pi m_cx}{\alpha}\cos^2\pi z\right)
\\&\quad+\mathcal{P}\left(
\left(v_{m_c,1}^{+,1}\right)^2
\frac{2\pi m_c}{\alpha}
\sin^2\frac{2\pi m_cx}{\alpha}\sin^2\pi z\right),
\end{aligned}\\
&\begin{aligned}
\mathcal{N}\left(\psi_{m_c,1}^{+,2},\psi_{m_c,1}^{+,1}\right)_1&=\mathcal{P}\left(-v\partial_x \tilde{v}
+\left(\int_0^z\partial_xv(t,x,\xi)
\,d\xi \right)\partial_z\tilde{v}\right)\\&=\mathcal{P}\left(
\left(v_{m_c,1}^{+,1}\right)^2
\frac{2\pi m_c}{\alpha}
\sin^2\frac{2\pi m_cx}{\alpha}\cos^2\pi z\right)
\\&\quad-\mathcal{P}\left(
\left(v_{m_c,1}^{+,1}\right)^2
\frac{2\pi m_c}{\alpha}
\cos^2\frac{2\pi m_cx}{\alpha}\sin^2\pi z\right),
\end{aligned}
\end{aligned}
\]
and 
\[
\begin{aligned}
&\begin{aligned}
\mathcal{N}\left(\psi_{m_c,1}^{+,1},\psi_{m_c,1}^{+,2}\right)_2&=-v\partial_x \tilde{\Theta}
+\left(\int_0^z\partial_xv(t,x,\xi)
\,d\xi \right)\partial_z\tilde{\Theta}\\&=-
\left(v_{m_c,1}^{+,1}\right)^2
\frac{A_{m_c}\pi m_c}{2\alpha}
\sin\frac{4\pi m_cx}{\alpha}\sin2\pi z
\\&\quad+
\left(v_{m_c,1}^{+,1}\right)^2
\frac{A_{m_c}\pi m_c}{2\alpha}
\sin\frac{4\pi m_cx}{\alpha}\sin2\pi z=0,
\end{aligned}\\
&\begin{aligned}
\mathcal{N}\left(\psi_{m_c,1}^{+,2},\psi_{m_c,1}^{+,1}\right)_2&=-v\partial_x \tilde{v}
+\left(\int_0^z\partial_xv(t,x,\xi)
\,d\xi \right)\partial_z\tilde{v}\\&=-
\left(v_{m_c,1}^{+,1}\right)^2
\frac{A_{m_c}\pi m_c}{2\alpha}
\sin\frac{4\pi m_cx}{\alpha}\sin2\pi z
\\&\quad+
\left(v_{m_c,1}^{+,1}\right)^2
\frac{A_{m_c}\pi m_c}{2\alpha}
\sin\frac{4\pi m_cx}{\alpha}\sin2\pi z=0.
\end{aligned}
\end{aligned}
\]
Collecting the preceding computing results, we have
\[
\begin{aligned}
&\mathcal{N}\left(\psi_{m_c,1}^{+,1},\psi_{m_c,1}^{+,1}\right)
=-\left(v_{m_c,1}^{+,1}\right)^2\begin{pmatrix}
0\\
\frac{A_{m_c}m_c}{4\alpha}\sin 2\pi z
\end{pmatrix},\\
&\mathcal{N}\left(\psi_{m_c,1}^{+,2},\psi_{m_c,1}^{+,2}\right)
=-\left(v_{m_c,1}^{+,1}\right)^2\begin{pmatrix}
0\\
\frac{A_{m_c}m_c}{4\alpha}\sin 2\pi z
\end{pmatrix},\\
&\mathcal{N}\left(\psi_{m_c,1}^{+,1},\psi_{m_c,1}^{+,2}\right)+
\mathcal{N}\left(\psi_{m_c,2}^{+,2},\psi_{m_c,1}^{+,1}\right)=0.
\end{aligned}
\]
Thus, the system \eqref{hig-order1} become
\begin{align}\label{hig-order1-1}
\begin{cases}
\mathcal{L}_{\text{R}}g_{11}-2g_{11}\beta_{m_c,1}^{+}=\left(v_{m_c,1}^{+,1}\right)^2\begin{pmatrix}
0\\
\frac{A_{m_c} \pi m_c}{2\alpha}\sin 2\pi z
\end{pmatrix},\\
\mathcal{L}_{\text{R}}g_{12}-2g_{12}\beta_{m_c,1}^{+}=0,\\
\mathcal{L}_{\text{R}}g_{22}-2g_{22}\beta_{m_c,1}^{+}=\left(v_{m_c,1}^{+,1}\right)^2\begin{pmatrix}
0\\
\frac{A_{m_c}\pi m_c}{2\alpha}\sin 2\pi z
\end{pmatrix}.
\end{cases}
\end{align}
Solving the system \eqref{hig-order1-1}, we get
\[
g_{11}=g_{22}=g_{12}=\begin{pmatrix}
0\\
-\frac{A_{m_c}\pi m_c\left(\psi_{m_c,1}^{+,1}\right)^2}{2\alpha\left(
4\kappa_{a} \pi^2+2\beta_{m_c,1}^{+}
\right)}\sin 2\pi z
\end{pmatrix},\quad
g_{12}=\begin{pmatrix}
0\\
0\end{pmatrix}.
\]
On the center manifold, the reduced system reads
\begin{align}\label{reduced-new}
\begin{aligned}
&\frac{dx_1}{dt}=\beta_{m_c,1}^{+}
x_1+ \left( \mathcal{N}(\psi_c + \psi_s, \psi_c + \psi_s), \psi_{m_c,1}^{+,1} \right),\\
&\frac{dx_2}{dt}=\beta_{m_c,1}^{+}
x_2+ \left( \mathcal{N}(\psi_c + \psi_s, \psi_c + \psi_s), \psi_{m_c,1}^{+,2} \right).
\end{aligned}
\end{align}

Taking $\psi_s=g_{11}x_1^2  + g_{22}x_2^2$ and 
substituting the preceding explicit expressions
for $g_{11}$ and $g_{22}$
 into the nonlinear terms \eqref{reduced-new}, we get
\begin{align}
\begin{aligned}
&\begin{aligned}
& \left( \mathcal{N}(\psi_c + \psi_s, \psi_c + \psi_s), \psi_{m_c,1}^{+,1} \right)=x_1\left(
x_1^2+x_2^2\right)
\left(\mathcal{N}\left(
\psi_{m_c,1}^{+,1},g_{11}\right),\psi_{m_c,1}^{+,1}\right)\\
&=\frac{\left(16-3 \pi ^2\right)A_{m_c}^2m_c^2\left(v_{m_c,1}^{+,1}\right)^2}{12\alpha 
\left(
4\kappa_{a} \pi^2+2\beta_{m_c,1}^{+}
\right)}
x_1\left(
x_1^2+x_2^2\right)+o\left(\abs{\mathbf{x}}^3\right),
\end{aligned}\\
&\begin{aligned}
& \left( \mathcal{N}(\psi_c + \psi_s, \psi_c + \psi_s), \psi_{m_c,1}^{+,2} \right)=x_2\left(
x_1^2+x_2^2\right)
\left(\mathcal{N}\left(
\psi_{m_c,1}^{+,2},g_{22}\right),\psi_{m_c,1}^{+,2}\right)\\
&=
\frac{\left(16-3 \pi ^2\right)A_{m_c}^2m_c^2\left(v_{m_c,1}^{+,1}\right)^2}{12\alpha 
\left(
4\kappa_{a} \pi^2+2\beta_{m_c,1}^{+}
\right)}
x_2\left(
x_1^2+x_2^2\right)+o\left(\abs{\mathbf{x}}^3\right),
\end{aligned}
\end{aligned}
\end{align}
where we have used
\[
\begin{aligned}
&\left(\mathcal{N}\left(
\psi_{m_c,1}^{+,1},\psi_{m_c,1}^{+,1}\right),\psi_{m_c,1}^{+,1}\right)=0,
\quad \left(\mathcal{N}\left(
\psi_{m_c,1}^{+,2},\psi_{m_c,1}^{+,2}\right),\psi_{m_c,1}^{+,2}\right)=0,
\\
&\left(\mathcal{N}\left(
\psi_{m_c,1}^{+,1},\psi_{m_c,1}^{+,2}\right),\psi_{m_c,1}^{+,1}\right)+
\left(\mathcal{N}\left(
\psi_{m_c,1}^{+,2},\psi_{m_c,1}^{+,1}\right),\psi_{m_c,1}^{+,1}\right)=0,
\\
&\left(\mathcal{N}\left(
\psi_{m_c,1}^{+,2},\psi_{m_c,1}^{+,1}\right),\psi_{m_c,1}^{+,2}\right)+
\left(\mathcal{N}\left(
\psi_{m_c,1}^{+,1},\psi_{m_c,1}^{+,2}\right),\psi_{m_c,1}^{+,2}\right)=0,
\\
&\left(\mathcal{N}\left(
\psi_{m_c,1}^{+,2},\psi_{m_c,1}^{+,2}\right),\psi_{m_c,1}^{+,1}\right)=0,
\quad
\left(\mathcal{N}\left(
\psi_{m_c,1}^{+,1},\psi_{m_c,1}^{+,1}\right),\psi_{m_c,1}^{+,2}\right)=0,
\\
&\left(\mathcal{N}\left(
\psi_{m_c,1}^{+,2},g_{11}\right),\psi_{m_c,1}^{+,1}\right)=0,\quad
\left(\mathcal{N}\left(g_{11},
\psi_{m_c,1}^{+,2}\right),\psi_{m_c,1}^{+,1}\right)=0,\\
&\left(\mathcal{N}\left(
\psi_{m_c,1}^{+,1},g_{11}\right),\psi_{m_c,1}^{+,2}\right)=0,\quad
\left(\mathcal{N}\left(g_{11},
\psi_{m_c,1}^{+,1}\right),\psi_{m_c,1}^{+,2}\right)=0,\\
&\left(\mathcal{N}\left(
\psi_{m_c,1}^{+,2},g_{22}\right),\psi_{m_c,1}^{+,1}\right)=0,\quad
\left(\mathcal{N}\left(g_{22},
\psi_{m_c,1}^{+,2}\right),\psi_{m_c,1}^{+,1}\right)=0,\\
&\left(\mathcal{N}\left(
\psi_{m_c,1}^{+,1},g_{22}\right),\psi_{m_c,1}^{+,2}\right)=0,\quad
\left(\mathcal{N}\left(g_{22},
\psi_{m_c,1}^{+,1}\right),\psi_{m_c,1}^{+,2}\right)=0,\\
&\left(\mathcal{N}\left(
g_{11},\psi_{m_c,1}^{+,1}\right),\psi_{m_c,1}^{+,1}\right)=0,\quad
\left(\mathcal{N}\left(
g_{22},\psi_{m_c,1}^{+,2}\right),\psi_{m_c,1}^{+,2}\right)=0.
\end{aligned}
\]

Taking the leading orders of system \eqref{reduced-new}, we get the final reduced equations as follows
\begin{align}\label{req-m1}
\begin{cases}
\frac{dx_1}{dt}=\beta_{m_c,1}^{+}
x_1+l
x_1\left(
x_1^2+x_2^2\right),\\
\frac{dx_2}{dt}=\beta_{m_c,1}^{+}
x_2+lx_2\left(
x_1^2+x_2^2\right),
\end{cases}
\end{align}
where
\begin{align}
l=\frac{\left(16-3 \pi ^2\right)A_{m_c}^2m_c^2\left(v_{m_c,1}^{+,1}\right)^2}{12\alpha 
\left(
4\kappa_{a} \pi^2+2\beta_{m_c,1}^{+}
\right)}<0,
\end{align}
in which $A_{m_c}$ and $\left(v_{m_c,1}^{+,1}\right)^{2}$ are given in \eqref{A-V}.

For the system \eqref{req-m1}, we note that $l < 0$, indicating a supercritical bifurcation at $\mathrm{R} = \mathrm{R}_c$. In this regime, the system admits infinitely many stable steady-state solutions $(x_1, x_2) = (s_1, s_2)$ for $\mathrm{R} > \mathrm{R}_c$, which are determined by the following equation
\[
s_1^2 + s_2^2 = -\beta_{m_c,1}^{+}/l.
\]
From the evolution equation
\[
\frac{d}{dt}\left( x_1^2 + x_2^2 \right) =
2\left( x_1^2 + x_2^2 \right) \left( \beta_{m_c,1}^{+} + l \left( x_1^2 + x_2^2 \right) \right),
\]
we conclude that these solutions constitute a local ring attractor for system \eqref{reduced-new}.
\end{proof}
\subsection{Proof of Theorem \ref{bifurcation}}\label{proof-theorem25}
\begin{proof}
The dynamics of the system \eqref{operator-form} with $\mathrm{R}$ near $\mathrm{R}_c$ are completely reduced to those of the ODE system \eqref{req-m1}. Consequently, the bifurcation type of the system \eqref{operator-form} at $\mathrm{R} = \mathrm{R}_c$ is determined by the bifurcation type of \eqref{req-m1} at $\mathrm{R} = \mathrm{R}_c$, which is supercritical. Moreover, the bifurcating solutions can be obtained by analyzing the reduced system \eqref{req-m1}.

Let $(s_1, s_2)$ be an arbitrary point on the blue circle $s_1^2 + s_2^2 = -\beta_{m_c,1}^{+}/l$, as illustrated in Fig.~\ref{attractor-ring}.. Then, based on \eqref{stable-space} and \eqref{decomp-2}, the corresponding bifurcation solution $\psi_b$ of the system \eqref{operator-form}, approximated to second order, takes the form:
\begin{align}\label{flow-p}
\begin{aligned}
\psi_b=\psi_c+\psi_s&=
s_1\psi_{m_c,1}^{+,1}+s_2\psi_{m_c,1}^{+,2}\\&\quad-
\abs{\mathbf{s}}^2
\begin{pmatrix}
0\\\frac{A_{m_c}\pi m_c\left(\psi_{m_c,1}^{+,1}\right)^2}{2\alpha\left(
4\kappa_{a} \pi^2+2\beta_{m_c,1}^{+}
\right)}\sin 2\pi z
\end{pmatrix}
+o(\abs{\mathbf{s}}^2),\quad 
\abs{\mathbf{s}}^2=\frac{\beta_{m_c,1}^{+}}{l}.
\end{aligned}
\end{align}

Recalling that 
\[
\begin{aligned}
\psi_{m_c,1}^{+,1} &=
\left(
\frac{\alpha }{4}+\frac{\alpha  }{4} A_{m_c}^2\right)^{-1}
\begin{pmatrix}
\cos \left( \dfrac{2\pi m_c x}{\alpha} \right) \cos (\pi z) \\[1.5ex]
A_{m_c} \sin \left( \dfrac{2\pi m_c x}{\alpha} \right) \sin (\pi z)
\end{pmatrix}, \\[3ex]
\psi_{m_c,1}^{+,2} &=
\left(
\frac{\alpha }{4}+\frac{\alpha  }{4} A_{m_c}^2\right)^{-1}
\begin{pmatrix}
\sin \left( \dfrac{2\pi m_c x}{\alpha} \right) \cos (\pi z) \\[1.5ex]
-A_{m_c} \cos \left( \dfrac{2\pi m_c x}{\alpha} \right) \sin (\pi z)
\end{pmatrix},
\end{aligned}
\]
substitutting these expressions into \eqref{flow-p}, we have
\[
\begin{aligned}
\psi_b=&s_1\left(
\frac{\alpha }{4}+\frac{\alpha  }{4} A_{m_c}^2\right)^{-1}
\begin{pmatrix}
\cos \left( \dfrac{2\pi m_c x}{\alpha} \right) \cos (\pi z) \\[1.5ex]
A_{m_c} \sin \left( \dfrac{2\pi m_c x}{\alpha} \right) \sin (\pi z)
\end{pmatrix}
\\&+s_2
\left(
\frac{\alpha }{4}+\frac{\alpha  }{4} A_{m_c}^2\right)^{-1}
\begin{pmatrix}
\sin \left( \dfrac{2\pi m_c x}{\alpha} \right) \cos (\pi z) \\[1.5ex]
-A_{m_c} \cos \left( \dfrac{2\pi m_c x}{\alpha} \right) \sin (\pi z)
\end{pmatrix}\\
&-\abs{\mathbf{s}}^2
\frac{A_{m_c}\pi m_c\left(
\frac{\alpha }{4}+\frac{\alpha  }{4} A_{m_c}^2\right)^{-2}.}{2\alpha\left(
4\kappa_{a} \pi^2+2\beta_{m_c,1}^{+}
\right)}
\begin{pmatrix}
0\\ \sin 2\pi z
\end{pmatrix}
+o(\abs{\mathbf{s}}^2).
\end{aligned}
\]
Reverting to the dimensional formulation, the solution in the original variables $(v^s_b, T_b^s)$ admits the leading-order asymptotic behavior characterized by
\begin{align*}
\begin{pmatrix}
v_b^s \\
T_b^s
\end{pmatrix}
= s_1
\begin{pmatrix}
V_1 \\
\Theta_1
\end{pmatrix}
+ s_2
\begin{pmatrix}
V_2 \\
\Theta_2
\end{pmatrix}
+ \abs{\mathbf{s}}^2
\begin{pmatrix}
V_3 \\
\Theta_3
\end{pmatrix}
+ o(\abs{\mathbf{s}}^2),
\end{align*}
where $\abs{\mathbf{s}}^2 = \beta_{m_c,1}^{+} / l$, and the vector profiles $(V_j, \Theta_j)^\top$ for $j = 1, 2, 3$ are defined as follows:
\begin{align*}
\begin{aligned}
\begin{pmatrix}
V_1 \\
\Theta_1
\end{pmatrix} &=
\left( \frac{L}{4H} + \frac{L}{4H} A_{m_c}^2 \right)^{-1}
\begin{pmatrix}
\frac{\kappa_x}{H}\cos \left( \frac{2\pi m_c x}{L} \right) \cos \left( \frac{\pi z}{H} \right) \\
\abs{T}_0A_{m_c} \sin \left( \frac{2\pi m_c x}{L} \right) \sin \left( \frac{\pi z}{H} \right)
\end{pmatrix}, \\
\begin{pmatrix}
V_2 \\
\Theta_2
\end{pmatrix} &=
\left( \frac{L}{4H} + \frac{L}{4H} A_{m_c}^2 \right)^{-1}
\begin{pmatrix}
\frac{\kappa_x}{H}\sin \left( \frac{2\pi m_c x}{L} \right) \cos \left( \frac{\pi z}{H} \right) \\
-\abs{T}_0A_{m_c} \cos \left( \frac{2\pi m_c x}{L} \right) \sin \left( \frac{\pi z}{H} \right)
\end{pmatrix}, \\
\begin{pmatrix}
V_3 \\
\Theta_3
\end{pmatrix} &=
- \frac{A_{m_c} \kappa_x \pi m_c}{2\alpha \left( 4 \kappa_z \pi^2 + 2 \kappa_x \beta_{m_c,1}^{+} \right)}
\left( \frac{L}{4H} + \frac{L}{4H} A_{m_c}^2 \right)^{-2}
\begin{pmatrix}
0 \\
\abs{T}_0\sin \left( \frac{2\pi z}{H} \right)
\end{pmatrix}.
\end{aligned}
\end{align*}

\end{proof}

\subsection{Examples}

For the system \eqref{primitive-reformulation-p}, we have show that 
as the temperature $T_0-T_1>T_c$, it bifurcates into an infinite number 
of nontrivial steady-state solutions, forming a local ring attractor,
as illustrated in Fig.~\ref{attractor-ring}. In this section, we give a specific example to
show that each nontrivial steady-state solution contained in the ring attractor
could describe the large-scale thermal convection phenomena.

Recall that the velocity component $v_b^s$ of bifurcation solutions
reads
\[
v_b^s=\frac{\kappa_x}{H}\left(
\frac{L }{4H}+\frac{L}{4H} A_{m_c}^2\right)^{-1}
\left(
s_1\cos \frac{2\pi m_cx}{L}\cos \frac{\pi z}{H}
+s_2\sin \frac{2\pi m_cx}{L}\cos \frac{\pi z}{H}
\right)++o(\abs{\mathbf{s}}^2).
\]
The corresponding stream function $\psi_b^s$ 
satisfying $v_b^s=-\partial_z\psi_b^s$
reads
\[
\psi_b^s=-\frac{\kappa_x}{\pi}\left(
\frac{L}{4H}+\frac{L}{4H} A_{m_c}^2\right)^{-1}
\left(
s_1\cos \frac{2\pi m_cx}{L}\sin \frac{\pi z}{H}
+s_2\sin \frac{2\pi m_cx}{L}\sin \frac{\pi z}{H}
\right)+o(\abs{\mathbf{s}}^2).
\]
We then plot the stream function $\psi_b^s$ 
to show that the flow behind has convection cells. To this end,
we select red points $\mathbf{s}_j$ $(j=1,\cdots,4)$ in Fig.~\ref{attractor-ring}
to analyze the structure of the flow.
\[
\begin{aligned}
&\psi_b^1=-\frac{\kappa_x}{\pi}
\sqrt{ -\frac{\beta_{m_c,1}^{+}}{l}}
\left(
\frac{L}{4H}+\frac{L}{4H} A_{m_c}^2\right)^{-1}
\cos \frac{2\pi m_cx}{L}\sin \frac{\pi z}{H}
+o(\abs{\mathbf{s}}^2),\\
&\psi_b^2=-\frac{\kappa_x}{\pi}
\sqrt{ -\frac{\beta_{m_c,1}^{+}}{l}}
\left(
\frac{L}{4H}+\frac{L}{4H} A_{m_c}^2\right)^{-1}
\sin \frac{2\pi m_cx}{L}\sin \frac{\pi z}{H}
+o(\abs{\mathbf{s}}^2),\\
&\psi_b^3=\frac{\kappa_x}{\pi}
\sqrt{ -\frac{\beta_{m_c,1}^{+}}{l}}
\left(
\frac{L}{4H}+\frac{L}{4H} A_{m_c}^2\right)^{-1}
\cos \frac{2\pi m_cx}{L}\sin \frac{\pi z}{H}
+o(\abs{\mathbf{s}}^2),\\
&\psi_b^4=\frac{\kappa_x}{\pi}
\sqrt{ -\frac{\beta_{m_c,1}^{+}}{l}}
\left(
\frac{L}{4H}+\frac{L}{4H} A_{m_c}^2\right)^{-1}
\sin \frac{2\pi m_cx}{L}\sin \frac{\pi z}{H}
+o(\abs{\mathbf{s}}^2).
\end{aligned}
\]

\begin{example}
For all parameters in the equations \eqref{primitive-reformulation-p}, we take their standard units. Taking $T_0=20, H=1000, L=50000, \nu_{z}=10, \kappa_z=1, \nu_{x}=10000, \kappa_x=100, \rho_0=1.2,\beta=0.0001, g=9.8$, we have $m_c=2, T_c=14.6021, T_1=5.3979, R_c=  \frac{H \sqrt{H \rho_0 g \beta T_c}}{\kappa_x}= 41.4392$.
In this case, the structure of the flow given by $\psi_b^1-\psi_b^4$
are shown in Fig.~\ref{Critical_p11}. From Fig.~\ref{Critical_p11}, we see that $2m_c=4$ gives the number of cells contained in the convection solutions $\psi_b^1-\psi_b^4$. In addition, although $\psi_b^1-\psi_b^4$ are different solutions,
the flow structures behind which are essentially the same.
\end{example}

\begin{figure}[h]
  \centering
  \includegraphics[width=.8\textwidth,height=.5\textwidth]{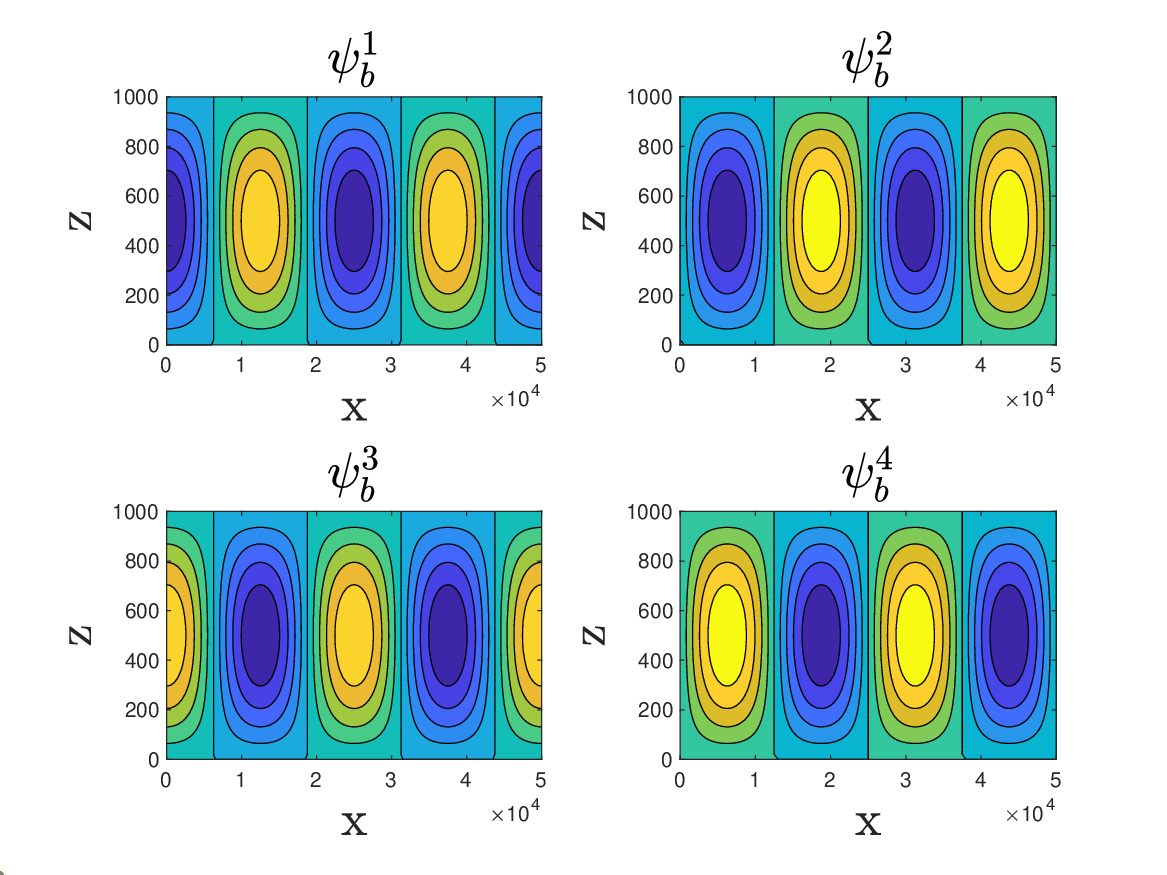}\\
  \caption{Bifurcation solutions.
  }\label{Critical_p11}
\end{figure}

\section{Appendix}
\begin{appendix}
\begin{lemma}\label{general-inequality1}
Let $f$ and $g$ be two nonnegative function defined on $[0,+\infty)$. Suppose that there are two positive numbers $\lambda_1$,
$\lambda_2$ such that 
\begin{align}
\frac{df}{dt}+\lambda_1f\leq g,\quad 
g\leq D_1(t)e^{-\lambda_2 t}+D_2(t).
\end{align}
where $D_j(t) (j=1,2)$ are two functions satisfying
\[
D_j(t) (j=1,2)\in L^1([0,+\infty)).
\]
Then, there exists a positive number $\lambda$ such that
\begin{align}
f\leq e^{-\lambda t}
\left(
f(0)+\int_0^{\infty}D_1(\tau)\,d\tau
\right)
+\int_0^{\infty}D_2(\tau)\,d\tau.
\end{align}
Particular, if $D_2(t)\equiv 0$, we have 
\begin{align}
f\leq e^{-\lambda t} \left(
f(0)+\int_0^{\infty}D_1(\tau)\,d\tau
\right).
\end{align}
\end{lemma}
\begin{proof}
Let $\epsilon$ be a small positive number such that $(\lambda_2-\epsilon \lambda_1)>0$, we have
\[
\begin{aligned}
\frac{df}{dt}e^{\epsilon \lambda_1t}+\epsilon \lambda_1e^{\epsilon \lambda_1t}f\leq D_1e^{-(\lambda_2-\epsilon \lambda_1) t}+D_2
e^{\epsilon \lambda_1t}
\end{aligned}
\]
Integrating the preceding inequality, we have
\[
\begin{aligned}
f&\leq e^{-\epsilon \lambda_1t}f(0)+
e^{-\epsilon \lambda_1t}\int_0^tD_1(\tau)e^{-(\lambda_2-\epsilon \lambda_1) \tau}\,d\tau+
e^{-\epsilon \lambda_1t}
\int_0^tD_2(\tau)e^{\epsilon \lambda_1\tau}\,d\tau\\
&\leq
e^{-\epsilon \lambda_1t}f(0)+
e^{-\epsilon \lambda_1t}\int_0^{\infty}D_1(\tau)\,d\tau
+
\int_0^tD_2(\tau)e^{-\epsilon \lambda_1(t-\tau)}\,d\tau\\
&\leq e^{-\epsilon \lambda_1t}
\left(
f(0)+\int_0^{\infty}D_1(\tau)\,d\tau
\right)
+\int_0^{\infty}D_2(\tau)\,d\tau.
\end{aligned}
\]
\end{proof}
\end{appendix}

\begin{lemma}\label{mixed-norm}
Let $f\in H^1(D)$ be a function
 periodic in $x$ with period $\alpha$, then there exists a constant $C>0$ such that
 \begin{align}\label{appeed20}
  \begin{aligned}
 &\norm{f}_{L_x^{\infty}L^2_z}\leq 
C\norm{f}_{L^2}\left(\norm{f}_{L^2}+\norm{\partial_xf}_{L^2}\right),\\
&\norm{f}_{L_z^{\infty}L^2_x}\leq 
C\norm{f}_{L^2}\left(\norm{f}_{L^2}+\norm{\partial_zf}_{L^2}\right).
\end{aligned}
\end{align}
Particularly, we have 
 \begin{align}\label{appeed2}
  \begin{aligned}
 &\norm{f}_{L_x^{\infty}L^2_z}\leq 
C\norm{f}_{L^2}\norm{\partial_xf}_{L^2},\quad \text{if}\quad
\int_0^{\alpha}f(x,z)\,dx=0,
\\
&\norm{f}_{L_z^{\infty}L^2_x}\leq 
C\norm{f}_{L^2}\norm{\partial_zf}_{L^2},
\quad \text{if}\quad
\int_0^{z}f(x,z)\,dz=0.
\end{aligned}
\end{align}
\end{lemma}
\begin{proof}
Becasue $C^1(\overline{D})$ is dense in $H^1(D)$, we only need prove
the conclusion for case of $f\in C^1(\overline{D})$. Let $x_0$ be the minimum point of $\int_{0}^1\left(f(x,z)\right)^2\,dz$, we have
\[
\begin{aligned}
\int_{0}^1\left(f(x,z)\right)^2\,dz&=
\int_{0}^1\left(f(x_0,z)\right)^2\,dz
+2\int_{x_0}^x\left(\int_{0}^1f(\sigma,z) \partial_{\sigma}f(\sigma,z) \,dz\right)\,d\sigma\\
&\leq \frac{1}{\alpha}
\int_{D}f^2\,dx\,dz
+2\int_{x_0}^x\left(\int_{0}^1f(\sigma,z) \partial_{\sigma}f(\sigma,z) \,dz\right)\,d\sigma\\
&\leq C\norm{f}_{L^2}\left(\norm{f}_{L^2}+\norm{\partial_xf}_{L^2}\right),
\end{aligned}
\]
which yields $\eqref{appeed20}_1$. $\eqref{appeed20}_2$. 
can be proved similarly.
If $\int_0^{\alpha}f(x,z)\,dx=0$, then there exists $x_0$ such that
$f(x_0,z)=0$. This enables us to get
\[
\int_0^1\left(f(x,z)\right)^2\,dz
=2\int_0^1\left(\int_{x_0}^xf(y,z)\partial_{y}f(y,z)\,dy\right)
\,dz\leq 2\norm{f}_{L^2}\norm{\partial_xf}_{L^2}.
\]
This yields $\eqref{appeed2}_1$. $\eqref{appeed2}_2$
can be proved similarly.
\end{proof}

\begin{lemma}\label{mixed-norm-1}
Let $f\in H^1(D)$ be a function
 periodic in $x$ with period $\alpha$ such that
 \[
 \int_0^1f(x,z)\,dz=0,
 \]
  then there exists a constant $C>0$ such that
 \begin{align}\label{appeed22}
  \begin{aligned}
 &\norm{f}_{L_x^{2}L^{\infty}_z}\leq 
C\norm{f}_{L^2}\norm{\partial_zf}_{L^2}.
\end{aligned}
\end{align}
\end{lemma}
\begin{proof}
Becasue $C^1(\overline{D})$ is dense in $H^1(D)$, we only need prove
the conclusion for case of $f\in C^1(\overline{D})$. 
$\int_0^{\alpha}f(x,z)\,dz=0$ means that there exists $z_0$ such that
$f(x,z_0)=0$. This enables us to get
\[
\abs{f(x,z)}^2\leq 2\int_{z_0}^z\abs{f(x,\xi)\partial_{\xi}f(x,\xi)}\,d\xi
\leq 2\int_{0}^1\abs{f(x,\xi)\partial_{\xi}f(x,\xi)}\,d\xi.
\]
This enables us to get 
\[
  \begin{aligned}
\norm{f}_{L_x^{2}L^{\infty}_z}\leq 
C\norm{f}_{L^2}\norm{\partial_zf}_{L^2}.
\end{aligned}
\]
\end{proof}

 \section*{Acknowledgement}

Q. Wang was supported by the Natural Science Foundation
of Sichuan Province (No.2025ZNSFSC0072).

    \section*{Conflict Of Interest Statement}
We declare that we have no financial and personal relationships with
other people or organizations that can inappropriately influence our
work, there is no professional or other personal interest of any nature
or kind in any product, service and/or company that could be
construed as influencing the position presented in, or the review of,
the manuscript.

    \section*{Data Availability Statement}
Data sets generated during the current study are available
from the corresponding author on reasonable request.

\itemsep=0pt

\begin{thebibliography}{10}

\bibitem{Busse1965}
A.~Schl\"{u}ter, D.~Lortz, and F.~Busse.
\newblock On the stability of steady finite amplitude convection.
\newblock {\em J. Fluid Mech.}, 23:129--144, 1965.

\bibitem{busse1979instabilities}
F.~H. Busse and R.~M. Clever.
\newblock Instabilities of convection rolls in a fluid of moderate prandtl
  number.
\newblock {\em J. Fluid Mech.}, 91(2):319--335, 1979.

\bibitem{Getling1998}
A.~V. Getling.
\newblock {\em Rayleigh-{B}\'{e}nard convection}, volume~11 of {\em Advanced
  Series in Nonlinear Dynamics}.
\newblock World Scientific Publishing Co., Inc., River Edge, NJ, 1998.
\newblock Structures and dynamics.

\bibitem{busse1994convection}
F.H. Busse.
\newblock Convection driven zonal flows and vortices in the major planets.
\newblock {\em Chaos: An Interdisciplinary Journal of Nonlinear Science},
  4(2):123--134, 1994.

\bibitem{Nield2007}
Donald~A. Nield and Adrian Bejan.
\newblock {\em Convection in porous media}.
\newblock Springer, Cham, 2017.
\newblock Fifth edition of [ MR1656781].

\bibitem{Lappa2010}
Marcello Lappa.
\newblock {\em Thermal convection: patterns, evolution and stability}.
\newblock John Wiley \& Sons, Ltd., Chichester, 2010.

\bibitem{Mizerski2021}
Krzysztof~A. Mizerski.
\newblock {\em Foundations of convection with density stratification}.
\newblock GeoPlanet: Earth and Planetary Sciences. Springer, Cham, [2021]
  \copyright 2021.

\bibitem{Radko2017}
Timour Radko.
\newblock {\em Double-diffusive convection}.
\newblock Cambridge University Press, Cambridge, 2017.
\newblock Paperback edition of the 2013 original [ MR3183649].

\bibitem{Furukawa2020}
Ken Furukawa, Yoshikazu Giga, Matthias Hieber, Amru Hussein, Takahito
  Kashiwabara, and Marc Wrona.
\newblock Rigorous justification of the hydrostatic approximation for the
  primitive equations by scaled {N}avier-{S}tokes equations.
\newblock {\em Nonlinearity}, 33(12):6502--6516, 2020.

\bibitem{Pedlosky1987}
J.~Pedlosky.
\newblock {\em Geophysical fluid dynamics}.
\newblock Springer, 1987.

\bibitem{Lions1992}
Jacques-Louis Lions, Roger Temam, and Shou~Hong Wang.
\newblock New formulations of the primitive equations of atmosphere and
  applications.
\newblock {\em Nonlinearity}, 5(2):237--288, 1992.

\bibitem{Lions1992-2}
Jacques-Louis Lions, Roger Temam, and Shou~Hong Wang.
\newblock On the equations of the large-scale ocean.
\newblock {\em Nonlinearity}, 5(5):1007--1053, 1992.

\bibitem{Majda2003}
Andrew Majda.
\newblock {\em Introduction to {PDE}s and waves for the atmosphere and ocean},
  volume~9 of {\em Courant Lecture Notes in Mathematics}.
\newblock New York University, Courant Institute of Mathematical Sciences, New
  York; American Mathematical Society, Providence, RI, 2003.

\bibitem{Chandrasekha1961}
Subrahmanyan Chandrasekhar.
\newblock {\em Hydrodynamic and hydromagnetic stability}.
\newblock International Series of Monographs on Physics. Clarendon Press,
  Oxford, 1961.

\bibitem{Drazin1981}
P.~G. Drazin and William~Hill Reid.
\newblock {\em Hydrodynamic stability}.
\newblock Cambridge Monographs on Mechanics and Applied Mathematics. Cambridge
  University Press, Cambridge-New York, 1981.

\bibitem{guo2010critical}
Yan Guo and Yongqian Han.
\newblock Critical rayleigh number in rayleigh-b{\'e}nard convection.
\newblock {\em Quarterly of Applied Mathematics}, 68(1):149--160, 2010.

\bibitem{Nguyen2025}
Tien-Tai Nguyen.
\newblock Nonlinear instability of {R}ayleigh-{B}\'{e}nard convection with
  temperature-dependent viscosity.
\newblock {\em Z. Angew. Math. Phys.}, 76(4):Paper No. 156, 21, 2025.

\bibitem{Mielke1997}
A.~Mielke.
\newblock Mathematical analysis of sideband instabilities with application to
  {R}ayleigh-{B}\'{e}nard convection.
\newblock {\em J. Nonlinear Sci.}, 7(1):57--99, 1997.

\bibitem{Ma2003}
Tian Ma and Shouhong Wang.
\newblock Attractor bifurcation theory and its applications to
  {R}ayleigh-{B}\'{e}nard convection.
\newblock {\em Commun. Pure Appl. Anal.}, 2(4):591--599, 2003.

\bibitem{Ma2004}
Tian Ma and Shouhong Wang.
\newblock Dynamic bifurcation and stability in the {R}ayleigh-{B}\'{e}nard
  convection.
\newblock {\em Commun. Math. Sci.}, 2(2):159--183, 2004.

\bibitem{Han2019}
Daozhi Han, Marco Hernandez, and Quan Wang.
\newblock Dynamic bifurcation and transition in the {R}ayleigh-{B}\'{e}nard
  convection with internal heating and varying gravity.
\newblock {\em Commun. Math. Sci.}, 17(1):175--192, 2019.

\bibitem{Braaksma2019}
Boele Braaksma and G\'{e}rard Iooss.
\newblock Existence of bifurcating quasipatterns in steady
  {B}\'{e}nard-{R}ayleigh convection.
\newblock {\em Arch. Ration. Mech. Anal.}, 231(3):1917--1981, 2019.

\bibitem{Haragus2021}
Mariana Haragus and G\'{e}rard Iooss.
\newblock Bifurcation of symmetric domain walls for the {B}\'{e}nard-{R}ayleigh
  convection problem.
\newblock {\em Arch. Ration. Mech. Anal.}, 239(2):733--781, 2021.

\bibitem{Watanabe2023}
Masahito Watanabe and Hiroaki Yoshimura.
\newblock Resonance, symmetry, and bifurcation of periodic orbits in perturbed
  {R}ayleigh-{B}\'{e}nard convection.
\newblock {\em Nonlinearity}, 36(2):955--999, 2023.

\bibitem{Welter2025}
Roland Welter.
\newblock Rotating {R}ayleigh-{B}\'{e}nard convection: attractors,
  bifurcations, and heat transport via a {G}alerkin hierarchy.
\newblock {\em SIAM J. Appl. Dyn. Syst.}, 24(2):1851--1890, 2025.

\bibitem{Roberts2025}
A.~J. Roberts.
\newblock Accurate families of multi-continuum micromorphic homogenisations in
  multi-{D} space-time via dynamical systems theory.
\newblock {\em Trans. Math. Appl.}, 9(1):Paper No. tnaf001, 72, 2025.

\bibitem{Doering2018}
Charles~R. Doering, Jiahong Wu, Kun Zhao, and Xiaoming Zheng.
\newblock Long time behavior of the two-dimensional {B}oussinesq equations
  without buoyancy diffusion.
\newblock {\em Phys. D}, 376/377:144--159, 2018.

\bibitem{Castro2019}
\'{A}ngel Castro, Diego C\'{o}rdoba, and Daniel Lear.
\newblock On the asymptotic stability of stratified solutions for the 2{D}
  {B}oussinesq equations with a velocity damping term.
\newblock {\em Math. Models Methods Appl. Sci.}, 29(7):1227--1277, 2019.

\bibitem{Tao2020}
Lizheng Tao, Jiahong Wu, Kun Zhao, and Xiaoming Zheng.
\newblock Stability near hydrostatic equilibrium to the 2{D} {B}oussinesq
  equations without thermal diffusion.
\newblock {\em Arch. Ration. Mech. Anal.}, 237(2):585--630, 2020.

\bibitem{Lai2021}
Suhua Lai, Jiahong Wu, Xiaojing Xu, Jianwen Zhang, and Yueyuan Zhong.
\newblock Optimal decay estimates for 2{D} {B}oussinesq equations with partial
  dissipation.
\newblock {\em J. Nonlinear Sci.}, 31(1):Paper No. 16, 33, 2021.

\bibitem{Adhikari2022}
Dhanapati Adhikari, Oussama Ben~Said, Uddhaba~Raj Pandey, and Jiahong Wu.
\newblock Stability and large-time behavior for the 2{D} {B}oussineq system
  with horizontal dissipation and vertical thermal diffusion.
\newblock {\em NoDEA Nonlinear Differential Equations Appl.}, 29(4):Paper No.
  42, 43, 2022.

\bibitem{Masmoudi2022}
Nader Masmoudi, Belkacem Said-Houari, and Weiren Zhao.
\newblock Stability of the {C}ouette flow for a 2{D} {B}oussinesq system
  without thermal diffusivity.
\newblock {\em Arch. Ration. Mech. Anal.}, 245(2):645--752, 2022.

\bibitem{Jang2023}
Juhi Jang and Junha Kim.
\newblock Asymptotic stability and sharp decay rates to the linearly stratified
  {B}oussinesq equations in horizontally periodic strip domain.
\newblock {\em Calc. Var. Partial Differential Equations}, 62(5):Paper No. 141,
  61, 2023.

\bibitem{Zhang2023}
Zhifei Zhang and Ruizhao Zi.
\newblock Stability threshold of {C}ouette flow for 2{D} {B}oussinesq equations
  in {S}obolev spaces.
\newblock {\em J. Math. Pures Appl. (9)}, 179:123--182, 2023.

\bibitem{Petcu2007}
M.~Petcu.
\newblock On the backward uniqueness of the primitive equations.
\newblock {\em J. Math. Pures Appl. (9)}, 87(3):275--289, 2007.

\bibitem{Cao2007}
Chongsheng Cao and Edriss~S. Titi.
\newblock Global well-posedness of the three-dimensional viscous primitive
  equations of large scale ocean and atmosphere dynamics.
\newblock {\em Ann. of Math. (2)}, 166(1):245--267, 2007.

\bibitem{Rousseau2008}
A.~Rousseau, R.~Temam, and J.~Tribbia.
\newblock The 3{D} primitive equations in the absence of viscosity: boundary
  conditions and well-posedness in the linearized case.
\newblock {\em J. Math. Pures Appl. (9)}, 89(3):297--319, 2008.

\bibitem{Guo2009}
Boling Guo and Daiwen Huang.
\newblock 3{D} stochastic primitive equations of the large-scale ocean: global
  well-posedness and attractors.
\newblock {\em Comm. Math. Phys.}, 286(2):697--723, 2009.

\bibitem{Temam2010}
R.~Temam and D.~Wirosoetisno.
\newblock Stability of the slow manifold in the primitive equations.
\newblock {\em SIAM J. Math. Anal.}, 42(1):427--458, 2010.

\bibitem{Hsia2013}
Chun-Hsiung Hsia and Ming-Cheng Shiue.
\newblock On the asymptotic stability analysis and the existence of
  time-periodic solutions of the primitive equations.
\newblock {\em Indiana Univ. Math. J.}, 62(2):403--441, 2013.

\bibitem{Cao2014}
Chongsheng Cao, Jinkai Li, and Edriss~S. Titi.
\newblock Local and global well-posedness of strong solutions to the 3{D}
  primitive equations with vertical eddy diffusivity.
\newblock {\em Arch. Ration. Mech. Anal.}, 214(1):35--76, 2014.

\bibitem{Cao2016}
Chongsheng Cao, Jinkai Li, and Edriss~S. Titi.
\newblock Global well-posedness of the three-dimensional primitive equations
  with only horizontal viscosity and diffusion.
\newblock {\em Comm. Pure Appl. Math.}, 69(8):1492--1531, 2016.

\bibitem{Li2017}
Jinkai Li and Edriss~S. Titi.
\newblock Existence and uniqueness of weak solutions to viscous primitive
  equations for a certain class of discontinuous initial data.
\newblock {\em SIAM J. Math. Anal.}, 49(1):1--28, 2017.

\bibitem{Chiodaroli2017}
Elisabetta Chiodaroli and Martin Mich\'{a}lek.
\newblock Existence and non-uniqueness of global weak solutions to inviscid
  primitive and {B}oussinesq equations.
\newblock {\em Comm. Math. Phys.}, 353(3):1201--1216, 2017.

\bibitem{De2009}
Camillo De~Lellis and L\'{a}szl\'{o} Sz\'{e}kelyhidi, Jr.
\newblock The {E}uler equations as a differential inclusion.
\newblock {\em Ann. of Math. (2)}, 170(3):1417--1436, 2009.

\bibitem{De2010}
Camillo De~Lellis and L\'{a}szl\'{o} Sz\'{e}kelyhidi, Jr.
\newblock On admissibility criteria for weak solutions of the {E}uler
  equations.
\newblock {\em Arch. Ration. Mech. Anal.}, 195(1):225--260, 2010.

\bibitem{Collot2024}
Charles Collot, Slim Ibrahim, and Quyuan Lin.
\newblock Stable singularity formation for the inviscid primitive equations.
\newblock {\em Ann. Inst. H. Poincar\'{e} C Anal. Non Lin\'{e}aire},
  41(2):317--356, 2024.

\bibitem{Binz2024}
Tim Binz, Matthias Hieber, Amru Hussein, and Martin Saal.
\newblock The primitive equations with stochastic wind driven boundary
  conditions.
\newblock {\em J. Math. Pures Appl. (9)}, 183:76--101, 2024.

\bibitem{Korn2024}
Peter Korn and E.~S. Titi.
\newblock Global well-posedness of the primitive equations of large-scale ocean
  dynamics with the {G}ent-{M}c{W}illiams-{R}edi eddy parametrization model.
\newblock {\em SIAM J. Math. Anal.}, 56(6):8011--8036, 2024.

\bibitem{Hu2025}
Ruimeng Hu, Quyuan Lin, and Rongchang Liu.
\newblock Regularization by noise for the inviscid primitive equations.
\newblock {\em J. Nonlinear Sci.}, 35(4):Paper No. 84, 23, 2025.

\bibitem{Hieber2025}
Matthias Hieber, Yoshiki Iida, Arnab Roy, and Tarek Z\"{o}chling.
\newblock The hydrostatic {L}agrangian approach to the compressible primitive
  equations.
\newblock {\em Math. Ann.}, 392(2):2277--2308, 2025.

\bibitem{Pascal2001}
Pascal Az\'{e}rad and Francisco Guill\'{e}n.
\newblock Mathematical justification of the hydrostatic approximation in the
  primitive equations of geophysical fluid dynamics.
\newblock {\em SIAM J. Math. Anal.}, 33(4):847--859, 2001.

\bibitem{Li2019}
Jinkai Li and Edriss~S. Titi.
\newblock The primitive equations as the small aspect ratio limit of the
  {N}avier-{S}tokes equations: rigorous justification of the hydrostatic
  approximation.
\newblock {\em J. Math. Pures Appl. (9)}, 124:30--58, 2019.

\bibitem{Vallis2017}
Geoffrey~K. Vallis.
\newblock {\em Atmospheric and Oceanic Fluid Dynamics: Fundamentals and
  Large-Scale Circulation}.
\newblock Cambridge University Press, 2017.
\newblock 2nd edition.

\bibitem{Bresch2004}
Didier Bresch, Alexandre Kazhikhov, and J\'{e}r\^{o}me Lemoine.
\newblock On the two-dimensional hydrostatic {N}avier-{S}tokes equations.
\newblock {\em SIAM J. Math. Anal.}, 36(3):796--814, 2004/05.

\bibitem{Temam2004}
Roger Temam and Mohammed Ziane.
\newblock Some mathematical problems in geophysical fluid dynamics.
\newblock In S.~Friedlander and D.~Serre, editors, {\em Handbook of
  Mathematical Fluid Dynamics, Vol. III}, pages 535--657. North-Holland,
  Amsterdam, 2004.

\bibitem{MR261182}
David~H. Sattinger.
\newblock The mathematical problem of hydrodynamic stability.
\newblock {\em J. Math. Mech.}, 19:797--817, 1969/70.

\bibitem{MR1003607}
Viktor~Iosifovich Yudovich.
\newblock {\em The linearization method in hydrodynamical stability theory},
  volume~74 of {\em Translations of Mathematical Monographs}.
\newblock American Mathematical Society, Providence, RI, 1989.
\newblock Translated from the Russian by J. R. Schulenberger.

\bibitem{MR610244}
Daniel Henry.
\newblock {\em Geometric theory of semilinear parabolic equations}, volume 840
  of {\em Lecture Notes in Mathematics}.
\newblock Springer-Verlag, Berlin-New York, 1981.

\bibitem{Ma2019}
T.~Ma and S.~Wang.
\newblock {\em Phase Transition Dynamics}.
\newblock Springer, 2019.

\bibitem{Han2021}
D.~Han, M.~Hernandez, and Q.~Wang.
\newblock Dynamic transitions and bifurcations for a class of axisymmetric
  geophysical fluid flow.
\newblock {\em SIAM J. Appl. Dyn. Syst.}, 20(1):38--64, 2021.

\end{thebibliography}

\end{document}